\documentclass[11pt, a4paper]{amsart}
\usepackage[top=15mm,bottom=15mm,left=15mm,right=15mm]{geometry}
\usepackage{amsfonts}
\usepackage[foot]{amsaddr}
\usepackage{xcolor}
\usepackage[utf8]{inputenc}
\usepackage{enumitem, listings}

\usepackage{graphicx,subcaption} 

\usepackage{soul}
\usepackage{amsmath, amsthm, amssymb, mathtools, amscd, graphicx, color}
\usepackage{accents}
\usepackage{bbm}
\usepackage{parskip}
\usepackage{textcomp}
\usepackage{comment}
\usepackage[colorlinks=true,urlcolor=blue,citecolor=blue,linkcolor=blue]{hyperref}

\newtheorem{theorem}{Theorem}[section]
\newtheorem{lemma}{Lemma}[section]
\newtheorem{corollary}{Corollary}[section]
\theoremstyle{definition}
\newtheorem{definition}{Definition}[section]
\newtheorem{conjecture}[theorem]{Conjecture}
\theoremstyle{remark}
\newtheorem{remark}{Remark}[section]
\numberwithin{equation}{section}
\newcommand{\ubar}[1]{\underaccent{\bar}{#1}}
\newcommand{\absmod}[1]{\left|#1\right|}

\author{Javed Hazarika}
\author{Debashis Paul}
\address{Statistical Sciences Division, Indian Statistical Institute, Kolkata}
\email{javarika@gmail.com; debpaul.isi@gmail.com}

\date{03/02/2026}
\title{Limiting Spectral Distribution of the Commutator of two data Matrices}

\begin{document}

\begin{abstract}
    We study the spectral properties of random matrices of the form $S_n^{-} = n^{-1}(X_1 X_2^* - X_2 X_1^*)$, where
$X_k = \Sigma_k^{1/2}Z_k$, 
$Z_k$'s are independent $p\times n$ complex-valued random matrices, and $\Sigma_k$ are $p\times p$ positive semi-definite matrices that commute and are independent of the $Z_k$'s for $k=1,2$. We assume that $Z_k$'s have independent entries with zero mean and unit variance. The skew-symmetric/skew-Hermitian matrix $S_n^{-}$ will be referred to as a random commutator matrix associated with the data matrices $X_1$ and $X_2$. We show that, when the dimension $p$ and sample size $n$ increase simultaneously, so that $p/n \to c \in (0,\infty)$, there exists a limiting spectral distribution (LSD) for $S_n^{-}$, supported on the imaginary axis, under the assumptions that the joint spectral 
distribution of $\Sigma_1, \Sigma_2$ converges weakly. 
This nonrandom LSD can be 
described through its Stieltjes transform, which satisfies a system of Mar\v{c}enko-Pastur-type functional equations. Moreover, we show that the 
companion matrix $S_n^{+} = n^{-1}(X_1X_2^* + X_2X_1^*)$, under identical assumptions, has an LSD supported on the real line, which can be similarly characterized.
\end{abstract}

\maketitle
\section{Introduction}

Since the seminal works on the behavior of the 
empirical distribution of eigenvalues of large-dimensional symmetric matrices and sample 
covariance matrices by Wigner 
\cite{Wigner1958}  and  Mar\v{c}enko and Pastur \cite{MarcenkoPastur1967} respectively,
there have been extensive studies on establishing limiting behavior of
various classes of random matrices. With the traditional definitions of sample size and dimension for multivariate observations, one may refer to the high-dimensional asymptotic regime where these quantities are proportional as the random matrix regime. In the random matrix regime, there have been discoveries of nonrandom limits for the empirical distribution of sample eigenvalues of various classes of symmetric or hermitian matrices. Notable classes of examples include matrices known as Fisher matrices (or ``ratios'' of independent sample covariance matrices (\cite{CLTFMatrix1}, \cite{CLTFMatrix2}), signal-plus-noise matrices (\cite{SignalNoise}) arising in signal processing, sample covariance corresponding to data with separable population covariance structure (\cite{Lixin06}, \cite{SepCovar}), with a given variance profile (\cite{varProfile}, symmetrized sample autocovariance matrices associated with stationary linear processes (\cite{Staionary1}, \cite{Staionary2}, \cite{Staionary3}), sample cross covariance matrix (\cite{CrossCov}), etc. Studies of the spectra of these classes of random matrices mentioned above are often motivated by various statistical inference problems. 

Commutators play an important role in quantum mechanics, for example in describing Heisenberg’s uncertainty principle.
Using combinatorial techniques, \cite{NicaSpeicher} derived the spectral distribution of the commutator of two free random variables. \cite{Tetilla} established the Tetilla Law, namely, the
law of the commutator of two free semicircular random variables,
which is absolutely continuous with a density having a closed form expression. \cite{Palheta} investigated the statistical properties of multiplicative commutators, i.e. matrices of the type $C = uvu^{-1} v^{-1}$, when $u$ and $v$ are
independent random matrices, uniformly distributed with respect to the Haar measure of the groups $U(N)$ and $O(N)$. \cite{Perales} analyzed the distribution of the anti-commutator of two free Poisson random variables. \cite{Vasilchuk} proved the existence of limiting spectral distributions for the commutator and the anti-commutator of two Hermitian random matrices, rotated independently with respect to one another, as the dimension grows to infinity.

Partially motivated by these, we look at a different class of ``commutator/ anti-commutator matrices", namely that of two independent rectangular data matrices under certain regularity conditions. In this paper, we study the asymptotic behavior of the spectra of random commutator matrices under the random matrix regime and discuss a potential application to
an inference problem involving covariance matrices.

As the setup for introducing these random matrices, suppose we have $p$-variate independent samples of the same size $n$ 
(expressed as $p\times n$ matrices) denoted by
$X_k = [X_{k,1}:\cdots:X_{k,n}]$, for $k = 1,2$,
from two populations with zero mean and variances $\Sigma_1$ and $\Sigma_2$ respectively. We shall study the spectral 
properties of the matrix $S_n^{-}$ defined as 
$$S_n^{-} := n^{-1}[X_1, X_2] := n^{-1}(X_1 X_2^* - X_2 X_1^*),$$
where $X_k^*$ denotes the Hermitian conjugate of $X_k$.
Given the analogy with a \textit{commutator matrix}, we shall refer to $S_n^{-}$ as a 
``sample commutator matrix'' associated with the data $(X_1,X_2)$.
A distinctive feature of $S_n^{-}$ is that it is skew-Hermitian, so that the eigenvalues of  $S_n^{-}$ are purely imaginary numbers. Analogously, we also study the properties of the Hermitian companion matrix $S_n^+$, which we shall refer to as the \textit{anti-commutator matrix}:
$$S_n^{+} :=n^{-1}\{X_1, X_2\} := n^{-1}(X_1 X_2^* + X_2 X_1^*).$$

As a primary contribution, in this paper we establish the existence of limits for the empirical spectral distribution (ESD) of $S_n^{-}$, when $p,n \to \infty$ such that $p/n \to c \in (0,\infty)$, and describe the limiting spectral distribution (LSD) through its Stieltjes transform, under additional technical assumptions on the statistical model. This LSD can be described as a unique solution of a 
pair of functional equations describing its Stieltjes transform. We also derived results related to continuity of the solution as a function of the limiting population spectrum of $\Sigma_1, \Sigma_2$. The proof techniques are largely based on the matrix decomposition based approach popularized by \cite{BaiSilv09}. Furthermore, in the special case when $\Sigma_1 = I_p = \Sigma_2$, we completely describe 
the LSD of $S_n^{-}$ as a mixture distribution on the imaginary axis with a point mass at zero (only if $c > 2$),
and a symmetric distribution with a density. Establishment of 
this result requires a very careful analysis of the Stieltjes transform
of the LSD of $S_n^{-}$, since the latter satisfies a cubic equation 
for each complex argument. The density function 
of the continuous component of the LSD can be derived in a closed (albeit
complicated) functional form that depends only on the value of $c$.

As a further contribution, we are able to derive the asymptotic
behavior of the spectrum of the companion matrix $S_n^{+}$. 
The results follow a similar pattern, which is why 
we state these results in parallel with our main results (about 
the spectral distribution of $S_n^{-}$).

The rest of the manuscript is organized as follows. Section \ref{sec:ModelPreliminary} describes the preliminaries and the model setup. Section \ref{sec:MeasuresOnImaginaryAxis} introduces new definitions to handle distributions over the imaginary axis. This is important since we will be working with skew-Hermitian matrices. As such, existing results related to metrics and convergence of measures over the real line are tweaked to handle measures over the imaginary axis.  The main result of this paper is Theorem \ref{mainTheorem} in Section \ref{sec:general_covariance} that covers the most general case with arbitrary pairs of commuting variance matrices. In Section \ref{sec:equal_covariance}, we present the special case when $\Sigma_1 = \Sigma_2$ and in Section \ref{sec:identity_covariance}, we analyze the case when $\Sigma_1 = I_p = \Sigma_2$. Finally, results regarding the anti-commutator matrix are derived in Section \ref{sec:AntiCommutator}. Whereas the results of Sections \ref{sec:general_covariance}, \ref{sec:equal_covariance} and \ref{sec:AntiCommutator} are derived under the requirement of commutativity between $\Sigma_1$ and $\Sigma_2$, Section \ref{sec:RelaxingCommutativity} relaxes this condition to some extent. Finally, Section \ref{sec:Application} introduces a hypothesis testing framework by making use of the properties of the LSD of the commutator and discusses some potential applications.

\section{Model and preliminaries}\label{sec:ModelPreliminary}
\textbf{Notations:} $\mathbbm{i}$ denotes $\sqrt{-1}$. $\mathbb{R}$ and $\mathbbm{i}\mathbb{R}$ denote the real and the imaginary axes of the complex plane, respectively. $\mathbb{C}^+$ and $\mathbb{C}^-$ denote the upper and the lower halves (excluding the real axis) of the complex plane, respectively, i.e. $\mathbb{C}^\pm = \{u \pm \mathbbm{i}v: u \in \mathbb{R}, v > 0\}$. Similarly, $\mathbb{C}_L := \{-u + \mathbbm{i}v : u > 0, v \in \mathbb{R}\}$ and $\mathbb{C}_R := \{u + \mathbbm{i}v : u > 0, v \in \mathbb{R}\}$ denote the left and right halves (excluding the imaginary axis) of the complex plane, respectively. $\Re(z)$ and $\Im(z)$ denote the real and imaginary parts respectively of the complex number $z$. The norm of a vector $x$ will be denoted as $||x||$ and the operator and Frobenius norms of a matrix $A$ will be denoted by $||A||_{op}$ and $||A||_F$, respectively.

\begin{definition}
    For a skew-Hermitian matrix $S \in \mathbb{C}^{p \times p}$ with eigenvalues $\{\mathbbm{i}\lambda_j\}_{j=1}^p$, we define the empirical spectral distribution (ESD) $F$ of $S$ as
    \begin{align}\label{ESD_of_skHerm}
F^S:\mathbbm{i}\mathbb{R} \rightarrow [0, 1]; \hspace{3mm} F^S(\mathbbm{i}x) = \frac{1}{p}\sum_{j=1}^p \mathbbm{1}_{\{\lambda_j \leq x\}}.
    \end{align}
\end{definition}
\begin{remark}
    Note that $-\mathbbm{i}S$ is Hermitian with real eigenvalues $\{\lambda_j\}_{j=1}^p$. Reconciling (\ref{ESD_of_skHerm}) with the \textit{standard} definition of ESD for Hermitian matrices (e.g. Section 2 of \cite{BaiSilv95}), we thus have 
\begin{align}\label{ESD_sksym_equal_sym}
F^S(\mathbbm{i}x) = F^{-\mathbbm{i}S}(x),  \hspace{2mm} \forall\, x \in \mathbb{R}.
\end{align}
In (\ref{ESD_sksym_equal_sym}), we have used the same notation, i.e. $F^A$ to denote the ESD of Hermitian and skew-Hermitian matrices alike. It is to be understood that the argument of the function  will be real or imaginary depending on whether the matrix in the superscript is Hermitian or skew-Hermitian, respectively.
\end{remark}

\begin{definition}\label{def:defining_JESD}
For commuting p.s.d. matrices $M_1, M_2 \in \mathbb{C}^{p \times p}$, let $P$ be \textit{a} unitary matrix such that $M_k = P D_k P^*$ where $D_k = \operatorname{diag}(\lambda_1^{(k)}, \ldots, \lambda_p^{(k)})$. For $j \in [p]$, let $\boldsymbol{\lambda}_j := \{\lambda_{j}^{(1)}, \lambda_{j}^{(2)}\}_{j=1}^p$, i.e. $\boldsymbol{\lambda}_j$ is the pair consisting of the $j^{th}$ eigenvalue (see Remark \ref{rem:JESD_is_well_defined}) from both the coordinates. Let $\textbf{M} := (M_1, M_2)$. The Joint Empirical Spectral Distribution (JESD) of $\textbf{M}$ is the probability measure on $\mathbb{R}_+^2$ that assigns equal mass to $\boldsymbol{\lambda}_j; j \in [p]$, i.e. 
\begin{align}\label{defining_JESD}
    \operatorname{JESD}(\textbf{M}) = \frac{1}{p}\sum_{j=1}^p \delta_{\boldsymbol{\lambda}_j}.
\end{align}

\end{definition}
\begin{remark}\label{rem:JESD_is_well_defined}
    Note that the choice of the unitary matrix $P$ in the spectral decomposition of both matrices is not unique. However, once we fix a $P$, the order of the $p$ eigenvalues within each of the two coordinate gets fixed. But we observe that $\operatorname{JESD}(\textbf{M})$ is independent of the choice of $P$ and is therefore well-defined.
\end{remark}

Suppose $\{Z_{1}^{(n)},Z_{2}^{(n)}\}_{n=1}^{\infty}$ are sequences of complex valued random matrices, each having dimension $p \times n$ such that $p/n \rightarrow c \in (0, \infty)$. The entries of $Z_k;k=1,2$ (denoted by $z_{ij}^{(k)}$) are independent, have zero mean, unit variance, and they satisfy some moment conditions to be stated later. These entries will be referred to as innovations. Let $\Sigma_{1n}, \Sigma_{2n} \in \mathbb{C}^{p \times p}$ be a sequence of pairs of random positive semi-definite matrices that commute (i.e. for each $n$, $\Sigma_{1n}\Sigma_{2n}=\Sigma_{2n}\Sigma_{1n}$). Henceforth, $Z_k^{(n)}$ shall be denoted by $Z_k$. We are interested in the limiting behavior (as $p, n \rightarrow \infty$) of the ESDs of matrices of the type:
\begin{align}\label{defining_Xk}
&S_n^\pm := \frac{1}{n}\bigg(X_1X_2^* \pm X_2X_1^*\bigg) \text{, where } X_k := \Sigma_{kn}^{\frac{1}{2}}Z_k.
\end{align}

We define the following central objects associated with our work.
\begin{definition}\label{defining_Sn}
    $S_n := \frac{1}{n}(X_1X_2^* - X_2X_1^*) = \frac{1}{n}\sum_{r=1}^n(X_{1r}X_{2r}^* - X_{2r}X_{1r}^*).$
\end{definition}
\begin{definition}\label{defining_Snj}
    $S_{nj} := \frac{1}{n}\sum_{r \neq j}(X_{1r}X_{2r}^* - X_{2r}X_{1r}^*)$ for $1 \leq j \leq n.$
\end{definition}
Additionally for $z \in \mathbb{C}_L$, we define the following. 
\begin{definition}\label{defining_Qz}
    $Q(z) := (S_n - zI_p)^{-1}$ is the resolvent of $S_n$.
\end{definition}
\begin{definition}\label{defining_Qzj}
    $Q_{-j}(z) := (S_{nj} - zI_p)^{-1}$ is the resolvent of $S_{nj}$ where $1 \leq j \leq n$.
\end{definition}
\begin{remark}\label{remark_norm_of_resolvent}
    For $z \in \mathbb{C}_L$, it is easy to see that any eigenvalue $\lambda$ of $(S_n-zI_p)$ satisfies $|\lambda| \geq |\Re(z)|$. Thus we have $||Q(z)||_{op} \leq 1/|\Re(z)|$. Similarly, we also have $||Q_{-j}(z)||_{op} \leq 1/|\Re(z)|$.
\end{remark}
\begin{definition}
Let $\boldsymbol{\Sigma}_n := (\Sigma_{1n}, \Sigma_{2n})$. Since these matrices commute, similar to (\ref{defining_JESD}), we represent their JESD as follows: 
\begin{align}\label{defining_Hn}
    H_n := \operatorname{JESD}(\boldsymbol{\Sigma}_n).
\end{align} 
\end{definition}

\section{Stieltjes Transforms of Measures on the imaginary axis}\label{sec:MeasuresOnImaginaryAxis}
The existing definition of Stieltjes transform and basic results deal with the weak convergence of probability measures supported on (subsets of) the real line. Since we will be dealing with skew-Hermitian matrices which have purely imaginary (or zero) eigenvalues, we modify/ develop existing definitions/ results related to convergence of measures. We will start by defining a distribution function over the imaginary axis.

Let $X$ be a purely imaginary random variable. We give the most natural definition for the distribution function $F$ of $X$. Let $\overline{F}$ be the distribution function of $-\mathbbm{i}X$, the real counterpart of $X$. Then, $F$ is defined as
\begin{align}\label{CDF_imag_equal_real}
F(\mathbbm{i}x) := \overline{F}(x)  \text{ for } x \in \mathbb{R}. 
\end{align}
It is clear that $F$ is the clockwise rotated version of $\overline{F}$. The analogous Levy metric between distribution functions $F, G$ on the imaginary axis can be defined as 
\begin{align}\label{defining_Levy}
    L_{im}(F, G) := L(\overline{F}, \overline{G}),
\end{align}
where $L(\overline{F}, \overline{G})$ is the ``standard" Levy distance between distributions $\overline{F}, \overline{G}$ over the real line. Similarly, we define the uniform distance between $F$ and $G$ as 
\begin{align}\label{defining_uniform}
    ||F - G||_{im} := ||\overline{F} - \overline{G}||,
\end{align}
where $||\overline{F} - \overline{G}||$ represents the ``standard" uniform metric between distributions over the real line. Therefore, using Lemma B.18 of \cite{BaiSilv09} leads to the following analogous inequality between Levy and uniform metrics:
\begin{align}\label{Levy_vs_uniform}
    L_{im}(F, G) = L(\overline{F}, \overline{G}) \leq ||\overline{F} - \overline{G}|| = ||F - G||_{im}.
\end{align}
This will be important specifically in establishing the weak convergence of measures over the imaginary axis.

\begin{definition}
\textbf{(Stieltjes Transform)} For a measure (not necessarily probability) $\mu$ supported on the imaginary axis, we define the Stieltjes Transform as
\begin{align}\label{defining_StieltjesTransform}
 s_{\mu}: \mathbb{C} \backslash \operatorname{supp}(\mu) \rightarrow \mathbb{C}, \hspace{3mm} \displaystyle s_{\mu}(z) = \int_{\mathbb{R}} \dfrac{\mu(dt)}{\mathbbm{i}t - z}.
\end{align}
\end{definition}

With this definition, we immediately observe the following properties. The proofs are exactly similar to those of the corresponding properties for Stieltjes Transforms of probability measures on the real line (for instance, Section 2.1.2 of \cite{Couillet}). 
\begin{description}
    \item[1] $s_\mu(.)$ is analytic on its domain and 
    \begin{align}\label{Property1}
        s_{\mu}(\mathbb{C}_L) \subset \mathbb{C}_R \text{ and } s_{\mu}(\mathbb{C}_R) \subset \mathbb{C}_L. 
    \end{align}
       \item[2] Let the total mass of $\mu$ be denoted by $M_\mu \geq 0$. Then a bound for the value of the transform at the point $z$ is given by
    \begin{align}\label{bound_Stieltjes_Transform}
        |s_{\mu}(z)| \leq M_\mu/|\Re(z)|.
    \end{align}
    \item[3] If a probability measure $\mu$ has a density at $\mathbbm{i}x$ where $x \in \mathbb{R}$, then 
    \begin{align}\label{inversion_density}
    f_{\mu}(x) = \dfrac{1}{\pi} \underset{\epsilon \downarrow 0}{\lim}\hspace{1mm}\, \Re  (s_{\mu}(-\epsilon + \mathbbm{i} x)).
    \end{align}
    \item[4] If a probability measure $\mu$ has a point mass at $\mathbbm{i}x$ where $x \in \mathbb{R}$, then 
    \begin{align}\label{inversion_pointMass}
    \mu(\{x\}) = \underset{\epsilon \downarrow 0}{\lim}\hspace{1mm} \,\epsilon s_{\mu}(-\epsilon + \mathbbm{i} x).   
    \end{align}
    \item[5] For $\mathbbm{i}a,\mathbbm{i}b$ continuity points of a probability measure $\mu$, we have
    \begin{align}\label{measureOfInterval}
        \mu([\mathbbm{i}a, \mathbbm{i}b]) = \dfrac{1}{\pi} \,\underset{\epsilon \downarrow 0}{\lim}\hspace{1mm}\int_a^b\Re(s_\mu(-\epsilon + \mathbbm{i}x)dx.
    \end{align}
\end{description}

Recall the definition of $S_n$ from (\ref{defining_Sn}). Let $S_n = P\Lambda P^*$ be a spectral decomposition of $S_n$ with $\{\mathbbm{i}\lambda_j\}_{j=1}^p$ being the $p$ purely imaginary (or zero) eigenvalues of $S_n$. In light of Definition (\ref{defining_StieltjesTransform}), we have the following expression for the Stieltjes Transform of $F^{S_n}$, the ESD of $S_n$.
\begin{definition}\label{defining_s_n}
$s_n(z) := \dfrac{1}{p}\operatorname{trace}(Q(z)) = \dfrac{1}{p}\sum_{j=1}^p\dfrac{1}{\mathbbm{i}\lambda_j - z}.$
\end{definition}

Let $P^*\Sigma_{kn}P := A^{(k)} = (a_{ij}^{(k)})$ for $k=1,2$. With this notation, we define another quantity that will play a key role in our work. 
\begin{definition}\label{defining_hn}
$\textbf{h}_{n}(z) := (h_{1n}(z), h_{2n}(z))^T$, where $h_{kn}(z) := \dfrac{1}{p}\operatorname{trace}(\Sigma_{kn}Q(z)) = \dfrac{1}{p}\sum_{j=1}^p\dfrac{a_{jj}^{(k)}}{\mathbbm{i}\lambda_j - z}.$

It is easy to see that $h_{kn}(\cdot)$ is the Stieltjes Transform of the discrete measure (say $\mu_{kn}$) that allocates a mass of $a_{jj}^{(k)}/p$ at the point $\mathbbm{i}\lambda_j$ for $1 \leq j \leq p$. At this point, we make a note of the total variation norm of the underlying measure ($\mu_{kn}$) which will be used later:
\begin{align}\label{totalVariation}
 TV(\mu_{kn}) = \frac{1}{p}\sum_{j=1}^p a_{jj}^{(k)} = \frac{1}{p}\operatorname{trace}(A^{(k)}) = \frac{1}{p}\operatorname{trace}(\Sigma_{kn}). 
\end{align}
\end{definition}

\begin{lemma}\label{StieltjesTransformRotated}
    For a probability distribution $F$ over the imaginary axis, let $s_{F}$ be the Stieltjes Transform (in the sense of Definition \ref{defining_StieltjesTransform}). For any random variable $X \sim F$, let $\overline{F}$ represent the distribution of the real-valued random variable $-\mathbbm{i}X$. Then, the Stieltjes Transform (in the standard sense) of $\overline{F}$ at $z \in \mathbb{C}^+$ is given by
    \begin{align}
        s_{\overline{F}}(z) = \mathbbm{i}s_{F}(\mathbbm{i}z).
    \end{align}
\end{lemma}

\begin{proof}
    For $x \in \mathbb{R}$, it is clear that $dF(\mathbbm{i}x) = d\overline{F}(x)$. Note that $z \in \mathbb{C}^+$ implies that $\mathbbm{i}z \in \mathbb{C}_L$. Thus, we have
    \begin{align}
        s_{\overline{F}}(z) =\int \frac{d\overline{F}(x)}{x - z}= \int \frac{dF(\mathbbm{i}x)}{-\mathbbm{i}(\mathbbm{i}x -\mathbbm{i}z)}
        = \mathbbm{i}\int \frac{dF(y)}{y - \mathbbm{i}z} = \mathbbm{i}s_{F}(\mathbbm{i}z).
    \end{align}
\end{proof}

The following is an analog of a result linking convergence of Stieltjes transforms to the weak convergence of measures on the real axis.
\begin{theorem}\label{GeroHill}
    For $n \in \mathbb{N}$, let $s_n(\cdot)$ be the Stieltjes transform of $F_n$, a probability distribution over the imaginary axis. If $s_n(z) \xrightarrow{} s(z)$ for $z \in \mathbb{C}_L$ and $\underset{y \rightarrow +\infty}{\lim}ys(-y) = 1$, then $F_n \xrightarrow{d} F$ where $s(\cdot)$ is the Stieltjes transform of $F$, a probability distribution over the imaginary axis.
\end{theorem}
\begin{proof}
The proof can be adapted with similar arguments from Theorem 1 of \cite{GeroHill03} which is stated below.

``Suppose that ($P_n$) are real Borel probability measures with Stieltjes transforms ($S_n$) respectively. If $\underset{n\rightarrow\infty}{\lim}S_n(z)=S(z)$ for all z with $\Im(z) > 0$, then there exists a Borel probability measure $P$ with Stieltjes transform $S_P=S$ if and only if 
$$\underset{y\rightarrow\infty}{\lim}\mathbbm{i}yS(\mathbbm{i}y)=-1,$$ in which case $P_n\rightarrow P$ in distribution."
\end{proof}

\begin{theorem}\label{SilvChoiResult1}
    Let $m_{G}(.)$ be the Stieltjes Transform of a probability measure $G$ on the imaginary axis. Then $G$ is differentiable at $\mathbbm{i}x_0$, if $m^*(\mathbbm{i}x_0) \equiv \underset{z \in \mathbb{C}_L \rightarrow \mathbbm{i}x_0}{\lim}\Re(m_{G}(z))$ exists and its derivative at $\mathbbm{i}x_0$ is $({1}/{\pi})m^*(\mathbbm{i}x_0)$.
\end{theorem}
\begin{proof}
The proof is similar to that of Theorem 2.1 of \cite{SilvChoi} which is stated below.

``Let G be a p.d.f. and $x_0 \in \mathbb{R}$. Suppose $\Im(m_G(x_0)) \equiv \underset{z \in \mathbb{C}^+ \rightarrow x_0}{\lim}\Im(m_G(z))$ exists. Then G is differentiable at $x_0$, and its derivative is $({1}/{\pi})\Im(m_G(x_0))$."    
\end{proof}

We mention the \textbf{Vitali-Porter Theorem} (Section 2.4, \cite{joelschiff}) below without proof.
\begin{theorem}\label{VitaliPorter}
Let $\{f_n\}_{n=1}^\infty$ be a locally uniformly bounded sequence of analytic functions in a domain $\Omega$ such that $\underset{n \rightarrow \infty}{\lim}f_n(z)$ exists for each $z$ belonging to a set $E\subset \Omega$ which has an accumulation point in $\Omega$. Then $\{f_n\}_{n=1}^\infty$ converges uniformly on compact subsets of $\Omega$ to an analytic function.    
\end{theorem}

We state the \textbf{Grommer-Hamburger Theorem} (page 104-105 of \cite{awintner}) below without proof.
\begin{theorem}\label{Grommer}
    Let $\{\mu_n\}_{n=1}^\infty$ be a sequence of measures in $\mathbb{R}$ for which the total variation is uniformly bounded.
\begin{enumerate}
    \item If $\mu_n \xrightarrow{d} \mu$, then $S(\mu_n;z) \rightarrow S(\mu;z)$ uniformly on compact subsets of $\mathbb{C}\backslash\mathbb{R}$.
    \item If $S(\mu_n;z) \rightarrow  S(z)$ uniformly on compact subsets of $\mathbb{C}\backslash\mathbb{R}$, then S(z) is the Stieltjes transform of a measure on $\mathbb{R}$ and $\mu_n \xrightarrow{d} \mu$.
\end{enumerate}
\end{theorem}

\section{LSD under arbitrary commuting pair of scaling matrices}\label{sec:general_covariance}
Before stating the main result of the paper, we first define a few functions.
\begin{definition}\label{defining_rho}
$\boldsymbol{\rho}:\mathbb{C}^2\backslash \{(z_1, z_2) \in \mathbb{C}^2:z_1z_2 \neq -1\} \rightarrow \mathbb{C}^2$ such that 
    $\boldsymbol{\rho}(z_1, z_2) = \bigg(\dfrac{z_2}{1 + z_1z_2}, \dfrac{z_1}{1 + z_1z_2}\bigg)^T.$
\end{definition}   

Letting $\boldsymbol{\rho}(z_1, z_2) = (\rho_1(z_1,z_2), \rho_2(z_1, z_2))^T$, we have the following relationships.
\begin{align}
    \Re(\rho_1(z_1, z_2)) = \frac{\Re(z_2(1+\bar{z}_1\bar{z}_2))}{|1+z_1z_2|^2} = \frac{\Re(z_2)+
    \Re(z_1)|z_2|^2}{|1+z_1z_2|^2}, \label{real_of_rho1}\\
    \Re(\rho_2(z_1, z_2)) = \frac{\Re(z_1(1+\bar{z}_1\bar{z}_2))}{|1+z_1z_2|^2} = \frac{\Re(z_1)+
    \Re(z_2)|z_1|^2}{|1+z_1z_2|^2}.\label{real_of_rho2}    
\end{align}

\begin{remark}\label{remark_real_of_rho}
It is clear that for $\textbf{z} = (z_1, z_2) \in \mathbb{C}_R^2$, we have $\boldsymbol{\rho}(\textbf{z}) := (\rho_1(\textbf{z}), \rho_2(\textbf{z}))^T \in \mathbb{C}_R^2$ or $\boldsymbol{\rho}(\mathbb{C}_R^2) \subset \mathbb{C}_R^2$.
\end{remark}

\begin{theorem}\label{mainTheorem}
    \textbf{Main Theorem:} Suppose the following conditions hold.
    \begin{enumerate}
        \item[$\boldsymbol{T}_1$:] $c_n := \dfrac{p}{n} \rightarrow c \in (0, \infty)$.
        \item[$\boldsymbol{T}_2$:] Entries of $Z_1, Z_2$ are independent with zero mean and unit variance and for some $\eta_0 > 0$, they satisfy
        \begin{align}\label{T_2}
            \underset{k=1,2}{\max} \,\sup_{i,j}\mathbb{E}|z_{ij}^{(k)}|^{2+\eta_0} < \infty.
        \end{align}
        \item[$\boldsymbol{T}_3:$] $\Sigma_{1n}\Sigma_{2n} = \Sigma_{2n}\Sigma_{1n}$ for all $n \in \mathbb{N}$.
        \item[$\boldsymbol{T}_4$:] $H_n \xrightarrow{d} H$ a.s. where $H$ is a non-random bi-variate probability distribution on $\mathbb{R}_+^2$ with support not contained entirely in the real or the imaginary axis.
        \item[$\boldsymbol{T}_5$:] There exists a constant $D_0>0$ such that 
        \begin{align}\label{secondSpectralBound}
         \underset{k=1,2}{\max}\,\underset{n \rightarrow \infty}{\limsup} \,\bigg\{\frac{1}{p}\operatorname{trace}(\Sigma_{kn}^2)\bigg\} < D_0.
        \end{align}
    \end{enumerate}
    Then, $F^{S_n} \xrightarrow{d} F$ a.s. where the Stieltjes Transform of $F$ at $z \in \mathbb{C}_L$ is characterized as follows:
    \begin{align}\label{s_main_eqn}
        s_{F}(z) &= \displaystyle \int_{\mathbb{R}_+^2} \frac{dH(\boldsymbol{\lambda})}{-z + \boldsymbol{\lambda}^T\boldsymbol{\rho}(c\textbf{h}(z))},
    \end{align}
    where $\textbf{h}(z) = (h_1(z), h_2(z)) \in \mathbb{C}_R^2$ are unique numbers such that
    \begin{align}\label{h_main_eqn}
         \textbf{h}(z) &= \displaystyle \int_{\mathbb{R}_+^2} \frac{\boldsymbol{\lambda} dH(\boldsymbol{\lambda})}{-z + \boldsymbol{\lambda}^T\boldsymbol{\rho}(c\textbf{h}(z))} \text{, and } \boldsymbol{\lambda} = (\lambda_1, \lambda_2)^T.
    \end{align}
    Moreover, $h_1, h_2$ themselves are Stieltjes Transforms of measures (not necessarily probability measures) over the imaginary axis and continuous as functions of $H$.
    An equivalent characterization of the Stieltjes Transform is given by 
    \begin{align}\label{s_main_eqn_1}
        s_F(z) &= \frac{1}{z}\bigg(\frac{2}{c}-1\bigg) - \frac{2}{cz}\bigg(\frac{1}{1+c^2h_1(z)h_2(z)}\bigg)\text{, for all } z \in \mathbb{C}_L.
    \end{align}
\end{theorem}
NOTE: Throughout the paper we will be using bold symbols such as \textbf{h}, $\boldsymbol{\lambda}$, $\boldsymbol{\rho},\textbf{0}$ to denote vector quantities.

\begin{remark}
    The alternate characterization (\ref{s_main_eqn_1}) of the Stieltjes Transform is useful when investigating the presence of point mass at $0$ in the LSD. The detailed proofs of Theorem \ref{mainTheorem} presented later on will thus focus only on proving (\ref{s_main_eqn}) and (\ref{h_main_eqn}).
\end{remark}

\begin{remark}
    For a fixed $z \in \mathbb{C}_L$, Theorem \ref{ContinuityGeneral} states a result regarding the continuity of the solution to (\ref{h_main_eqn}), i.e. $\textbf{h}(z)$ w.r.t. $H$ under a certain technical condition (\ref{TechnicalCondition_continuity}) and Condition $\boldsymbol{T}_5$. In the special case covered in Section \ref{sec:equal_covariance}, we prove a stronger result (Theorem \ref{Continuity_Special}) without requiring this technical condition or any assumption on spectral moment bounds.

Due to the conditions (i.e., $\boldsymbol{T}_5$) imposed on $\Sigma_{1n}, \Sigma_{2n}$ in Theorem \ref{mainTheorem}, there exists $0 < C_0 < \infty$ such that
\begin{align}\label{firstSpectralBound}
    \underset{k=1,2}{\max}\,\bigg\{\underset{n \rightarrow \infty}{\limsup}\frac{1}{p}\operatorname{trace}(\Sigma_{kn})\bigg\} < C_0,
\end{align}
and by Skorohod Representation Theorem and Fatou's Lemma,
\begin{align} \label{defining_C0}
    \underset{k=1,2}{\max}\,\int \lambda_k dH(\boldsymbol{\lambda}) \leq C_0.
\end{align}
\end{remark}

\begin{remark}
    The proof for existence of some solution of (\ref{h_main_eqn}) can be done using a uniform bound only on the first spectral moments of $\Sigma_{1n}, \Sigma_{2n}$ (i.e. \ref{firstSpectralBound}). The proofs of uniqueness and continuity of the solution requires the second moments to be bounded, i.e. Condition $\boldsymbol{T}_5$.
\end{remark}

\begin{remark}
    The assumptions on $\Sigma_{kn};k=1,2$ hold in an almost sure sense. Moreover, $H_n$ (defined in (\ref{defining_Hn})) converges weakly to a non-random $H$ almost surely. By the end of this Section, we show that conditioning on $\Sigma_{kn}$, $F^{S_n}$ converges weakly to a non-random limit $F$ that depends on $\Sigma_{kn}$ only through their non-random limit $H$. This result holds irrespective of whether $\Sigma_{kn}$ is random or not. Therefore, we will henceforth treat $\{\Sigma_{kn}\}_{n=1}^\infty$ as a non-random sequence.
\end{remark}

\subsubsection{Proof of the equivalent characterization, i.e. (\ref{s_main_eqn_1}) in Theorem \ref{mainTheorem}}
\begin{proof}
    The second expression for the Stieltjes Transform of the LSD $F$ in Theorem \ref{mainTheorem}, i.e. (\ref{s_main_eqn_1}) follows easily from (\ref{s_main_eqn}) and (\ref{h_main_eqn}). For $k=1,2$, we have 
\begin{align}
    &h_k(z) = \int \frac{\lambda_k dH(\boldsymbol{\lambda})}{-z + \boldsymbol{\lambda}^T\boldsymbol{\rho}(c\textbf{h})}\\ \notag
    \implies& h_k(z)\rho_k(c\textbf{h}(z)) = \int \frac{\lambda_k\rho_k(c\textbf{h}(z))dH(\boldsymbol{\lambda})}{-z + \boldsymbol{\lambda}^T\boldsymbol{\rho}(c\textbf{h}(z))} \\ \notag
    \implies& h_1(z)\rho_1(c\textbf{h}(z))+h_2(z)\rho_2(c\textbf{h}(z)) = \int \frac{z - z + \lambda_1\rho_1(c\textbf{h}(z)) + \lambda_2\rho_2(c\textbf{h}(z))}{- z + \lambda_1\rho_1(c\textbf{h}(z)) + \lambda_2\rho_2(c\textbf{h}(z))}dH(\boldsymbol{\lambda})\\ \notag
    \implies& \frac{2ch_1(z)h_2(z)}{1+c^2h_1(z)h_2(z)} = zs_{F}(z) + 1\\ \notag
    \implies& s_{F}(z) = \frac{1}{z}\bigg(\frac{2}{c}-1\bigg) -\frac{2}{cz}\bigg(\frac{1}{1+c^2h_1(z)h_2(z)}\bigg).
\end{align}
\end{proof}

\begin{lemma}\label{tightness}
    Under the conditions of Theorem \ref{mainTheorem}, if we instead had $H = \delta_{(0,0)}$, we have $F^{S_n} \xrightarrow{d} \delta_{(0,0)}$ a.s. For any probability distribution $H$ supported on $\mathbb{R}_+^2$, $\{F^{S_n}\}_{n=1}^\infty$ is a tight sequence.
\end{lemma}
The proof is given in Section \ref{sec:proofOfTightness}. Below, we present an overview of the steps to prove Theorem \ref{mainTheorem}.

\subsection{Sketch of the proof}\label{ProofSketch}
First, we show that for all $z \in \mathbb{C}_L$, equation (\ref{h_main_eqn}) can have at most one solution within the class of analytic functions mapping $\mathbb{C}_L$ to $\mathbb{C}_R^2$. This is established in Theorem \ref{Uniqueness}. After this, we impose a set of assumptions on $\Sigma_k,z_{ij}^{(k)};k=1,2$ similar to \cite{Lixin06}, \cite{BaiSilv09}. This will act as a stepping stone to prove the result under general conditions of Theorem \ref{mainTheorem}. The assumptions are as follows.
\subsubsection{Assumptions}\label{A12}
\begin{itemize}
    \item \textbf{A1} There exists a constant $\tau > 0$ such that $\underset{k=1,2}{\max}\bigg(\underset{n \in \mathbb{N}}{\operatorname{sup}} ||\Sigma_{kn}||_{op}\bigg) \leq \tau$.
    \item \textbf{A2} $\mathbb{E}z_{ij}^{(k)} = 0$, 
    $\mathbb{E}|z_{ij}^{(k)}|^2 =1$, $|z_{ij}^{(k)}| = O(n^b)$, where $b \in (\frac{1}{2+\eta_0}, \frac{1}{2})$ and $\eta_0>0$ is the same as in $\boldsymbol{T}_2$.
\end{itemize}
\color{black}
Under these assumptions, the proof of Theorem \ref{mainTheorem} is done in the following steps.
\begin{itemize}
    \item[1] For $k=1,2$, the sequences $\{h_{kn}(z)\}_{n=1}^\infty$ have at least one subsequential limit by Theorem \ref{CompactConvergence}. Every subsequential limit of $\{h_{kn}(z)\}_{n=1}^\infty$ satisfies (\ref{h_main_eqn}) and moreover, they are Stieltjes Transforms of measures over the imaginary axis. This is done in Theorem \ref{Existence_A12}. Thus, (\ref{h_main_eqn}) has a unique solution.
    \item[2] Next, we establish a deterministic equivalent for $Q(z)$ in terms of $\Sigma_k,\mathbb{E}h_{kn}(z)$. This is done in Theorem \ref{DeterministicEquivalent}.
    \item[3] Finally, we show that $s_n(z) \xrightarrow{a.s.} s_{F}(z)$ and $s_{F}$ satisfies the condition in Proposition \ref{GeroHill}. Therefore, $F^{S_n} \xrightarrow{d} F$ a.s., where $F$ is the LSD of interest. This is done in Theorem \ref{Existence_A12}.
\end{itemize}

Next, we show items (1--3) under the general conditions of Theorem \ref{mainTheorem}. The idea is to construct sequences of matrices whose ESDs are \textit{close} (in uniform metric) to that of $S_n$ but satisfy Assumptions \ref{A12}. This allows us to use the above results. The outline of this part of the proof is provided in Theorem \ref{Existence_General}.

\begin{definition}
      For $0 < b$, we define the bounded sector of $\mathbb{C}_R$ denoted by $\mathcal{S}(b)$ as follows.
      \begin{align}\label{defining_sector}
          \mathcal{S}(b):= \{z \in \mathbb{C}_R: |\Im(z)| \leq \Re(z), |z| \leq b\}.
      \end{align}
\end{definition}

\begin{lemma}\label{location_of_solution_General}
(\textbf{Location of Solution:}) Let $z = -u+\mathbbm{i}v \in \mathbb{C}_L$ and $\textbf{h}=(h_1, h_2) \in \mathbb{C}_R^2$ be a solution to (\ref{h_main_eqn}). Then for $k=1,2$,
    \begin{enumerate}
        \item[(1)] Under the conditions of Theorem \ref{mainTheorem}, we have $|h_k(z)| \leq C_0/u$ where $C_0$ is defined in (\ref{defining_C0}).
        \item[(2)] If $|v| \leq u$ and $u$ is sufficiently large, then $|\Im(h_k(z))| \leq \Re(h_k(z))$.
    \end{enumerate}
\end{lemma}
The proof can be found in Section \ref{sec:ProofOflocation_of_solution_General}. 
\begin{theorem}\label{Uniqueness}
    \textbf{(Uniqueness)} For a bi-variate distribution $H$ supported on $\mathbb{R}_+^2$ and $c \in (0, \infty)$, there can be at most one solution to (\ref{h_main_eqn}) within the class of analytic functions that map $\mathbb{C}_L$ to $\mathbb{C}_R^2$.
\end{theorem}
The proof can be found in Section \ref{sec:ProofOfUniqueness}.

\subsection{Existence of Solution under Assumptions \ref{A12}}
\begin{theorem}\label{CompactConvergence}
\textbf{Compact Convergence:} For $k =1,2,$ $\mathcal{H}_k = \{h_{kn}\}_{n \in \mathbb{N}}$ are normal families\footnote{A class of functions where every sequence has a further subsequence that converges uniformly on compact subsets}.
\end{theorem}
\begin{proof}
By \textit{Montel's theorem} (Theorem 3.3 of \cite{SteinShaka}), it is sufficient to show that $s_n$, $h_{1n}$ and $h_{2n}$ are uniformly bounded on every compact subset of $\mathbb{C}_L$. Let $K \subset \mathbb{C}_L$ be an arbitrary compact subset. Define $u_0 := \inf\{|\Re(z)|: z \in K\}$. It is clear that $u_0 > 0$. Then for arbitrary $z \in K$, using (\ref{bound_Stieltjes_Transform}) we have 
$$|s_n(z)| = \frac{1}{p}|\operatorname{trace}(Q(z))| \leq \frac{1}{|\Re(z)|} \leq \frac{1}{u_0}.$$
Using (\ref{R5}), (\ref{firstSpectralBound}) and Remark \ref{remark_norm_of_resolvent}, for sufficiently large $n$, we have
    \begin{align}\label{h_nBound}
        |h_{kn}(z)| &= \frac{1}{p}|\operatorname{trace}(\Sigma_{kn}Q)| \leq \bigg(\frac{1}{p}\operatorname{trace}(\Sigma_{kn})\bigg) ||Q(z)||_{op} 
        \leq \frac{C_0}{|\Re(z)|} \leq \frac{C_0}{u_0}.
    \end{align}
\end{proof}
\begin{remark}
    Note that the proof relied simply on $\boldsymbol{T}_5$ of Theorem \ref{mainTheorem} and not on Assumptions \ref{A12}.
\end{remark}

\begin{theorem}\label{DeterministicEquivalent}
\textbf{Deterministic Equivalent:} Under Assumptions \ref{A12}, for $z \in \mathbb{C}_L$, a deterministic equivalent for $Q(z)$ is given by 
\begin{align}
\Bar{Q}(z) = \bigg(-zI_p + \rho_1(c_n\mathbb{E}{\textbf{h}}_{n}(z))\Sigma_{1n} + \rho_2(c_n\mathbb{E}{\textbf{h}}_{n}(z))\Sigma_{2n}\bigg)^{-1}.
\end{align}
\end{theorem}

\begin{remark}
    By (\ref{Real_of_eigen_QBarz}), for large $n$, all the eigenvalues of $\bar{Q}(z)$ are non-zero for any $z \in \mathbb{C}_L$. In particular, this implies that $\bar{Q}(z)$ is invertible for sufficiently large $n$ depending on $z$. The proof is given in Section \ref{sec:ProofDeterministicEquivalent}.
\end{remark}

At this point, we define a few additional deterministic quantities that will serve as approximations to the random quantity $h_{kn}(z)$ for $z \in \mathbb{C}_L$ and $k=1,2$.
\begin{definition}\label{defining_hn_tilde}
    $\tilde{\textbf{h}}_n(z) := (\tilde{h}_{1n}(z), \tilde{h}_{2n}(z))$, where  $\Tilde{h}_{kn}(z) = \frac{1}{p}\operatorname{trace}\{\Sigma_{kn}\Bar{Q}(z)\}$, $k=1,2$, $z \in \mathbb{C}_L$.
\end{definition}
Under Assumptions \ref{A12}, a direct consequence of Theorem \ref{DeterministicEquivalent} is as follows. For $k=1,2$, $z \in \mathbb{C}_L$, we have
\begin{align}\label{Consequence_DetEqv}
    \absmod{\frac{1}{p}\operatorname{trace}(\Sigma_{kn}(Q(z)-\overline{Q}(z)))} \xrightarrow{a.s.} 0 \implies \absmod{h_{kn}(z) - \tilde{h}_{kn}(z)} \xrightarrow{a.s.} 0.
\end{align}

\begin{definition}\label{defining_QBarBar}
    $\Bar{\Bar{Q}}(z) = \bigg(-zI_p + \rho_1(c_n\Tilde{\textbf{h}}_{n}(z))\Sigma_{1n} + \rho_2(c_n\Tilde{\textbf{h}}_{n}(z))\Sigma_{2n}\bigg)^{-1}$.
\end{definition}
\begin{definition}\label{defining_hn_tilde2}
   $\tilde{\tilde{\textbf{h}}}_n(z) := (\tilde{\tilde{h}}_{1n}(z), \tilde{\tilde{h}}_{2n}(z))$ where $\Tilde{\Tilde{h}}_{kn}(z) = \frac{1}{p}\operatorname{trace}\{\Sigma_{kn}\Bar{\Bar{Q}}(z)\}$ , $k =1,2$.
\end{definition}

Using $\boldsymbol{T}_3$ of Theorem \ref{mainTheorem} and as per the notation $\boldsymbol{\lambda} = (\lambda_1, \lambda_2)$, we can simplify $\tilde{h}_{kn}(z)$ and $\tilde{\tilde{h}}_{kn}(z)$ as follows:
\begin{align}\label{simplifying_hn_tilde}
    {\tilde{h}}_{kn}(z) &= \int \frac{\lambda_k dH(\boldsymbol{\lambda})}{-z + \lambda_1\rho_1(c_n\mathbb{E}{\textbf{h}}_n(z)) + \lambda_2\rho_2(c_n\mathbb{E}{\textbf{h}}_n(z))}, \text{ and}\\ \notag
    \tilde{\tilde{h}}_{kn}(z) &= \int \frac{\lambda_k dH(\boldsymbol{\lambda})}{-z + \lambda_1\rho_1(c_n\tilde{\textbf{h}}_n(z)) + \lambda_2\rho_2(c_n\tilde{\textbf{h}}_n(z))}.
\end{align}

Note that, by Theorem \ref{DeterministicEquivalent} and Lemma \ref{hn_tilde_tilde2},  $\tilde{h}_{kn}(z), \tilde{\tilde{h}}_{kn}(z)$ are deterministic approximations to $h_{kn}(z)$. This serves as a critical step in the proof for the existence of the unique solution to (\ref{h_main_eqn}).

\begin{theorem}\label{Existence_A12}
    \textbf{Existence of solution:}
    Under Assumptions \ref{A12}, for $z \in \mathbb{C}_L$, we have 
    \begin{description}
        \item[1] For $k =1,2$, $h_{kn}(z) \xrightarrow{a.s.} h_k^\infty(z)$ where $h_k^\infty$ are Stieltjes transforms of measures over the imaginary axis,
        \item[2] $h_1^\infty(z), h_2^\infty(z)$ uniquely satisfy (\ref{h_main_eqn}),
        \item[3] $s_n(z) \xrightarrow{a.s.} s_{F}(z)$ where $s_{F}(\cdot)$ is as defined in (\ref{s_main_eqn}), and
        \item[4] $s_{F}(\cdot)$ satisfies $\underset{y \rightarrow +\infty}{\lim}ys_{F}(-y) = 1$.
    \end{description} 
\end{theorem}
The proof is given in Section \ref{sec:ProofExistence_A12}.
\begin{remark}\label{A2_dashed}
    Suppose $z_{ij}^{(k)}$ satisfy all conditions in \textbf{A2} except $\mathbb{E}|z_{ij}^{(k)}|^2 = 1$. However, they satisfy the following condition:
    \begin{align}\label{U1}
        \underset{i,j,k,n}{\sup}|\mathbb{E}|z_{ij}^{(k)}|^2 -1|\longrightarrow 0.
    \end{align}
    In other words, $\mathbb{E}|z_{ij}^{(k)}|^2$ converge uniformly to $1$. We remark that the conclusions of Theorem \ref{Existence_A12} will continue to hold even in this case. This is because the variance of the innovations are invoked at a stage (refer to Lemma \ref{quadraticForm}) which establishes uniform concentration of an array of quadratic forms around their respective means. In Corollary \ref{quadraticFormCorollary}, we observe that said result holds even when the variance terms converge uniformly to $1$ instead of exact equality.
\end{remark}

\color{black}

\subsection{Existence of Solution under General Conditions}
Theorem \ref{Existence_A12} proved the statement of Theorem \ref{mainTheorem} under Assumptions \ref{A12}. We now repeat this under the general conditions $\boldsymbol{T}_1 - \boldsymbol{T}_5$ of Theorem \ref{mainTheorem}.
\begin{theorem}\label{Existence_General}
    Under the conditions of Theorem \ref{mainTheorem}, for $z \in \mathbb{C}_L$, we have 
    \begin{description}
       \item[1] For $k =1,2$, $h_{k}^\tau(z) \xrightarrow{} h_k^\infty(z)$ where $h_k^\infty$ are Stieltjes transforms of measures over the imaginary axis,
        \item[2] $h_1^\infty(z), h_2^\infty(z)$ uniquely satisfy (\ref{h_main_eqn}),
        \item[3] $s_n(z) \xrightarrow{a.s.} s_{F}(z)$ where $s_{F}(\cdot)$ is as defined in (\ref{s_main_eqn}), and
        \item[4] $s_{F}(\cdot)$ satisfies $\underset{y \rightarrow +\infty}{\lim}ys_{F}(-y) = 1$.
    \end{description}
\end{theorem}

\begin{proof}    
We construct a sequence of matrices \textit{similar} to $\{S_n\}_{n=1}^\infty$ but satisfying \textbf{A1-A2} of Assumptions \ref{A12}. The steps below give an outline of the proof, with the essential details shifted to individual modules wherever necessary. 
\begin{description}
    \item[Step1] Let $H$ be a bi-variate distribution supported on $\mathbb{R}_+^2$. Consider the random vector $\textbf{e} = (e_1, e_2) \sim H$. For $\tau > 0$, define $H^\tau$ as the joint distribution of $\textbf{e}^\tau := (e_1^\tau, e_2^\tau)$ where $e_k^\tau := e_k \mathbbm{1}_{\{e_k \leq \tau\}};k=1,2$.
    \item[Step2] For a p.s.d. matrix A and a fixed $\tau > 0$, let $A^\tau$ represent the matrix obtained by replacing all eigenvalues of $A$ greater than $\tau$ with 0 in its spectral decomposition. Recall the definition of $H_n$ from  (\ref{defining_Hn}). It is clear that for any fixed $\tau > 0$, as $n \rightarrow \infty$, we have
    $$H_n^\tau := F^{\Sigma_{1n}^\tau, \Sigma_{2n}^\tau} \xrightarrow{d} H^\tau.$$ 
However, we will choose $\tau > 0$ such that $(\tau, \tau)$ is a continuity point of $H$. This will be essential in Section \ref{Step10}.
    \item[Step3] For $k=1,2$, let $\Lambda_{kn} := \Sigma_{kn}^\frac{1}{2}$ and $\Lambda_{kn}^\tau := (\Sigma_{kn}^\tau)^\frac{1}{2}$. Then $S_n = \frac{1}{n}(\Lambda_{1n}Z_1Z_2^*\Lambda_{2n} - \Lambda_{2n}Z_2Z_1^*\Lambda_{1n})$.
    \item[Step4] Define 
    \begin{align}\label{defining_Tn}
        T_n := \frac{1}{n}(\Lambda_{1n}^\tau Z_1Z_2^*\Lambda_{2n}^\tau - \Lambda_{2n}^\tau Z_2Z_1^*\Lambda_{1n}^\tau).
    \end{align}
    \item[Step5] Recall that, we have $Z_k = (z_{ij}^{(k)}) \in \mathbb{C}^{p \times n}$. Define $\hat{Z}_k := (\hat{z}_{ij}^{(k)})$ with $\hat{z}_{ij}^{(k)} = z_{ij}^{(k)}\mathbbm{1}_{\{|z_{ij}^{(k)}| \leq n^b\}}$ where $b$ follows from \textbf{A2}. Now, let 
    \begin{align}\label{defining_Un}
        U_n := \frac{1}{n}(\Lambda_{1n}^\tau\hat{Z_1}\hat{Z_2}^*\Lambda_{2n}^\tau - \Lambda_{2n}^\tau\hat{Z_2}\hat{Z_1}^*\Lambda_{1n}^\tau).
    \end{align}
    \item[Step6] Let $\Tilde{Z}_k = \hat{Z}_k - \mathbb{E}\hat{Z}_k$. Then, define
    \begin{align}\label{defining_Un_tilde}
        \Tilde{U}_n:=\frac{1}{n}(\Lambda_{1n}^\tau\Tilde{Z_1}\Tilde{Z_2}^*\Lambda_{2n}^\tau - \Lambda_{2n}^\tau\Tilde{Z_2}\Tilde{Z_1}^*\Lambda_{1n}^\tau).
    \end{align}
    Let $s_n(\cdot), t_n(\cdot), u_n(\cdot), \Tilde{u}_n(\cdot)$ be the Stieltjes transforms of $F^{S_n}, F^{T_n}$, $F^{U_n}, F^{\Tilde{U}_n}$ respectively.
    \item[Step7] 
    By construction, $\Sigma_{kn}^\tau$ satisfies \textbf{A1}. Note that $\{|z_{ij}^{(k)}|^2\}_{i,j,n}$ is a uniformly integrable class due to $\boldsymbol{T}_2$ of Theorem \ref{mainTheorem}. As a result,
    \begin{align*}
        &\underset{i,j,k,n}{\sup}|\mathbb{E}|\hat{z}_{ij}^{(k)}|^2 - 1| = \underset{i,j,k,n}{\sup}\big|\mathbb{E}|z_{ij}^{(k)}|^2\mathbbm{1}_{\{|z_{ij}^{(k)}|\leq n^b\}} - \mathbb{E}|z_{ij}^{(k)}|^2\big| = \underset{i,j,k,n}{\sup} \mathbb{E}|z_{ij}^{(k)}|^2\mathbbm{1}_{\{|z_{ij}^{(k)}|> n^b\}} \longrightarrow 0.
    \end{align*}
    \text Thus, $Var(\tilde{z}_{ij}^{(k)}) = \mathbb{E}|\tilde{z}_{ij}^{(k)}|^2 = \mathbb{E}|\hat{z}_{ij}^{(k)}|^2 \rightarrow 1$ uniformly.  
    In view of Remark \ref{A2_dashed}, Theorem \ref{Existence_A12} implies that $F^{\Tilde{U}_n} \xrightarrow{a.s.} F^\tau$, where the limiting distribution is characterized by a pair $(\textbf{h}^\tau, s^\tau)$ satisfying (\ref{s_main_eqn}) and (\ref{h_main_eqn}) with $H^\tau$ instead of $H$. In particular, $|\Tilde{u}_n(z) - s^\tau(z)| \xrightarrow{a.s.} 0$ by the same theorem.
    \color{black}
    \item[Step8] 
    Next we show that $\textbf{h}^\tau$ converges to some limit as $\tau \rightarrow \infty$ through continuity points of $\tau$. Using Montel's Theorem, we are able to show that any arbitrary subsequence of $\{\textbf{h}^\tau\}$ has a further subsequence $\{\textbf{h}^{\tau_m}\}_{m=1}^\infty$ that converges uniformly on compact subsets (of $\mathbb{C}_L$) as $m \rightarrow \infty$. Each subsequential limit will be shown to belong to $\mathbb{C}_R$ and satisfy (\ref{h_main_eqn}). Moreover, by Theorem \ref{Uniqueness}, all these subsequential limits must be the same, which we denote by $\textbf{h}^\infty$. Therefore, $\textbf{h}^\tau \xrightarrow{} \textbf{h}^\infty$.
    \item[Step9] Next we show that $s^\tau(z) \xrightarrow{a.s.} s_F(z)$ with $s_F(\cdot)$ defined in (\ref{s_main_eqn}) and that $s_F(\cdot)$ satisfies the necessary and sufficient condition in Proposition \ref{GeroHill} for a Stieltjes transform of a measure over the imaginary axis. So, there exists some distribution $F$ corresponding to $s_F$ which is our LSD of interest. Therefore, it suffices to show that $s_n(z) \xrightarrow{a.s.} s_F(z)$. This is done in \textbf{Step10}. \textbf{Step8} and \textbf{Step9} are shown explicitly in Section \ref{Step_8_9}.
    \item[Step10]  Finally, we show that $|s_n(z) - s_F(z)| \rightarrow 0$. We have,
    \begin{align*}
        |s_n(z) - s_F(z)| &\leq |s_n(z) - t_n(z)| + |t_n(z)-u_n(z)| + |u_n(z)-\Tilde{u}_n(z)|\\ &+ |\tilde{u}_n(z)-s^\tau(z)| + |s^\tau(z) - s_F(z)|.
    \end{align*}
We will show that each term on the RHS goes to 0 as $n \rightarrow \infty$ and $
\tau \rightarrow \infty$ through continuity points of $H$. From Section \ref{Step10} and (\ref{Levy_vs_uniform}), we get the following inequalities: \begin{itemize}
    \item $L_{im}(F^{S_n}, F^{T_n}) \leq ||F^{S_n} - F^{T_n}||_{im} \xrightarrow{a.s.} 0,$
    \item $L_{im}(F^{T_n}, F^{U_n}) \leq ||F^{T_n} - F^{U_n}||_{im} \xrightarrow{a.s.} 0,$ 
    \item $L_{im}(F^{U_n}, F^{\Tilde{U}_n}) \leq ||F^{U_n} - F^{\Tilde{U}_n}||_{im} \xrightarrow{a.s.} 0.$
    
\end{itemize}
Application of Lemma \ref{Levy_Stieltjes} to the three items above implies $|s_n(z) - t_n(z)| \xrightarrow{a.s.} 0$, $|t_n(z) - u_n(z)| \xrightarrow{a.s.} 0$ and, $|u_n(z) - \Tilde{u}_n(z)| \xrightarrow{a.s.} 0$ respectively. From \textbf{Step7}, we already have $|\Tilde{u}_n(z) - s^\tau(z)| \xrightarrow{a.s.} 0$. From Section \ref{Step_8_9}, we have $|s^\tau(z) - s_F(z)| \rightarrow 0$.
\item[Step11] Hence, $s_n(z) \xrightarrow{a.s.} s_F(z)$ which is a Stieltjes transform. By Proposition \ref{GeroHill}, $F^{S_n} \xrightarrow{d} F$ a.s., where $F$ is characterized by $(\textbf{h}^\infty, s_F)$ which satisfy (\ref{s_main_eqn}) and (\ref{h_main_eqn}). This concludes the proof.
\end{description}
\end{proof}

\subsection{Properties of the LSD}
\begin{theorem}\label{symmetry}
    The LSD $F$ in Theorem \ref{mainTheorem} is symmetric about $0$.
\end{theorem}
\begin{proof}
    Note that
    \begin{align}
        \rho_1(\overline{z_1}, \overline{z_2}) = \frac{\overline{z_2}}{1+\overline{z_1}\overline{z_2}} = \overline{\bigg(\frac{z_2}{1+z_1z_2}\bigg)} = \overline{\rho_1(z_1, z_2)}.
    \end{align}

Similarly $\rho_2(\overline{z_1}, \overline{z_2}) = \overline{\rho_2(z_1, z_2)}$. Thus, we find that $h_k(\overline{z}) = \overline{h_k(z)}$ and $s_F(\overline{z}) = \overline{s_F(z)}$. The symmetry of the LSD is immediate upon using (\ref{measureOfInterval}).
\end{proof}
\begin{remark}
For real skew symmetric matrices, the ESDs ($F^{S_n}$) are exactly symmetric about $0$.
\end{remark}

\begin{theorem}\label{pointMass}
In Theorem \ref{mainTheorem}, let $H = (1-\beta) \delta_\textbf{0} + \beta H_1$ where $H_1$ is a probability distribution over $\mathbb{R}_+^2$ which has no point mass at $\textbf{0}=(0,0)$ and $0 < \beta \leq 1$. Then,
\begin{enumerate}
    \item[(1)] When $0 < c < 2/\beta$, the LSD $F$ has a point mass at $0$ equal to $1 - \beta$,
    \item[(2)] When $2/\beta \leq c$, the LSD $F$ has a point mass at $0$ equal to $1 - 2/c$.
\end{enumerate}    
\end{theorem}
The proof is given in Section \ref{sec:ProofOfPointMass}.

\begin{figure}[ht]
    \centering
    \includegraphics[width=0.75\linewidth]{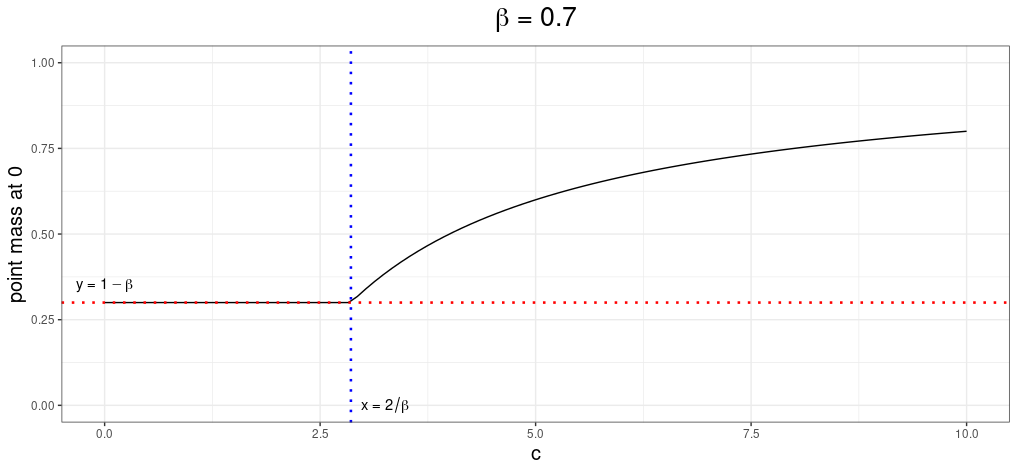}
    \caption{Illustration of the result of Theorem \ref{pointMass} as $c$ varies when $\beta = 0.7$}
    \label{fig:1}
\end{figure}

\begin{theorem}\label{ContinuityGeneral}
    Suppose $L(H_n,H) \rightarrow 0$ where $H_n, H$ are bi-variate distributions over $\mathbb{R}_+^2$ and $L(\cdot, \cdot)$ denotes the Levy distance. If there exists $K > 0$ such that 
    \begin{align}\label{TechnicalCondition_continuity}
        \underset{k=1,2}{\max}\bigg\{\underset{n \rightarrow \infty}{\limsup} \int_{\mathbb{R}_+^2} \lambda_k^2 dH_n(\lambda_1, \lambda_2)\bigg\} < K.
    \end{align}
     Then, $\underset{n \rightarrow \infty}{\lim}\textbf{h}(z, H_n) = \textbf{h}(z, H)$ for any $z \in \mathbb{C}_L$.
\end{theorem}

For the proof, please refer to Section \ref{sec:ProofOfContinuityGeneral}.

\section{Special Case: Equal covariance matrices}\label{sec:equal_covariance}
 Now we consider the special case where $\Sigma_{1n} = \Sigma_n = \Sigma_{2n}$. Here, Theorem \ref{mainTheorem} reduces to a simpler form and holds under weaker conditions. In this case, we replace Assumptions $\boldsymbol{T}_4$ and $\boldsymbol{T}_5$ of Theorem \ref{mainTheorem} with $\boldsymbol{T}_4^{'}$ and, $\boldsymbol{T}_5^{'}$ respectively.
\begin{enumerate}
    \item[$\boldsymbol{T}_4^{'}$:] The ESD of $\Sigma_n$ converges weakly to a uni-variate probability distribution $H \neq \delta_0$ almost surely, i.e. $F^{\Sigma_n} \xrightarrow{d} H$ a.s. and $\operatorname{supp}(H) \subset \mathbb{R}_+$.
    \item[$\boldsymbol{T}_5^{'}$:] Further, there exists $C_0 > 0$ such that $\underset{n \rightarrow \infty}{\limsup} \frac{1}{p}\operatorname{trace}(\Sigma_{n}) < C_0$.
\end{enumerate}
It is clear that Assumption $\boldsymbol{T}_5^{'}$ follows from Assumption $\boldsymbol{T}_5$. To characterize the main result of this section, we need the uni-variate analog of the functions (\ref{defining_rho}) that were central to the main result of Section \ref{sec:general_covariance}. 
\begin{definition}
Define the complex-valued functions $\sigma(\cdot), \sigma_2(\cdot)$ as
\begin{align}
\sigma(z) :&= \frac{1}{\mathbbm{i} + z} + \frac{1}{-\mathbbm{i} + z} = \frac{2z}{1+z^2}, z \not \in \{\mathbbm{i}, -\mathbbm{i}\} \text{, and} \label{defining_sigma}\\
\sigma_2(z) :&= \frac{1}{|\mathbbm{i} + z|^2} + \frac{1}{|-\mathbbm{i} + z|^2}, z \not \in \{\mathbbm{i}, -\mathbbm{i}\}. \label{defining_sigma2}
\end{align}
Then for $z \not \in \{\mathbbm{i}, -\mathbbm{i}\}$, we have
\begin{align}\label{realOfSigma}
\Re(\sigma(z)) &= \frac{\Re(\overline{\mathbbm{i}+z})}{|\mathbbm{i}+z|^2} + \frac{\Re(\overline{-\mathbbm{i}+z})}{|-\mathbbm{i}+z|^2} = \sigma_2(z)\Re(z).
\end{align}    
\end{definition}

\begin{corollary}\label{mainCorollary}
    In Theorem \ref{mainTheorem}, suppose we have $\Sigma_{1n} = \Sigma_{2n} = \Sigma_n$ for $n \in \mathbb{N}$ such that $F^{\Sigma_n} \xrightarrow{d} H$ a.s. where $H \neq \delta_0$ is a non-random uni-variate distribution on $\mathbb{R}_+$. Then under $\boldsymbol{T}_1, \boldsymbol{T}_2, \boldsymbol{T}_4^{'}, \boldsymbol{T}_5^{'}$, we have $F^{S_n} \xrightarrow{d} F$ a.s. where $F$ is a non-random distribution with Stieltjes Transform at $z \in \mathbb{C}_L$ given by
\begin{align}\label{s_main_eqn1}
s(z) = \int \frac{dH(\lambda)}{-z + \lambda \sigma(ch(z))} = \dfrac{1}{z}\bigg(\dfrac{2}{c}-1\bigg) - \dfrac{2}{cz}\bigg(\dfrac{1}{1 + c^2h^2(z)}\bigg), 
\end{align}
where, $h(z) \in \mathbb{C}_R$ is the unique number such that 
\begin{align}\label{h_main_eqn1}
h(z) = \int \dfrac{\lambda dH(\lambda)}{-z + \lambda \sigma(ch(z)))}.
\end{align}
Further, $h$ is the Stieltjes Transform of a measure (not necessarily a probability) over the imaginary axis and has a continuous dependence on $H$.
\end{corollary}
 
Unlike in Section \ref{sec:general_covariance}, when both covariance matrices are equal, the uniqueness and continuity (w.r.t the weak topology) of the solution of (\ref{h_main_eqn}) can be proved without requiring any spectral moment bounds (i.e., \ref{secondSpectralBound}, \ref{firstSpectralBound}) and/ or other technical conditions (\ref{TechnicalCondition_continuity}). Moreover, in the special case, the result regarding the continuity of the solution w.r.t. the weak topology is much stronger in the sense that it holds for any weakly converging sequence of distribution functions. Hence, to complete the proof of Corollary \ref{mainCorollary}, we will prove the uniqueness and continuity of the solution of (\ref{h_main_eqn1}) without these extra conditions.

\begin{theorem}\label{Uniqueness_Special}
\textbf{Uniqueness of solution when $\Sigma_{1n} =\Sigma_n = \Sigma_{2n}$:} There exists at most one solution to the following equation within the class of functions that map $\mathbb{C}_L$ to $\mathbb{C}_R$:
    $$h(z) = \int \dfrac{\lambda dH(\lambda)}{-z + \lambda \sigma(ch(z))},$$
where H is any probability distribution function such that $\operatorname{supp}(H) \subset \mathbb{R}_{+}$ and $H \neq \delta_0$.\end{theorem}
The proof is given in Section \ref{sec:ProofOfUniquenessSpecial}.

\begin{theorem} \label{Continuity_Special}
    \textbf{Continuity of solution when $\Sigma_{1n} =\Sigma_n = \Sigma_{2n}$:} Let $H_n, H$ be uni-variate distribution functions 
    satisfying the conditions in Corollary \ref{mainCorollary} and $H_n \xrightarrow{d} H$. For a fixed $z \in \mathbb{C}_L$, denote the unique solutions to (\ref{h_main_eqn1}) corresponding to $H_n$ and $H$ as $h(z, H_n)$ and $h(z, H)$ respectively. Then $h(z,H_n) \xrightarrow{} h(z, H)$.
\end{theorem}
The proof is given in Section \ref{sec:ProofOfContinuitySpecial}.

\section{LSD when the common covariance is the Identity Matrix}\label{sec:identity_covariance}
When $\Sigma_n = I_p$ a.s., we have $F^{\Sigma_n} = \delta_1$ for all $ n \in \mathbb{N}$ and thus $F^{\Sigma_n} \xrightarrow{d} \delta_1$ a.s. So plugging in $H = \delta_1$ in Corollary \ref{mainCorollary}, there exists a probability distribution function $F$ on the imaginary axis such that $F^{S_n} \xrightarrow{d} F$. The LSD $F$ is characterized by $(h, s_{F})$ with $h$ satisfying (\ref{h_main_eqn1}) with $H = \delta_1$ and $(h,s_{F})$ satisfies (\ref{s_main_eqn1}). We will shortly see that $F$ in this case becomes an explicit function of $c$. Therefore, we will henceforth refer to the LSD as $F_c$. The goal of this section is to recover closed form expressions for the distribution $F_c$.

We first note that $h(z)$, the unique solution to (\ref{h_main_eqn1}) with positive real part is the same as $s_{F_c}(z)$ in this case. This is shown below. Writing $h(z)$ as $h$ for simplicity, we have from (\ref{h_main_eqn1}):
\begin{align}\label{h_equal_s}
     &\frac{1}{h} = -z + \frac{2ch}{1 +c^2h^2}\\ \notag
     \implies & c^2zh^3 + (c^2-2c)h^2 +zh + 1=0\\ \notag
     \implies & c^2zh^3 + zh = -1 - h^2(c^2-2c)\\ \notag
     \implies & c^3zh^3 +czh = -c - c^2h^2(c-2)\\ \notag
     \implies & czh(c^2h^2 + 1) = 2-c + c^2h^2(2-c) - 2 = (2-c)(c^2h^2+1) - 2\\ \notag
     \implies & czh = 2-c - \frac{2}{1 + c^2h^2} = 2-c + \frac{1}{\mathbbm{i}}\bigg(\frac{1}{\mathbbm{i} +ch} - \frac{1}{-\mathbbm{i} +ch}\bigg)\\ \notag
     \implies& h = \frac{1}{z}\bigg(\frac{2}{c}-1\bigg) - \frac{2}{cz}\bigg(\frac{1}{1 + c^2h^2}\bigg) = s_{F_c}(z) \text{, by } (\ref{s_main_eqn_1}). 
\end{align}

Therefore, the Stieltjes Transform ($s_{F_c}(z)$) of the LSD at $z \in \mathbb{C}_L$ can be recovered by finding the unique solution with positive real part (exactly one exists by Theorem \ref{mainTheorem}) to the following equation:
\begin{align}\label{5A}
     &\frac{1}{m(z)} = -z + \frac{2cm(z)}{1 +c^2m^2(z)}.
 \end{align}

We simplify (\ref{5A}) to an equivalent functional cubic equation which is more amenable for recovering the roots.
\begin{align}\label{5B}
         c^2zm^3(z) + (c^2 - 2c)m^2(z) + zm(z) + 1 = 0.
\end{align}

For $z \in \mathbb{C}_L$, we extract the functional roots $\{m_j(z)\}_{j=1}^3$ of (\ref{5B}) using \textit{Cardano's method} (subsection 3.8.2 of \cite{Abramowitz}) and select the one which has a positive real component.

\subsection{Deriving the functional roots}\label{Cardano}

We define the following quantities as functions of $c \in (0 , \infty)$.
\begin{align}\label{RQD}
    \left\{ \begin{aligned} 
    &q_0 = \frac{1}{3c^2}; \hspace{18mm} q_2 =-\frac{(c-2)^2}{9c^2}; \hspace{5mm} \Tilde{q} = (q_0, q_2),\\
    &r_1 = -\frac{c+1}{3c^3}; \hspace{12mm} r_3 = -\frac{(c-2)^3}{27c^3}; \hspace{5mm} \Tilde{r} = (r_1, r_3),\\
    &d_0 = \dfrac{1}{27c^6}; \hspace{17mm} d_2 = \dfrac{2c^2 + 10c - 1}{27c^6}; \hspace{5mm}d_4 = \dfrac{(c-2)^3}{27c^5}; \hspace{5mm} \Tilde{d} = (d_0,d_2, d_4),\\
    &Q(z) := q_0 + \frac{q_2}{z^2}; \hspace{7mm} R(z) := \frac{r_1}{z} + \frac{r_3}{z^3}; z \in \mathbb{C}\backslash\{0\}.
\end{aligned} \right.
\end{align}

By Cardano's method, the three roots of the cubic equation (\ref{5B}) are given as follows, where $1, \omega_1, \omega_2$ are the cube roots of unity.
\begin{align}\label{RootFormula}
\left\{ \begin{aligned} 
m_1(z) &= -\dfrac{1-2/c}{3z} + S_{0} + T_{0},\\
m_2(z) & = -\dfrac{1-2/c}{3z} + \omega_1S_{0} + \omega_2T_{0},\\
m_3(z) & = -\dfrac{1-2/c}{3z} + \omega_2S_{0} + \omega_1T_{0},
\end{aligned} \right.
\end{align} 
where, $S_0$ and $T_0$ satisfy
\begin{align}\label{S0T0}
S_0^3 + T_0^3 = 2R(z); \quad
S_0T_0 = -Q(z).
\end{align}

Note that if ($S_0, T_0$) satisfy (\ref{S0T0}), then so do ($\omega_1S_0, \omega_2T_0$) and ($\omega_2S_0, \omega_1T_0$). But exactly one of the functional roots of (\ref{5B}) is the Stieltjes Transform $s_{F}(\cdot)$. This ambiguity in the definition of $S_{0}$ and $T_{0}$ prevents us from pinpointing which one among $\{m_j(z)\}_{j=1}^3$ is the Stieltjes transform of $F$ at $z$ unless we explicitly solve for the roots. However, we will show in Theorem \ref{DensityDerivation} that at points arbitrarily close to the imaginary axis, it is possible to calculate the value of the Stieltjes transform thus allowing us to recover the distribution.

\subsection{Deriving the density of the LSD}

Certain properties of the LSD such as symmetry about $0$ and existence and value of point mass at $0$ have already been established in Proposition \ref{symmetry} and Theorem \ref{pointMass} respectively. Before deriving the density and support of the LSD $F_c$, we introduce a few quantities that parametrize said density.

\begin{definition}\label{supportParameters}
 For $c > 0$, let $\Tilde{d}, R(.), Q(.)$ be as in (\ref{RQD}). Then define
\begin{enumerate}
    \item $R_{\pm} := \dfrac{d_2 \pm \sqrt{d_2^2-4d_0d_4}}{2d_0}; \quad R_\pm$ are real numbers as shown in Lemma \ref{DistributionParameters}.
    \item $L_c := \sqrt{R_{-}\mathbbm{1}_{\{R_{-} > 0\}}}$; \hspace{3mm}  $U_c := \sqrt{R_{+}}$.
    \item $S_c := (-U_c, -L_c)\cup (L_c, U_c)$; It denotes the smallest open set excluding the point $0$\footnote{The point $0$ is treated separately in Theorem \ref{DensityDerivation} as the density at $0$ exists only when $0 < c < 2$.} where the density of the LSD is finite.
    \item For $x \neq 0$, let $r(x) := \underset{\epsilon \downarrow 0}{\lim}\hspace{1mm} R(-\epsilon + \mathbbm{i}x)$ and $q(x) := \underset{\epsilon \downarrow 0}{\lim}\hspace{1mm} Q(-\epsilon + \mathbbm{i}x)$.
    Results related to these limits are established in Lemma \ref{DistributionParameters}.
    \item For $x \neq 0$, $d(x) :=  d_0 - \dfrac{d_2}{x^2} + \dfrac{d_4}{x^4}$.
\end{enumerate}
\end{definition}

\begin{theorem}\label{DensityDerivation}
    $F_c$ is differentiable at $x \neq 0$ for any $c > 0$. Define $V_\pm(x) := |r(x)| \pm \sqrt{-d(x)}$. The functional form of the density is given by    
    $$f_c(x) = \dfrac{\sqrt{3}}{2\pi}\left((V_+(x))^{\frac{1}{3}} - (V_-(x))^{\frac{1}{3}}\right)\mathbbm{1}_{\{x \in S_c\}}.$$ 
    At $x=0$, the derivative exists only when $0 < c < 2$ and is given by
    $$f_c(0) = \dfrac{1}{\pi\sqrt{2c - c^2}}.$$
    The density is continuous wherever it exists.
    \end{theorem}
The proof can be found in Section \ref{ProofDensityDerivation}.

\subsection{Simulation study}\label{sec:SimulationWithSigmaIdentity}
We ran simulations for different values of $c$ while keeping $p = 2000$. A random half of the innovation entries (i.e. $Z_1, Z_2$) were simulated from $\mathcal{N}(0, 1)$\footnote{$\mathcal{N}(\mu,\sigma^2)$ represents the Gaussian distribution with mean $\mu$ and variance $\sigma^2$.} and the other half from $\mathcal{U}(-\sqrt{3}, \sqrt{3}$)\footnote{$\mathcal{U}(a,b)$ represents the uniform distribution over the interval $(a,b)$.}. We proceed to construct the respective commutator matrices (i.e. $S_n = n^{-1}[Z_1, Z_2]$), chart a histogram of their eigen values and overlay the density curve as per Theorem \ref{DensityDerivation}. Figures \ref{fig:c=0.5}, \ref{fig:c=1}, \ref{fig:c=2}, \ref{fig:c=3} \ref{fig:c=4} and \ref{fig:c=5} below show the comparison of the ESDs of these matrices for different against the theoretical distribution (in red). We have also run these simulations for smaller values of $p$ such as $p = 750$. The results are visually similar to the ones provided below.
\begin{figure}[ht!]
 \begin{subfigure}{0.38\textwidth}
     \includegraphics[width=\textwidth]{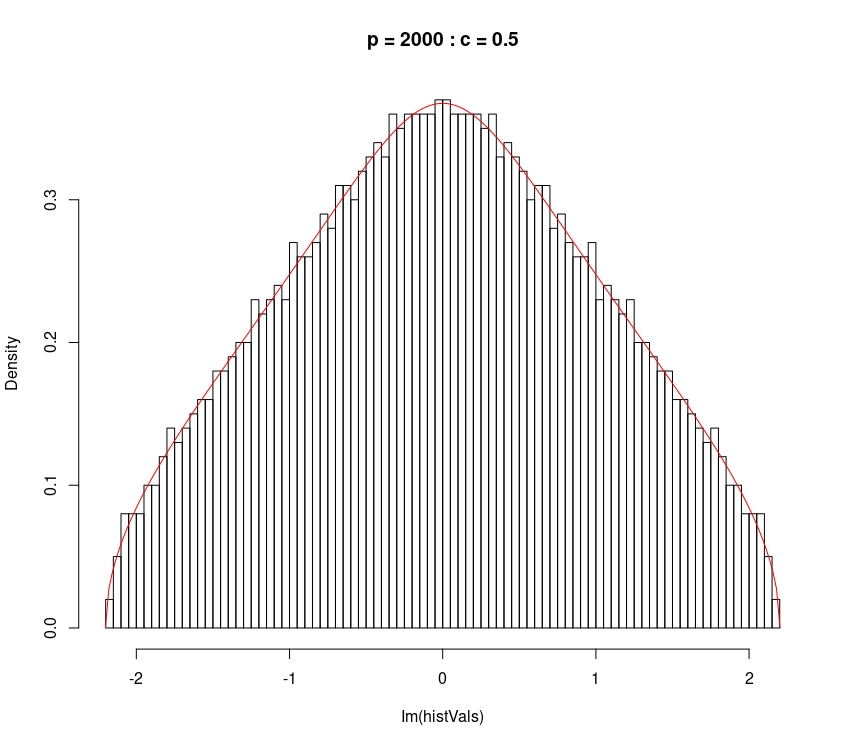}
     \caption{c=0.5}
     \label{fig:c=0.5}
 \end{subfigure}
 \hfill
 \begin{subfigure}{0.38\textwidth}
     \includegraphics[width=\textwidth]{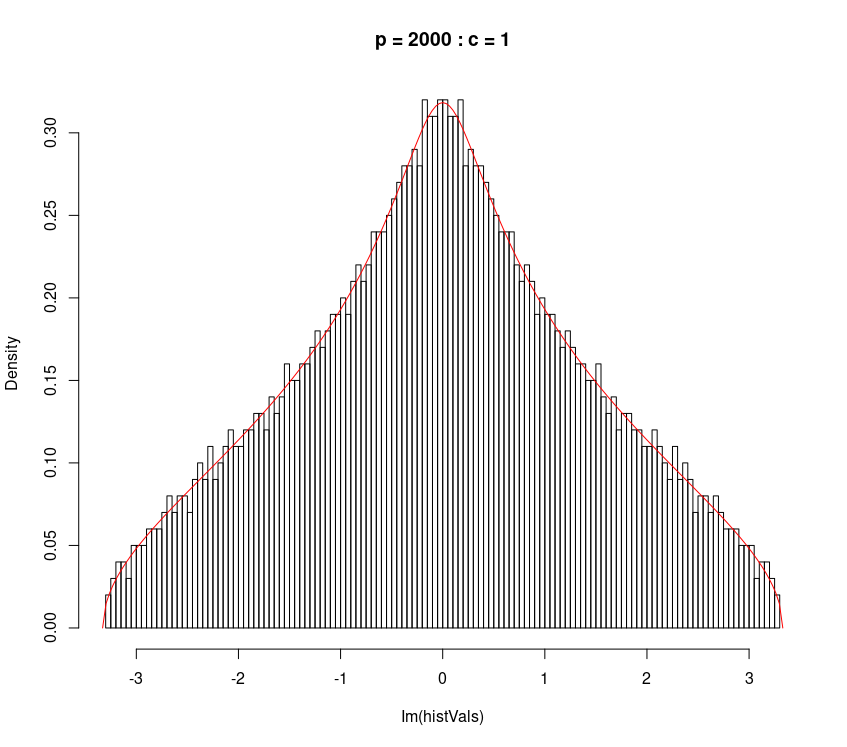}
     \caption{c=1}
     \label{fig:c=1}
 \end{subfigure}
  \medskip

 \begin{subfigure}{0.38\textwidth}
     \includegraphics[width=\textwidth]{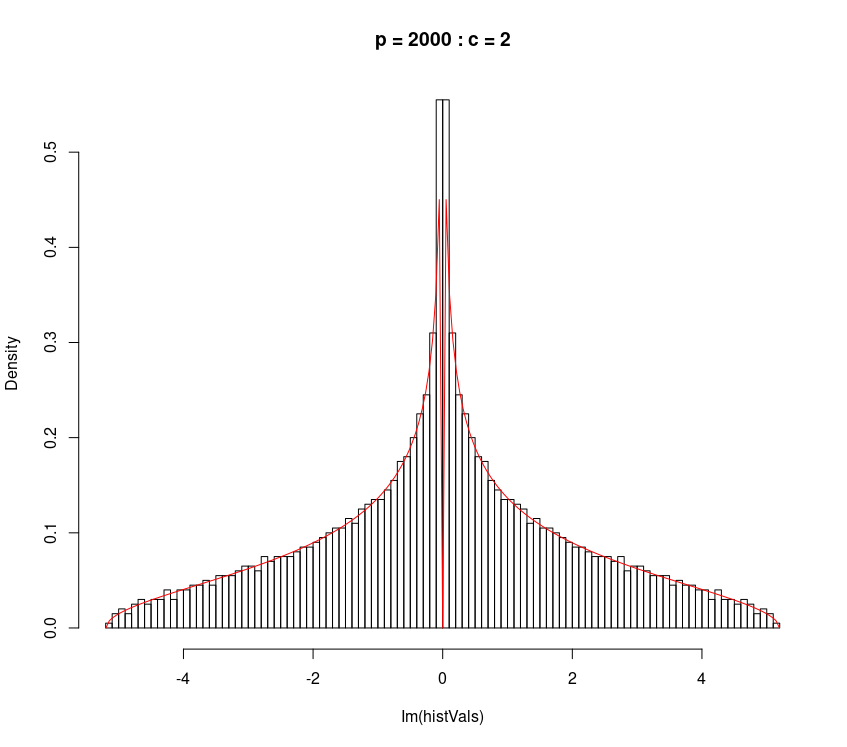}
     \caption{c=2}
     \label{fig:c=2}
 \end{subfigure}
 \hfill
 \begin{subfigure}{0.38\textwidth}
     \includegraphics[width=\textwidth]{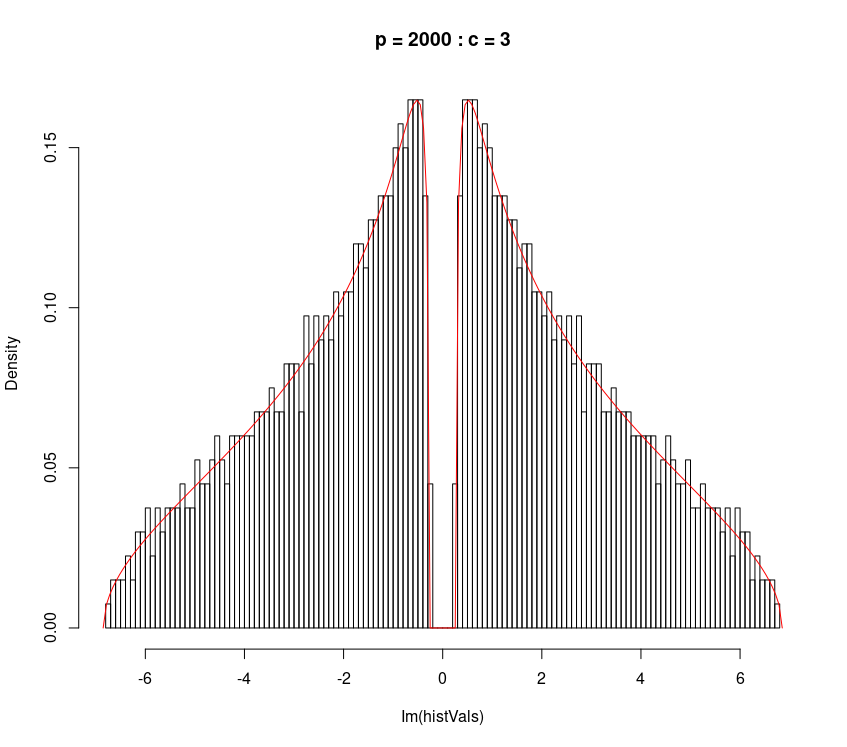}
     \caption{c=3}
     \label{fig:c=3}
 \end{subfigure}

 \medskip
 \begin{subfigure}{0.38\textwidth}
     \includegraphics[width=\textwidth]{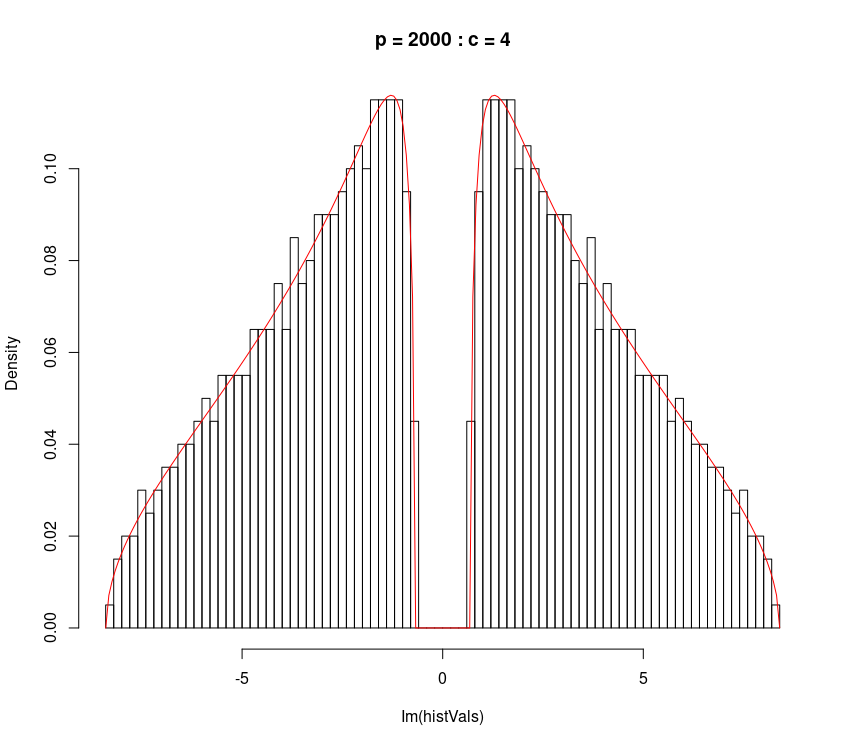}
     \caption{c=4}
     \label{fig:c=4}
 \end{subfigure}
 \hfill
 \begin{subfigure}{0.38\textwidth}
     \includegraphics[width=\textwidth]{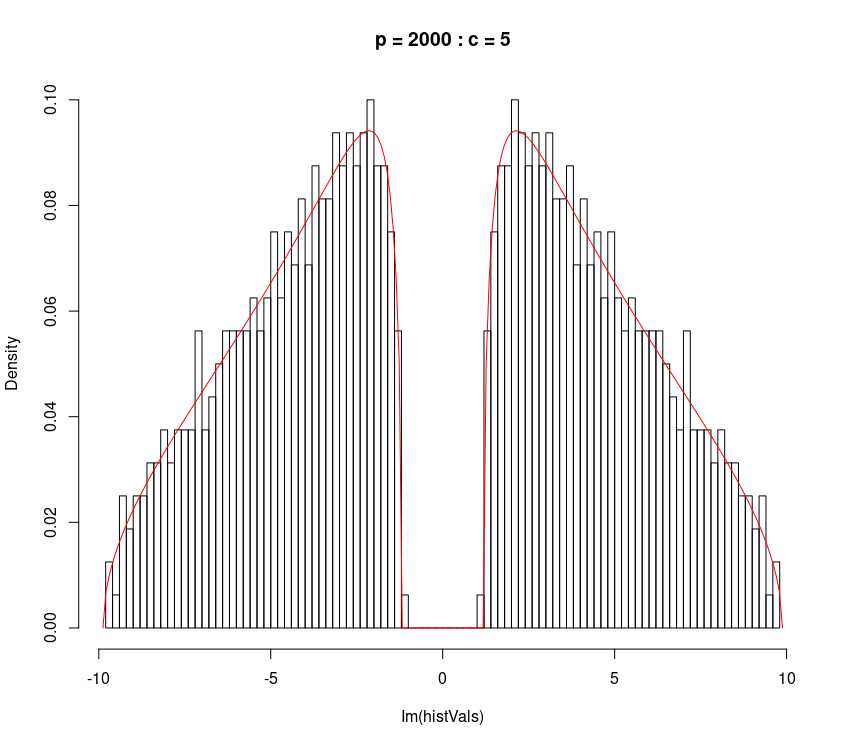}
     \caption{c=5}
     \label{fig:c=5}
 \end{subfigure}
 \caption{Simulated vs. \textcolor{red}{Theoretical} limit distributions at various levels of $c$ for $\Sigma_n = I_p$}
 \label{Label}
\end{figure}

\newpage
\section{The case of the Anti-Commutator Matrix}\label{sec:AntiCommutator}
We define the anti-commutator matrix of $X_1, X_2$ as 
\begin{align}\label{defining_SnPlus}
  S_n^+ = \frac{1}{n}(X_1X_2^* + X_2X_1^*) = \frac{1}{n}\{X_1,X_2\}.
\end{align}

Note that 
\begin{align}\label{anticomm_equivalence}
    \frac{1}{n}[X_1,\mathbbm{i}X_2] = \frac{1}{n}(X_1(\mathbbm{i}X_2)^* - \mathbbm{i}X_2X_1^*) = -\frac{1}{n}\mathbbm{i}(X_1X_2^*+X_2X_1^*) = -\frac{\mathbbm{i}}{n}\{X_1,X_2\}.
\end{align}
This in particular implies that the LSD of the anti-commutator of $X_1$ and $\mathbbm{i}X_2$ is the same as that of the commutator of $X_1$ and $X_2$ upon counter-clockwise rotation by $\pi/2$. Noting that $X_2$ and $\mathbbm{i}X_2$ both satisfy $\boldsymbol{T}_1$ of Theorem \ref{mainTheorem}, we have the following result.

\begin{corollary}\label{mainTheorem_plus}
    Under ($\boldsymbol{T}_1-\boldsymbol{T}_5$) of Theorem \ref{mainTheorem}, $F^{S_n^+} \xrightarrow{d} G$ a.s. where, the Stieltjes Transform of G at $z \in \mathbb{C}^+$ is characterized by the set of equations:
    \begin{align}\label{s_main_eqn_plus}
        s_G(z) = \displaystyle \int_{\mathbb{R}_+^2} \frac{dH(\boldsymbol{\lambda})}{-z -\mathbbm{i}\boldsymbol{\lambda}^T\boldsymbol{\rho}(c\textbf{h}(\mathbbm{i}z))} = \frac{1}{z}\bigg(\frac{2}{c}-1\bigg) - \frac{2}{cz}\bigg(\frac{1}{1+c^2h_1(\mathbbm{i}z)h_2(\mathbbm{i}z)}\bigg),
    \end{align}
    where, $\textbf{h}(\mathbbm{i}z) = (h_1(\mathbbm{i}z), h_2(\mathbbm{i}z))^T \in \mathbb{C}_R^2$ are unique numbers, such that
    \begin{align}\label{h_main_eqn_plus}
        \mathbbm{i} \textbf{h}(\mathbbm{i}z) &= \displaystyle \int_{\mathbb{R}_+^2} \frac{\boldsymbol{\lambda} dH(\boldsymbol{\lambda})}{-z -\mathbbm{i}\boldsymbol{\lambda}^T\boldsymbol{\rho} (c\textbf{h}(\mathbbm{i}z))}.
    \end{align}
    Moreover, $h_1, h_2$ themselves are Stieltjes Transforms of measures (not necessarily probability measures) over the imaginary axis and continuous as functions of $H$.
\end{corollary}
\begin{proof}
    The proof is immediate from (\ref{anticomm_equivalence}) and Lemma \ref{StieltjesTransformRotated}.
\end{proof}

\section{Relaxation of commutativity requirement}\label{sec:RelaxingCommutativity}

We present a set of conditions which are strictly weaker than $\boldsymbol{T}_3$ but which are sufficient for Theorem \ref{mainTheorem} to hold. For the result below, we denote the nuclear norm of a matrix $A$ by $||A||_*$.

\begin{theorem}\label{thm:RelaxingCommutativty}
Let $\Sigma_{jn} = P_{jn}D_{jn}P_{jn}^*$ denote a spectral decomposition of $\Sigma_{jn};j=1,2$. We construct $\Phi_{1n} := P_{2n}D_{1n}P_{2n}^*$ and $\Phi_{2n} := P_{1n}D_{2n}P_{1n}^*$. Instead of $\boldsymbol{T}_3$ of Theorem \ref{mainTheorem}, suppose any of the following conditions hold:
\begin{enumerate}
    \item[\textbf{C1}:] $\dfrac{1}{p}\operatorname{rank}(\Sigma_{1n} - \Phi_{1n})= o(1) \text{ or } \frac{1}{p}\operatorname{rank}(\Sigma_{2n} - \Phi_{2n})= o(1)$,
    \vspace{3mm}
    \item[\textbf{C2}:]
$\frac{1}{p}\operatorname{rank}(P_{1n} - P_{2n}) = o(1)$, 
\vspace{3mm}
    \item[\textbf{C3}:] 
$\frac{1}{p}||\Sigma_{1n} - \Phi_{1n}||_*= o(1) \text{ or } \frac{1}{p}||\Sigma_{2n} - \Phi_{2n}||_*= o(1).$
\end{enumerate}
Note that $\Sigma_{1n}$ and $\Phi_{2n}$ share the same eigen basis and so do $\Sigma_{2n}$ and $\Phi_{1n}$. Therefore, we can define $H_{1n} := \operatorname{JESD}(\Sigma_{1n}, \Phi_{2n})$ and $H_{2n}:= \operatorname{JESD}(\Phi_{1n}, \Sigma_{2n})$. Suppose either $\{H_{1n}\}$ or $\{H_{2n}\}$ converges weakly to some $H$ that satisfies \textbf{T4} of Theorem \ref{mainTheorem}. Then, the conclusion of Theorem \ref{mainTheorem} holds replacing $H_n$ appropriately with $H_{1n}$ or $H_{2n}$.
\end{theorem}

\begin{proof}
Analogous to $$S_n = \frac{1}{n}(\Sigma_{1n}^\frac{1}{2}Z_1Z_2^*\Sigma_{2n}^\frac{1}{2} - \Sigma_{2n}^\frac{1}{2}Z_2Z_1^*\Sigma_{1n}^\frac{1}{2}),$$
we consider the matrices,
$$M_{1n} = \frac{1}{n}(\Sigma_{1n}^\frac{1}{2}Z_1Z_2^*\Phi_{2n}^\frac{1}{2} - \Phi_{2n}^\frac{1}{2}Z_2Z_1^*\Sigma_{1n}^\frac{1}{2}),$$
and
$$M_{2n} = \frac{1}{n}(\Phi_{1n}^\frac{1}{2}Z_1Z_2^*\Sigma_{2n}^\frac{1}{2} - \Sigma_{2n}^\frac{1}{2}Z_2Z_1^*\Phi_{1n}^\frac{1}{2}).$$

Note that $M_{1n}, M_{2n}$ are random commutators whose components $(Z, \Sigma, \Phi)$ satisfy the conditions of the Theorem \ref{mainTheorem}. Observe that,
\begin{align*}
    ||F^{S_n} - F^{M_{1n}}||  \leq \frac{1}{p}\operatorname{rank}(S_n - M_{1n}) \leq \frac{2}{p}\operatorname{rank}(\Sigma_{2n}^\frac{1}{2} - \Phi_{2n}^\frac{1}{2}) \leq \frac{4}{p}\operatorname{rank}(P_{2n}-P_{1n}) \longrightarrow 0.
\end{align*}
Therefore, under conditions \textbf{C1} or \textbf{C2} of the theorem, $F^{S_n}$ almost surely shares the same weak limit as that of $F^{M_{1n}}$ and similarly also with $F^{M_{2n}}$.

To show sufficiency of \textbf{C3}, note that the deterministic equivalent for the resolvent $Q(z) = (S_n - zI_p)^{-1};z \in \mathbb{C}_L$ from Theorem \ref{DeterministicEquivalent} was as follows:
\begin{align}
\Bar{Q}(z) = \bigg(-zI_p + \rho_1(c_n\mathbb{E}{\textbf{h}}_{n}(z))\Sigma_{1n} + \rho_2(c_n\mathbb{E}{\textbf{h}}_{n}(z))\Sigma_{2n}\bigg)^{-1}.
\end{align}
Define
\begin{align*}
    R_1(z) &:= \bigg(-zI_p + \rho_1(c_n\mathbb{E}{\textbf{h}}_{n}(z))\Sigma_{1n} + \rho_2(c_n\mathbb{E}{\textbf{h}}_{n}(z))\Phi_{2n}\bigg)^{-1}, \text{ and}\\
    R_2(z) &:= \bigg(-zI_p + \rho_1(c_n\mathbb{E}{\textbf{h}}_{n}(z))\Phi_{1n} + \rho_2(c_n\mathbb{E}{\textbf{h}}_{n}(z))\Sigma_{2n}\bigg)^{-1}.
\end{align*}

Note that $R_1$ and $R_2$ are ideal candidates for a deterministic  equivalent of $Q(z)$ since, $(\Sigma_{1n},\Phi_{2n})$ and $(\Phi_{1n},\Sigma_{2n})$ commute. In particular, using $R_1$ or $R_2$ in place of $\bar{Q}$ in our work will lead exactly to the results of Theorem \ref{mainTheorem}. Hence, it suffices to show that 
$$\absmod{\frac{1}{p}\operatorname{trace}(\bar{Q}(z) - R_j(z))} \rightarrow 0.$$

By (\ref{R0}), we have
\begin{align*}
    &\absmod{\frac{1}{p}\operatorname{trace}(\bar{Q}(z)-R_1(z))} =\frac{1}{p}\absmod{\rho_2(c_n\mathbb{E}{\textbf{h}}_{n}(z))}\absmod{\operatorname{trace}(R_1\bar{Q}( \Sigma_{2n} - \Phi_{2n} ))}.
\end{align*}
We have the following observations.
\begin{enumerate}
    \item For fixed $z \in \mathbb{C}_L$, $\rho_2(c_n\mathbb{E}{\textbf{h}}_{n}(z))$ is bounded since $\mathbb{E}(\textbf{h}_n(z)) \rightarrow \textbf{h}(z) \in \mathbb{C}_R^2$.
    \item $ ||R_1(z)||, ||R_2(z)|| \leq 1/|\Re(z)|$ follow from standard results.
    \item $||\bar{Q}(z)|| \leq 1/|\Re(z)|$. This follows from the fact that $\textbf{h}_n(\mathbb{C}_L) \subset \mathbb{C}_R^2$ and $\rho_j(\mathbb{C}_R^2) \subset \mathbb{C}_R;j=1,2$ which for $z = -u+\mathbbm{i}v;u>0$ implies that 
    $$(\bar{Q}(z))^{-1} = (uI_p + A) + \mathbbm{i}(-v + B),$$
    where, A is a real p.s.d. matrix and B is some $p\times p$ matrix.
\end{enumerate}
For any $p \times p$ matrix $A$, by Cauchy Schwarz we have   
$$|\operatorname{trace}(A)| \leq \sqrt{p\operatorname{trace}(A^*A)},$$
and when B is p.s.d, we have 
$$|\operatorname{trace}(AB)| \leq ||A||_{op}\operatorname{trace}(B).$$
Using these, we observe that
\begin{align*}
    \frac{1}{p^2}\absmod{\operatorname{trace}\bigg(R_1\bar{Q}(\Sigma_{2n} - \Phi_{2n})\bigg)}^2
    &\leq \frac{1}{p^2} \times p\absmod{\operatorname{trace}\bigg(\bar{Q}^*R_1^*R_1\bar{Q}(\Sigma_{2n} - \Phi_{2n})(\Sigma_{2n} - \Phi_{2n})^*\bigg)}\\
    &\leq \frac{1}{p}||R_1\bar{Q}||_{op}^2 \times ||\Sigma_{2n} - \Phi_{2n}||_F^2\\
    &\leq \frac{1}{|\Re(z)|^4} \times  \frac{1}{p} ||\Sigma_{2n} - \Phi_{2n}||_F^2\\
    & \leq \frac{1}{|\Re(z)|^4} \times  \frac{1}{p} ||\Sigma_{2n} - \Phi_{2n}||_{op}||\Sigma_{2n} - \Phi_{2n}||_*\\
    &\leq \frac{2\tau}{|\Re(z)|^4} \times  \frac{1}{p} ||\Sigma_{2n} - \Phi_{2n}||_*\longrightarrow 0.
\end{align*}
Here, we used the fact that  $||A||_F^2 \leq ||A||_{op}||A||_*$. Therefore,
$$\frac{1}{p}\operatorname{trace}(\bar{Q}(z)- R_1(z)) \longrightarrow 0,$$
and similarly,
$$\frac{1}{p}\operatorname{trace}(\bar{Q}(z)- R_2(z)) \longrightarrow 0.$$
\end{proof}
\begin{remark}\label{SpecialCasesOfNonCommutativity}
As a generalization of the Householder construction for unitary matrices, let $P_{jn} = I_p - 2U_{jn}U_{jn}^*$, where $U_{jn}$ is a $p \times k$ matrix with orthonormal vectors. If $k = o(p)$, then $P_{jn}$ satisfy \textbf{C2} of Theorem \ref{thm:RelaxingCommutativty}. Therefore, if the eigen bases of $\Sigma_{1n},\Sigma_{2n}$ are constructed as above, the result of Theorem \ref{mainTheorem} holds even without commutativity of $\Sigma_{2n}$ and $\Sigma_{2n}$. 
\end{remark}

We now state a conjecture regarding a sufficient condition for our main result to hold.
\begin{conjecture}
    Suppose for any $n_1,n_2 \in \mathbb{N}$ and every selection of non-negative integers, $(k_1,\ldots,k_{n_1})$ and $(l_1,\ldots,l_{n_2})$, the following holds:

$$\frac{1}{p}\operatorname{trace}(\Sigma_1^{k_1}\Sigma_2^{l_1}\Sigma_1^{k_2}\Sigma_2^{l_2}\ldots) \longrightarrow \int \lambda_1^{\sum_i k_i}\lambda_2^{\sum_j l_j}dH(\lambda_1,\lambda_2).$$

\textit{In other words, the joint tracial moment converge to appropriate quantities which are functions of the joint limiting spectral distribution of ($\Sigma_1,\Sigma_2$).} Then, the conclusion of Theorem \ref{mainTheorem} holds. 
\end{conjecture}

This can be seen by carrying out a formal power series expansion of the trace of the deterministic equivalent
and matching the coefficients of powers of $z$ with those for a similar expansion of the Stieltjes transform of the LSD. However, a complete analysis of this would require adopting advanced combinatorial techniques. Such analysis is outside the scope of the current paper since we are utilizing the method of Stieltjes Transform.

\textbf{Numerical simulations: }
We now show the impact of non-commuting scaling matrices on the main result. Remark \ref{SpecialCasesOfNonCommutativity} already states classes of non-commuting scaling matrices under which Theorem \ref{mainTheorem} holds. So for the purpose of simulations, we will use non-commuting matrices generated by Haar-distributed orthogonal matrices. This is a class of matrices that go beyond Theorem \ref{thm:RelaxingCommutativty}.

Unlike in Section \ref{sec:SimulationWithSigmaIdentity}, here we do not know the exact functional form of the density. However, we can estimate it by numerically inverting the Stieltjes transform as given in Theorem \ref{mainTheorem} and compare against the observed eigen values. For various values of $c$, figures \ref{fig:c=0.5_nc}, \ref{fig:c=1_nc}, \ref{fig:c=2_nc}, \ref{fig:c=3_nc}, \ref{fig:c=4_nc} and \ref{fig:c=5_nc} below show the ESDs of the commutators against the numerically estimated theoretical density values. The exact steps that we have followed are as follows.
\begin{enumerate}
    \item Take $H=0.25\delta_{(1,1)} + 0.25\delta_{(1,2)} + 0.25\delta_{(2,1)} + 0.25\delta_{(2,2)}$ and $p=2000$.
    \item For $j=1,2$, $P_{jn} = (V_{jn}^*V_{jn})^{-\frac{1}{2}}V_{jn}^T$ where $V_{jn} \in \mathbb{R}^{p\times p}$ are independent matrices with i.i.d. standard Gaussian entries.
    \item Simulate $p$ pairs of eigen values from $H$ and denote it by $E = [E_{\cdot 1},E_{\cdot,2}]$ which is a $p \times 2$ matrix.
    \item Set $\Sigma_{jn} = P_{jn} \operatorname{Diag}(E_{\cdot j}) P_{jn}^T$. Then, $\Sigma_{1n}$ and $\Sigma_{2n}$ do not commute almost surely.
    \item $Z_1,Z_2$ were constructed exactly as in Section \ref{sec:SimulationWithSigmaIdentity}. 
    \item Let $X_j = \Sigma_{jn}^\frac{1}{2}Z_j$ and $S_n = n^{-1}[X_1,X_2]$ as earlier and plot the ESD. 
    \item Solving for equations \ref{h_main_eqn} and \ref{s_main_eqn} at $z \in \mathbb{C}_L$ close to the imaginary axis, we use (\ref{inversion_density}) to numerically estimate the density across the support and superimpose on top of the ESD. 
\end{enumerate}
This empirical evidence suggests that Theorem \ref{mainTheorem} may continue to hold even in a setting where $\Sigma_1$ and $\Sigma_2$ do not commute, even approximately (in the sense described in conditions \textbf{C1} -- \textbf{C3} above).

\begin{figure}[ht!]
 \begin{subfigure}{0.47\textwidth}
     \includegraphics[width=\textwidth]{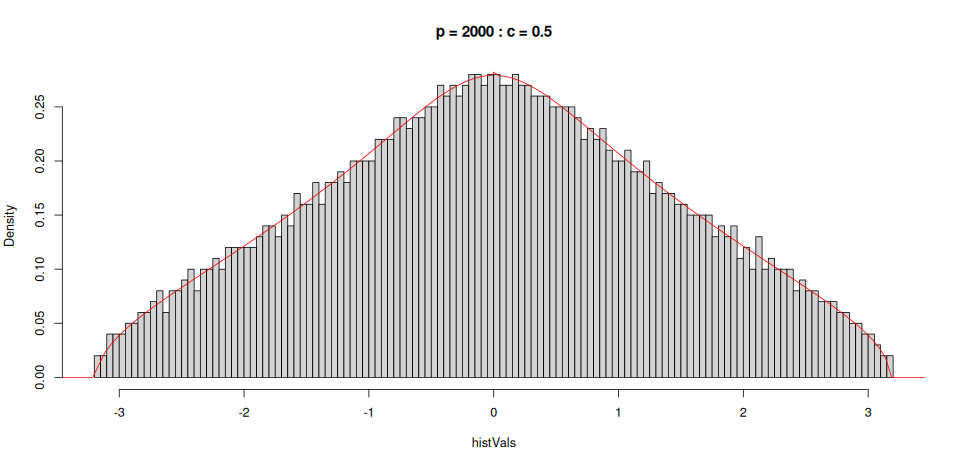}
     \caption{c=0.5}
     \label{fig:c=0.5_nc}
 \end{subfigure}
 \hfill
 \begin{subfigure}{0.47\textwidth}
     \includegraphics[width=\textwidth]{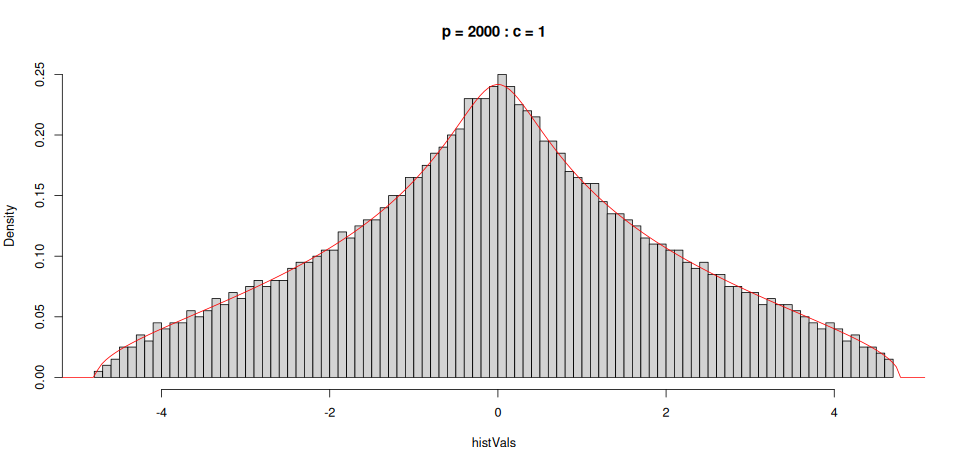}
     \caption{c=1}
     \label{fig:c=1_nc}
 \end{subfigure}
  \medskip

 \begin{subfigure}{0.47\textwidth}
     \includegraphics[width=\textwidth]{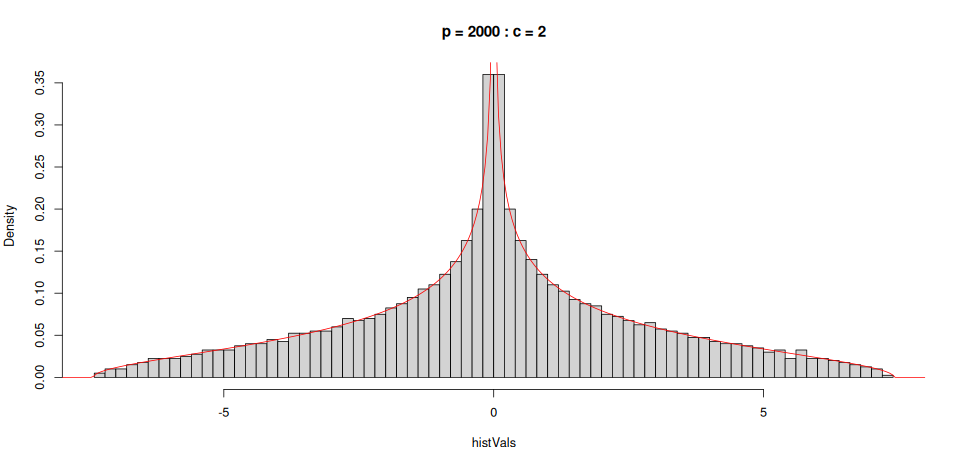}
     \caption{c=2}
     \label{fig:c=2_nc}
 \end{subfigure}
 \hfill
 \begin{subfigure}{0.47\textwidth}
     \includegraphics[width=\textwidth]{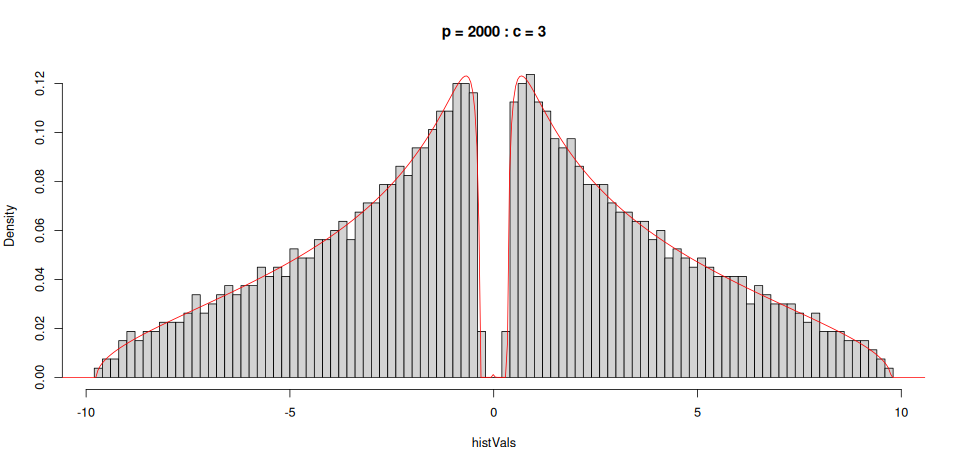}
     \caption{c=3}
     \label{fig:c=3_nc}
 \end{subfigure}

 \medskip
 \begin{subfigure}{0.47\textwidth}
     \includegraphics[width=\textwidth]{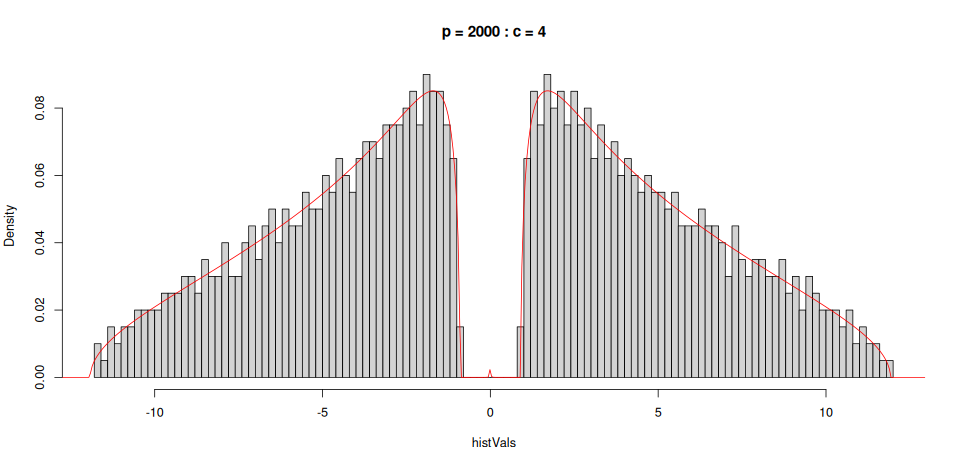}
     \caption{c=4}
     \label{fig:c=4_nc}
 \end{subfigure}
 \hfill
 \begin{subfigure}{0.47\textwidth}
     \includegraphics[width=\textwidth]{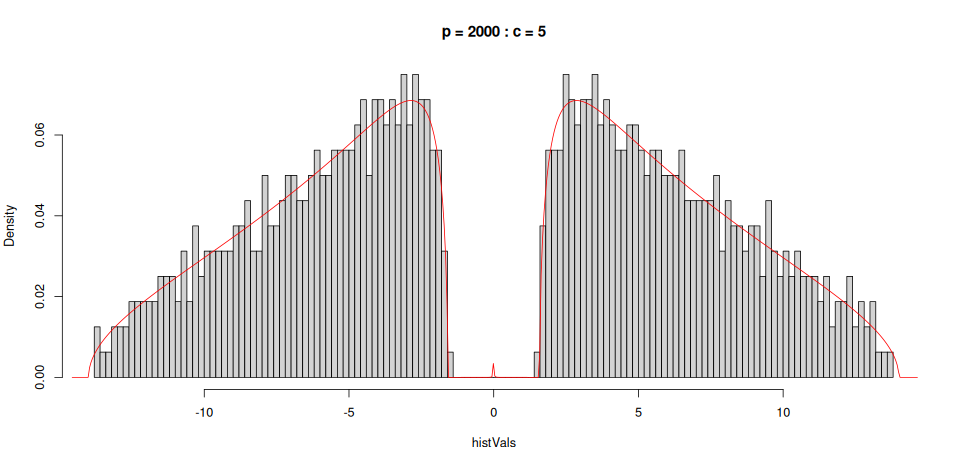}
     \caption{c=5}
     \label{fig:c=5_nc}
 \end{subfigure}
 \caption{Simulated vs. \textcolor{red}{Theoretical} limit distributions at various levels of $c$ for $H = 0.25\delta_{(1,1)} + 0.25\delta_{(1,2)} + 0.25\delta_{(2,1)} + 0.25\delta_{(2,2)}$ under non-commuting scaling matrices}
 \label{Label1}
\end{figure}

\section{An Inference Problem for Equi-correlated Paired Data}\label{sec:Application}

In this section, we propose a statistical model for paired, high-dimensional data, and show how the LSD of the commutator of the data matrices can be used for the purpose of determining independence between the paired populations.


Consider a set of $n$ \textit{paired} $p$-dimensional observations from jointly multivariate Gaussian distributions. Denote the two samples as $X_1 =\Sigma^\frac{1}{2}Z$ and $X_2 = \Sigma^\frac{1}{2}W$ where    $W = (W_{ij}), Z = (Z_{ij}) \in \mathbb{R}^{p \times n}$ are such that, $(W_{ij},Z_{ij})$ are i.i.d. bivariate normal, with zero mean, unit variance, and Corr$(W_{ij},Z_{ij})=\rho$. Thus, the parameter $\rho$ represents element-wise dependence in the underlying innovations. An investigator would like to test the hypothesis $H_0: \rho = 0$ against $H_1: \rho \neq 0$.

We can characterize this dependence in terms of another independent Gaussian random matrix $V = (V_{ij})$, with i.i.d. standard normal entries, as follows. Observe that, distributionally, we have the following representation:
$$
W_{ij} = \rho Z_{ij} + \sqrt{1 - \rho^2}V_{ij}, \qquad \mbox{for}~i=1,\ldots,p,~j=1,\ldots,n.
$$
We see that 
\begin{align}\label{test1}
n^{-1}[X_1, X_2] = n^{-1}\Sigma^\frac{1}{2}[Z, W]\Sigma^\frac{1}{2} &=n^{-1} \Sigma^\frac{1}{2}(ZW^* - WZ^*)\Sigma^\frac{1}{2}\\
       &= n^{-1}\Sigma^\frac{1}{2}\bigg(Z(\rho Z^* + \sqrt{1-\rho^2})V^*) - (\rho Z + \sqrt{1-\rho^2}V)Z^* \bigg)\Sigma^\frac{1}{2}\notag\\
       &= n^{-1}\sqrt{1-\rho^2}\Sigma^\frac{1}{2}(ZV^*-VZ^*)\Sigma^\frac{1}{2}\notag\\ 
       &= \sqrt{1-\rho^2}\bigg(n^{-1}\Sigma^\frac{1}{2}[Z, V]\Sigma^\frac{1}{2}\bigg). \notag
\end{align}
Note that under the null hypothesis, $Z$ and $W$ are independent, thus allowing us to derive the limiting spectral distribution of $n^{-1}[X_1, X_2]$ using Corollary \ref{mainCorollary}. Even under the alternative, (\ref{test1}) allows us to derive the limiting spectral distribution of $n^{-1}[X_1, X_2]$, by using the fact that $Z$ and $V$ are independent and applying the same corollary. Indeed, under the alternative, the only change in the form of the LSD is that the support shrinks by a factor of $\sqrt{1-\rho^2}$. This result can be helpful in deriving asymptotic properties of test statistics for testing $H_0:\rho=0$ vs. $H_1:\rho \neq 0$ if such statistics are derived from linear functionals of the eigenvalues of $n^{-1}[X_1, X_2]$.

\subsection{Potential applications}

Below, we present three real-life scenarios involving paired data where we can potentially use our result to find solutions.
\begin{enumerate}
    \item[1] The data arise from sib-ship studies involving $n$ pairs of siblings, where the $p$-dimensional observations are quantitative traits measured at different locations along the genome. Here,  $p$ denotes the number of measurement locations along the genome. The \textbf{underlying} assumption is that the measurements at different genetic locations are correlated (described by the matrix $\Sigma$), whereas the \textbf{underlying} innovations have the same correlation ($\rho$) between the sibling pairs across different coordinates. Our model is related to, but different from commonly used models for QTL (quantitative trait loci) mapping involving sibling pairs (cf. \cite{Risch}, \cite{Cardon2000}, \cite{Howe}).
    
    \item[2] Another application can be studies investigating the joint behavior of two pollutants. The data consist of $p$-dimensional observations measured at different spatial locations (e.g. observation centers) across $n$ time points. As before, we assume that the measurements at different spatial locations are correlated (described by $\Sigma$), with the underlying innovations sharing the same correlation ($\rho$) between the pollutant pairs across different coordinates. 
    
    \item[3] A further potential application involves EEG data of the brain in the resting state versus awake state, for $n$ individuals. The data consist of $p$-dimensional observations measured at the $p$ electrode locations in the scalp of $n$ individuals in the two states. As before, we assume that the measurements at different electrode locations are correlated (described by $\Sigma$), with the underlying innovations sharing the same correlation ($\rho$) between the pair of brain states across different coordinates.
\end{enumerate}

\subsection{Testing method}

As noted earlier, the key observation behind formulating a test for the hypothesis $H_0:\rho=0$ against $H_A: \rho\neq 0$ is that under the alternative, the only change in the form of the LSD is that the support shrinks by a factor of $\sqrt{1-\rho^2}$ compared to that in the independent (i.e., $H_0:\rho =0$) setting. Therefore, we focus on the behavior of the statistic $\sqrt{\int \lambda^2 dF_n(\lambda)}$, where $F_n$ denotes the ESD of the commutator matrix,  and use it to formulate the test procedure.

We demonstrate this idea with a numerical example. Taking $p = 1000,c = 2$ and $\Sigma_1 = I_p = \Sigma_2$, we simulated $X_1,X_2$ once under $\rho=0$ and then under $\rho = 0.7$. Denoting the commutator under the $\rho = 0$ scenario as $S_0$ and the one under the $\rho=0.7$ scenario as $S_1$, we calculate the following quantities:
$$\lambda_{1} := \sqrt{\int \lambda^2dF^{S_1}(\lambda)}, \qquad \qquad \lambda_{0} := \sqrt{\int \lambda^2dF^{S_0}(\lambda)}.$$

The observed values are $\lambda_{1} = 1.43$ and $\lambda_0 = 1.99$. The shrinkage ratio is $0.714$ which is very close to the theoretical estimate given by $\sqrt{1 - 0.7^2}$. The shrinkage effect is displayed in Figure \ref{fig:fig1}.

To further validate this point, we repeated this exercise $100$ times. Keeping $p$ fixed at $1000$, we randomly generated values of $c \sim \operatorname{Unif}(0.25, 3)$ and $\rho \sim \operatorname{Unif}(-1, 1)$, and then constructed the $S_1$ matrices under the actual value of $\rho$ and then constructed the $S_0$ matrices when $\rho=0$. Figure \ref{fig:fig2} plots the theoretical shrinkage factor (i.e., $\sqrt{1-\rho^2}$) on the X-axis and its observed counterpart (i.e., $\lambda_1/\lambda_0$) on the Y-axis. The observed shrinkage factors are extremely close to their theoretical counterparts.

If $\Sigma$ is known, several possible approaches can be used to test $H_0:\rho = 0$ by first de-correlating the data. The problem is significantly more challenging when $\Sigma$ is unknown. We propose a method based on the spectral statistics of the commutator of data matrices that involve the Population Spectral Distribution (PSD) of $\Sigma$. When $\Sigma$ is unknown, we first estimate the PSD and then use this as a plug-in estimate in the test statistic. We explain each case in detail along with some numerical results.

\subsection{\texorpdfstring{$\Sigma$}{} is known:}
The steps of the hypothesis testing procedure are as described below.
\begin{enumerate}
    \item 
    Let $X_1^{obs},X_2^{obs}$ be the two observed matrices of dimension $p \times n$ with $F_n$ denoting the ESD of the commutator matrix $S_n = \frac{1}{n}[X_1^{obs}, X_2^{obs}]$. Calculate $$T_{obs} = \sqrt{\int \lambda^2 dF_n(\lambda)}.$$ 
    \item 
    Let $B$ be a  large integer. For $b = 1 \text{ to } B$, we repeat the following operations.
\begin{enumerate}
    \item 
    Construct independent random matrices $Z_1^{(b)}$ and $Z_2^{(b)}$ of dimension $p \times n$ with i.i.d. standard Gaussian entries. 
    \item Using the known value of $\Sigma$, we  generate $$S_0^{(b)} := \frac{1}{n}\Sigma^\frac{1}{2}[Z_1^{(b)},Z_2^{(b)}]\Sigma^\frac{1}{2}.$$
    \item 
    Let $F_0^{(b)}$ denote the ESD of $S_0^{(b)}$. Calculate $$T_b := \sqrt{\int \lambda^2 dF_0^{(b)}(\lambda)}.$$
\end{enumerate}
    \item 
    Smaller values of $T_{obs}$ (with reference to the distribution of $T_b$) lead to rejection of the null hypothesis with a $p$-value ($PV$) which is derived from the sampling distribution of $T_b$ as follows:
    $$PV = \frac{1}{B}\sum_{b=1}^B \mathbbm{1}_{\{T_b \leq T_{obs}\}}.$$
\end{enumerate}

\textbf{Numerical Results:} We have tested the above algorithm under various settings of true population spectral distribution and true values of equi-correlation coefficients. Table \ref{tab:1} ($p = 50; n = 500; B = 1000$) and Table \ref{tab:2} ($p = 100; n=500; B = 1000$) below lists the p-values obtained under these combinations. We observe that the algorithm rejects the null hypothesis with high power as soon as $\rho$ exceeds $0.25$.

\begin{table}[ht!]
    \centering
    \begin{tabular}{|l||c|c|c|c|c|c|}
    \hline
         \textbf{True Spectral Distribution} & $\boldsymbol{\rho}=\textbf{0.1}$ & $\boldsymbol{\rho}=\textbf{0.25}$ & $\boldsymbol{\rho}=\textbf{0.5}$ & $\boldsymbol{\rho}=\textbf{0.75}$ & $\boldsymbol{\rho}=\textbf{0.8}$ & $\boldsymbol{\rho}=\textbf{0.95}$\\
         \hline
         \hline
          $H_p = \delta_1$ & 0.649 & 0.192 & 0 & 0 & 0 & 0 \\
          \hline
          $H_p = 0.4 \delta_0 + 0.6 \delta_1$ & 0.725 & 0.392  & 0 & 0 & 0 & 0 \\
          \hline
          $H_p = 0.5 \delta_1 + 0.5 \delta_2$ & 0.001 & 0 & 0 & 0 & 0 & 0 \\
          \hline
          $H_p = 0.3 \delta_0 + 0.4 \delta_1 + 0.3 \delta_2$ & 0.500 & 0.242 & 0.001 & 0 & 0 & 0 \\
      \hline
    \end{tabular}
    \caption{p-values under different true spectral distributions and equi-correlation coefficients under the setting $p=50;n=500$}
    \label{tab:1}
\end{table}

\begin{table}[ht!]
    \centering
    \begin{tabular}{|l||c|c|c|c|c|c|}
    \hline
         \textbf{True Spectral Distribution} & $\boldsymbol{\rho}=\textbf{0.1}$ & $\boldsymbol{\rho}=\textbf{0.25}$ & $\boldsymbol{\rho}=\textbf{0.5}$ & $\boldsymbol{\rho}=\textbf{0.75}$ & $\boldsymbol{\rho}=\textbf{0.8}$ & $\boldsymbol{\rho}=\textbf{0.95}$\\
         \hline
         \hline
          $H_p = \delta_1$ & 0.466 & 0.004 & 0 & 0 & 0 & 0\\
          \hline
          $H_p = 0.4 \delta_0 + 0.6 \delta_1$ & 0.710 & 0.206 & 0 & 0 & 0 & 0 \\
          \hline
          $H_p = 0.5 \delta_1 + 0.5 \delta_2$ & 0.025 & 0 & 0 & 0 & 0 & 0 \\
          \hline
          $H_p = 0.3 \delta_0 + 0.4 \delta_1 + 0.3 \delta_2$ & 0.392 & 0.049 & 0 & 0 & 0 & 0 \\
      \hline
    \end{tabular}
    \caption{p-values under different true spectral distributions and equi-correlation coefficients under the setting $p=100;n=500$}
    \label{tab:2}
\end{table}

\begin{figure}
    \centering
    \includegraphics[width=0.75\linewidth]{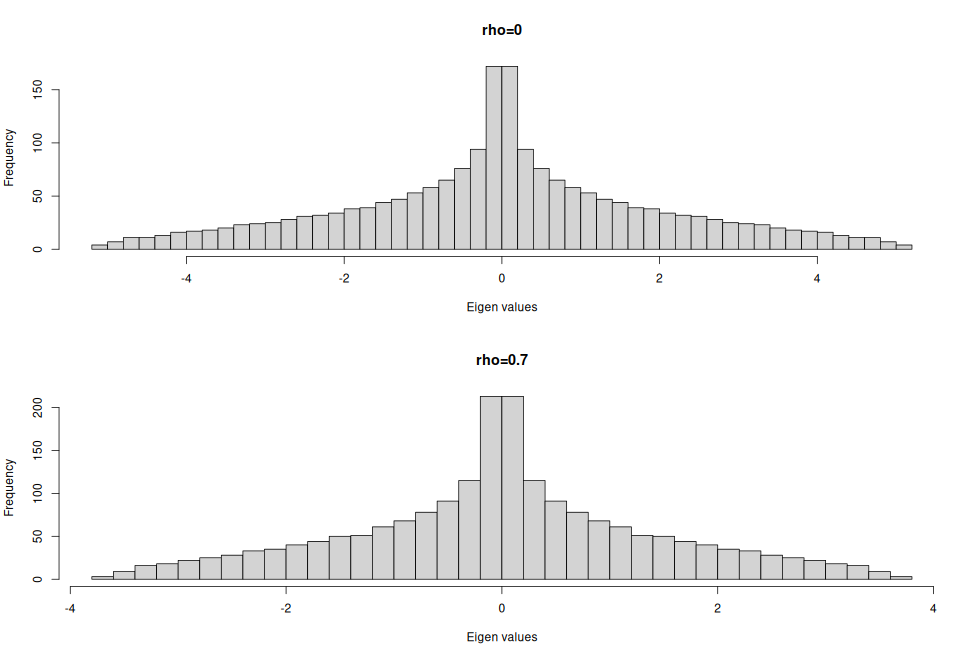}
    \caption{Shrinkage effect between $\rho=0$ (top) vs. $\rho=0.7$ (bottom)}
    \label{fig:fig1}
\end{figure}

\begin{figure}
    \centering
    \includegraphics[width=0.75\linewidth]{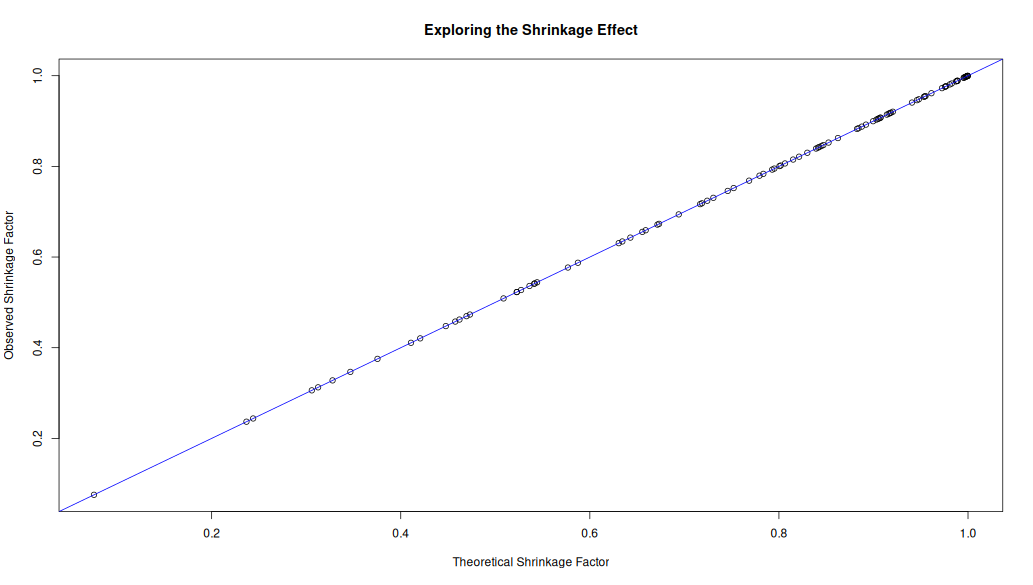}
    \caption{Observed vs. Theoretical Shrinkage factors}
    \label{fig:fig2}
\end{figure}

\subsection{\texorpdfstring{$\Sigma$}{} is unknown:}

We now propose a scheme to handle the case when $\Sigma$ is unknown, in which case $\Sigma$ is essentially a nuisance parameter
for inference on $\rho$. In this case, the
method proposed above needs to be modified. 
The key issue is obtaining the sampling distribution of $\lambda_0$ under $\rho=0$, since that involves the PSD of the true $\Sigma$ which is unknown. However, we can resolve this by first estimating the PSD $H_p$ of $\Sigma$ from the data. Specifically, from the sample covariance matrix of the observed data (either $X_1$ or $X_2$), we get an estimate of $H_p$  using El Karoui's discretization method (\cite{elkaroui}). We may also consider a pooled estimate obtained as the average of the estimates based on $X_1$ and $X_2$, separately. This method has been shown to be consistent in the $L^\infty$ norm sense. Since we have already shown continuity of the LSD with respect to the PSD (see Theorem \ref{ContinuityGeneral}), our method produces consistent estimates of the sampling distribution of $\lambda_0$. Using the sampling distribution of the estimated $\lambda_0$ under $H_0$, we can obtain an approximate $p$-value for the test which rejects the null for small values of $\lambda_{obs}$. The steps of the testing procedure are as described below.


\begin{enumerate}
    \item Let $X_1^{obs},X_2^{obs}$ be the two observed matrices of dimension $p \times n$ with $F_n$ denoting the ESD of the commutator matrix $S_n = \frac{1}{n}[X_1^{obs}, X_2^{obs}]$. Calculate $$T_{obs} = \sqrt{\int \lambda^2 dF_n(\lambda)}.$$ 
    \item From the sample covariance of $X_1^{obs}$ or $X_2^{obs}$, we derive an estimate $\hat{H}_p$ of the PSD of $\Sigma$ using El-Karoui's discretization algorithm.
    \item Let $B$ be a very large integer. For $b = 1 \text{ to } B$, we repeat the following operations.
\begin{enumerate}
    \item Construct a $p \times p$ p.s.d. matrix $\Sigma_b$ with eigen values distributed according to $\hat{H}_p$, where 
    the matrix of eigenvectors of $\Sigma_b$ are taken to be arbitrary orthogonal matrices. 
    
    \item Construct independent random matrices $Z_1^{(b)}$ and $Z_2^{(b)}$ of dimension $p \times n$ with i.i.d. standard Gaussian entries. 
    
    \item Generate $$S_0^{(b)} := \frac{1}{n}(\Sigma^{(b)})^\frac{1}{2}[Z_1^{(b)},Z_2^{(b)}](\Sigma^{(b)})^\frac{1}{2}.$$
    
    \item Let $F_0^{(b)}$ denote the ESD of $S_0^{(b)}$. Calculate $$T_b := \sqrt{\int \lambda^2 dF_0^{(b)}(\lambda)}.$$
\end{enumerate}
    \item Smaller values of $T_{obs}$ (with reference to the distribution of $T_b$) lead to rejection of the null hypothesis with a $p$-value ($PV$) which is derived from the sampling distribution of $T_b$ as follows:
    $$ PV = \frac{1}{B}\sum_{b=1}^B \mathbbm{1}_{\{T_b \leq T_{obs}\}}.$$
\end{enumerate}

\textbf{Numerical Results:} Similar to the known $\Sigma$ case, Table \ref{tab:3} ($p = 50; n = 500; B = 1000$) and Table \ref{tab:4} ($p = 100; n=500; B = 1000$) below lists the p-values obtained under these combinations.

\begin{table}[ht!]
    \centering
    \begin{tabular}{|l||c|c|c|c|c|c|}
    \hline
         \textbf{True Spectral Distribution} & $\boldsymbol{\rho}=\textbf{0.1}$ & $\boldsymbol{\rho}=\textbf{0.25}$ & $\boldsymbol{\rho}=\textbf{0.5}$ & $\boldsymbol{\rho}=\textbf{0.75}$ & $\boldsymbol{\rho}=\textbf{0.8}$ & $\boldsymbol{\rho}=\textbf{0.95}$\\
         \hline
         \hline
          $H_p = \delta_1$ & 0.366 & 0.053 & 0 & 0 & 0 & 0 \\
          \hline
          $H_p = 0.4 \delta_0 + 0.6 \delta_1$ & 0.439 & 0.354 & 0.113 & 0.006 & 0 & 0 \\
          \hline
          $H_p = 0.5 \delta_1 + 0.5 \delta_2$ & 0.316 & 0.239 & 0.041 & 0 & 0 & 0 \\
          \hline
          $H_p = 0.3 \delta_0 + 0.4 \delta_1 + 0.3 \delta_2$ & 0.377 & 0.323 & 0.127 & 0.015 & 0 & 0 \\
      \hline
    \end{tabular}
    \caption{p-values under different true spectral distributions and equi-correlation coefficients under the setting $p=50;n=500$}
    \label{tab:3}
\end{table}

\begin{table}[ht!]
    \centering
    \begin{tabular}{|l||c|c|c|c|c|c|}
    \hline
         \textbf{True Spectral Distribution} & $\boldsymbol{\rho}=\textbf{0.1}$ & $\boldsymbol{\rho}=\textbf{0.25}$ & $\boldsymbol{\rho}=\textbf{0.5}$ & $\boldsymbol{\rho}=\textbf{0.75}$ & $\boldsymbol{\rho}=\textbf{0.8}$ & $\boldsymbol{\rho}=\textbf{0.95}$\\
         \hline
         \hline
          $H_p = \delta_1$ & 0.273 & 0.001 & 0 & 0 & 0 & 0\\
          \hline
          $H_p = 0.4 \delta_0 + 0.6 \delta_1$ & 0.368 & 0.248 & 0.035 & 0 & 0 & 0 \\
          \hline
          $H_p = 0.5 \delta_1 + 0.5 \delta_2$ & 0.093 & 0.054 & 0.032 & 0 & 0 & 0 \\
          \hline
          $H_p = 0.3 \delta_0 + 0.4 \delta_1 + 0.3 \delta_2$ & 0.287 & 0.220 & 0.058 & 0.002 & 0.001 & 0 \\
      \hline
    \end{tabular}
    \caption{p-values under different true spectral distributions and equi-correlation coefficients under the setting $p=100;n=500$}
    \label{tab:4}
\end{table}

\begin{lstlisting}[language = R]

\end{lstlisting}

\color{black}
\textbf{Acknowledgments}: The authors would like to thank Professor Arup Bose for some insightful suggestions and discussions.

\bibliographystyle{plain}
\bibliography{bibli}

@article{BaiSilv95,
    author = {Jack W. Silverstein and Zhidong Bai},
    title = {On the Empirical Distribution of Eigenvalues of a Class of Large Dimensionsal Random Matrices},
    journal = {Journal of Multivariate Analysis},
    year = {1995}
}

@article{GeroHill03,
    author = {Jeffrey S. Geronimo and Theodore P. Hill},
    title = {Necessary and sufficient condition that the limit of {Stieltjes} transforms is a {Stieltjes} transform},
    journal = {Journal of Approximation Theory},
    volume = {121}, 
    issue = {1},
    pages = {54--60},
    year = {2003}
}

@article{PaulSilverstein2009,
    author = {Debashis Paul and Jack W. Silverstein},
    title = {No eigenvalues outside the support of the limiting empirical spectral
distribution of a separable covariance matrix},
    journal = {Journal of Multivariate Analysis},
    volume = {100},
    number = {1}, 
    pages = {37-57},
    year = {2009}
}

@article{SilvChoi,
    author = {Sang-Il choi and Jack W. Silverstein},
    title = {Analysis of the Limiting Spectral Distribution of Large dimensional Random Matrices},
    journal = {Journal of Multivariate Analysis},
    year = {1995},
    volume = {54},
    issue = {2}, 
    pages = {295-309}
}

@article{MarcenkoPastur1967,
    author = {Volodymyr Mar\v{c}enko and Leonid Pastur},
    title = {Distribution of eigenvalues for some sets of random matrices},
    journal = {Mathematics of the USSR Sbornik},
    volume = {1}, 
    pages = {457-483},
    year = {1967}
}

@article{Wigner1958,
    author = {Eugene Wigner},
    title = {On the distribution of the roots of certain symmetric matrices},
    journal = {Annals of Mathematics},
    year = {1958},
    volume = {67}, 
    pages = {325-328}
}

@phdthesis{Lixin06,
    author = {Lixin Zhang},
    title = {Spectral Analysis of Large Dimensional Random matrices},
    school = {National University of Singapore},
    year = {2006}
}

@book{BaiSilv09,
    author = {Zhidong Bai and Jack W. Silverstein},
    title = {Spectral Analysis of Large Dimensional Random Matrices},
    edition={2nd},
    publisher = {Springer},
    year = {2009}
}

@book{Couillet,
    author = {Romain Couillet and Zhenyu Liao},
    title = {Random Matrix Methods for Machine Learning},
    publisher = {Cambridge University Press},
    year = {2022}
}

@book{Dudley,
    author = {Richard M. Dudley},
    title = {Probabilities and Metrics; Convergence of Laws on Metric Spaces, with a View to Statistical Testing},
    publisher = {Aarhus universitet, Matematisk Institut},
    year = {1976}
}

@book{SteinShaka,
    author = {Elias M. Stein and Rami Shakarchi},
    title = {Complex Analysis},
    publisher = {Princeton University Press},
    year = {2003}
}

@book{Abramowitz,
    author = {Milton Abramowitz and Irene A. Stegun},
    title = {Handbook of Mathematical Functions},
    publisher = {National Bureau of Standards},
    year = {1964}
}

@article{CLTFMatrix1,
    author = {Shurong Zheng},
    title = {Central limit theorems for linear spectral statistics of large dimensional {F}-matrices},
    journal = {Annales de l'Institut Henri Poincar\'{e}},
    volume = {48}, 
    issue = {2}, 
    pages = {444--476},
    year = {2012}
}

@article{CLTFMatrix2,
    author = {Shurong Zheng and Zhidong Bai and Jianfeng Yao},
    title = {{CLT} for eigenvalue statistics of large-dimensional general {Fisher} matrices with applications},
    journal = {Bernoulli},
    volume = {},
    year = {2017}
}

@article{SignalNoise,
    author = {R. Brent Dozier and Jack W. Silverstein},
    title = {Analysis of the limiting spectral distribution of large dimensional information-plus-noise type matrices},
    journal = {Annals (of the Institut) Henri Poincaré},
    year = {2007}
}

@article{SepCovar,
    author = {Romain Couillet and Walid Hachem},
    title = {Analysis of the limiting spectral measure of large random matrices of the separable covariance type},
    journal = {Random Matrix Theory and Applications},
    volume = {},
    pages = {},
    year = {2014}
}

@article{varProfile,
    author = {Walid Hachem and Philippe Loubaton and Jamal Najim},
    title = {The empirical distribution of the eigenvalues of a Gram matrix
with a given variance profile},
    journal = {Annales de l'Institut Henri Poincar\'{e}},
    year = {2006}
}

@article{Staionary1,
    author = {Walid Hachem and Philippe Loubaton and Jamal Najim},
    title = {The Empirical Eigenvalue Distribution of a Gram Matrix: From Independence to Stationarity},
    journal = {Markov Processes and Related Fields},
    year = {2005}
}

@article{Staionary2,
    author = {Haoyang Liu and Alexander Aue and Debashis Paul},
    title = {On the {Mar\v{c}enko–Pastur} law for linear time series},
    journal = {Annals of Statistics},
    volume = {45}, 
    issue = {2}, 
    pages = {675--712},
    year = {2015}
}

@article{Staionary3,
    author = {Monika Bhattacharjee and Arup Bose},
    title = {Large Sample Behaviour of High Dimensional Autocovariance Matrices},
    journal = {Annals of Statistics},
    volume = {44}, 
    issue = {2}, 
    pages = {598--628},
    year = {2016}
}

@article{CrossCov,
    author = {Monika Bhattacharjee and Arup Bose and Apratim Dey},
    title = {Joint convergence of sample cross-covariance matrices},
    journal = {arXiv:2103.11946},  
    year = {2021}
}

@article{Tetilla,
    author = {Aur´elien Deya and Ivan Nourdin},
    title = {Convergence of Wigner integrals to the tetilla law},
    journal = {Latin American Journal of Probability, Mathematics and Statistics},  
    year = {2012}
}

@article{NicaSpeicher,
    author = {Alexandru Nica and Roland Speicher},
    title = {Commutators of free random variables},
    journal = {arXiv:funct-an/9612001},  
    year = {1996},
    doi = {https://doi.org/10.48550/arXiv.funct-an/9612001}
}

@article{Palheta,
    author = {Pedro H. S. Palheta, Marcelo R. Barbosa and Marcel Novaes},
    title = {Commutators of random matrices from the unitary
and orthogonal groups},
    journal = {Journal of Mathematical Physics},  
    year = {2022},
}

@article{Perales,
    title = {ON THE ANTI-COMMUTATOR OF TWO FREE RANDOM VARIABLES},
    author = {Daniel Perales},
    journal = {arXiv:2101.09444},
    year = {2021},
    doi = {https://doi.org/10.48550/arXiv.2101.09444},
}

@article{Vasilchuk,
    author = {V. Vasilchuk},
    title = {On the asymptotic distribution of the commutator and anticommutator of random matrices},
    journal = {Journal of Mathematical Physics},
    year = {2003}
}

@book{joelschiff,
    author = {Joel L. Schiff},
    title = {Normal Families},
    publisher = {Springer-Verlag, New York},
    year = {1993}
}

@book{awintner,
    author = {Aurel Wintner},
    title = {Spektraltheorie der Unendlichen Matrizen,},
    publisher = {Hirzel, Leipzig},
    year = {1929}
}

@article{elkaroui,
    author = {Noureddine El Karoui},
    title = {Spectrum Estimation For Large Dimensional Covariance Matrices Using Random Matrix Theory},
    journal = {The Annals Of Statistics},
    year = {2008}
}

@article{Risch,
    author = "Neil Risch and Heping Zhang",
    title = "Extreme Discordant Sib Pairs
for Mapping Quantitative Trait
Loci in Humans",
    journal = "Science",
    year = "1995",
    volume = "268", 
    issue = "5217",
    pages = "1584--1589"
}

@article {Howe,
	author = {Howe, Laurence J and Nivard, Michel G and Morris, Tim T and Hansen, Ailin F and Rasheed, Humaira and Cho, Yoonsu and Chittoor, Geetha and Lind, Penelope A and Palviainen, Teemu and van der Zee, Matthijs D and Cheesman, Rosa and Mangino, Massimo and Wang, Yunzhang and Li, Shuai and Klaric, Lucija and Ratliff, Scott M and Bielak, Lawrence F and Nygaard, Marianne and Reynolds, Chandra A and Balbona, Jared V and Bauer, Christopher R and Boomsma, Dorret I and Baras, Aris and Campbell, Archie and Campbell, Harry and Chen, Zhengming and Christofidou, Paraskevi and Dahm, Christina C and Dokuru, Deepika R and Evans, Luke M and de Geus, Eco JC and Giddaluru, Sudheer and Gordon, Scott D and Harden, K. Paige and Havdahl, Alexandra and Hill, W. David and Kerr, Shona M and Kim, Yongkang and Kweon, Hyeokmoon and Latvala, Antti and Li, Liming and Lin, Kuang and Martikainen, Pekka and Magnusson, Patrik KE and Mills, Melinda C and Lawlor, Deborah A and Overton, John D and Pedersen, Nancy L and Porteous, David J and Reid, Jeffrey and Silventoinen, Karri and Southey, Melissa C and Mallard, Travis T and Tucker-Drob, Elliot M and Wright, Margaret J and Social Science Genetic Association Consortium and Within Family Consortium and Hewitt, John K and Keller, Matthew C and Stallings, Michael C and Christensen, Kaare and Kardia, Sharon LR and Peyser, Patricia A and Smith, Jennifer A and Wilson, James F and Hopper, John L and H{\"a}gg, Sara and Spector, Tim D and Pingault, Jean-Baptiste and Plomin, Robert and Bartels, Meike and Martin, Nicholas G and Justice, Anne E and Millwood, Iona Y and Hveem, Kristian and Naess, {\O}yvind and Willer, Cristen J and {\r A}svold, Bj{\o}rn Olav and Koellinger, Philipp D and Kaprio, Jaakko and Medland, Sarah E and Walters, Robin G and Benjamin, Daniel J and Turley, Patrick and Evans, David M and Smith, George Davey and Hayward, Caroline and Brumpton, Ben and Hemani, Gibran and Davies, Neil M},
	title = {Within-sibship GWAS improve estimates of direct genetic effects},
	elocation-id = {2021.03.05.433935},
	year = {2021},
	doi = {10.1101/2021.03.05.433935},
	publisher = {Cold Spring Harbor Laboratory},
	URL = {https://www.biorxiv.org/content/early/2021/03/07/2021.03.05.433935},
	eprint = {https://www.biorxiv.org/content/early/2021/03/07/2021.03.05.433935.full.pdf},
	journal = {bioRxiv}
}

@article{Cardon2000,
    author = {Lon R. Cardon},
    title = {A Sib-Pair Regression Model of Linkage Disequilibrium for Quantitative Traits
},
    journal = {Human Heredity},
    year = {2000},
    volume  = {50},
    issue = {6},
    pages = {350-358}
}
\newpage
\appendix

\section{A few general results}

\subsection{A few basic results related to matrices }\label{R123}
\begin{itemize}
    \item[$\boldsymbol{R}_0$:] \textbf{Resolvent identity}: For invertible matrices $A,B$ of same dimension, we have
    \begin{align}\label{R0}
        A^{-1} - B^{-1} = A^{-1}(B - A)B^{-1} = B^{-1}(B - A)A^{-1}.
    \end{align}
    \item[$\boldsymbol{R}_1$:] For skew-Hermitian matrices $A, B \in \mathbb{C}^{p \times p}$, we have
    \begin{align}\label{R1}
     ||F^A - F^B||_{im} = ||F^{-\mathbbm{i}A} - F^{-\mathbbm{i}B}|| \leq \frac{1}{p}\operatorname{rank}(A - B).
    \end{align}
    The first equality follows from (\ref{ESD_sksym_equal_sym}) and the last inequality follows from Lemma 2.4 of \cite{BaiSilv95}.
    \item[$\boldsymbol{R}_2$:]  From Lemma 2.1 of \cite{BaiSilv95}, for a rectangular matrix, we have 
    \begin{align}\label{R2}
      \operatorname{rank}(A) \leq \sum_{i,j} \mathbbm{1}_{\{a_{ij} \neq 0\}}.  
    \end{align}
    \item[$\boldsymbol{R}_3$:] 
    For rectangular matrices $A,B,P,Q, X$ of compatible dimensions, we have 
    \begin{align}\label{R3}
        \operatorname{rank}(AXB-PXQ) \leq \operatorname{rank}(A-P)+\operatorname{rank}(B-Q).
    \end{align}
    \item[$\boldsymbol{R}_4$:] Cauchy-Schwarz Inequality: 
    \begin{align}\label{R4}
      |a^*Xb| \leq ||X||_{op}\,||a||\,||b||.  
    \end{align}
    \item[$\boldsymbol{R}_5$:] For a p.s.d. matrix B and any square matrix A, we have 
    \begin{align}\label{R5}
        |\operatorname{trace}(AB)| \leq ||A||_{op} \operatorname{trace}(B).
    \end{align}
    \item[$\boldsymbol{R}_6$:] For $N \times N$ matrices $A, B$, we have
    \begin{align}\label{R6}
        |\operatorname{trace}(AB)| \leq N||A||_{op}\,||B||_{op}.
    \end{align}
\end{itemize}

\begin{lemma}\label{Levy_Stieltjes}
    Let $\{F_n, G_n\}_{n=1}^\infty$ be sequences of distribution functions on $\mathbbm{i}\mathbb{R}$ with $s_{F_n}(z), s_{G_n}(z)$ denoting their respective Stieltjes transforms at $z \in \mathbb{C}_L$. If $L_{im}(F_n, G_n) \rightarrow 0$, then $|s_{F_n}(z) - s_{G_n}(z)| \rightarrow 0$.
\end{lemma}
\begin{proof}
As usual, for a distribution function $F$ on $\mathbbm{i}\mathbb{R}$, we denote its real counterpart as $\overline{F}$. Let $\mathcal{P}(\mathbb{R})$ represent the set of all probability distribution functions on $\mathbb{R}$. Then the bounded Lipschitz metric is defined as follows:
$$\beta: \mathcal{P}(\mathbb{R}) \times \mathcal{P}(\mathbb{R}) \rightarrow \mathbb{R}_{+} \text{, where } \beta(\overline{F}, \overline{G}) := \underset{}{\sup} \bigg\{\bigg|\int hd\overline{F} - \int hd\overline{G}\bigg|: ||h||_{BL} \leq 1\bigg\},$$

$$\text{and, }||h||_{BL} = \sup\{|h(x)|: x \in \mathbb{R}\} + \underset{x \neq y}{\sup}\dfrac{|h(x) - h(y)|}{|x - y|}.$$

From Corollary 18.4 and Theorem 8.3 of \cite{Dudley}, we have the following relationship between Levy (L) and bounded Lipschitz ($\beta$) metrics:
\begin{align}\label{Levy_B}
\frac{1}{2}\beta(\overline{F}, \overline{G}) \leq L(\overline{F}, \overline{G}) \leq 3\sqrt{\beta(\overline{F}, \overline{G})}.
\end{align}

Fix $z \in \mathbb{C}_L$ arbitrarily. Define $g_z(x) := (\mathbbm{i}x - z)^{-1}; x\in \mathbb{R}$. Note that, $|g_z(x)| \leq 1/|\Re(z)|$ for all $x \in \mathbb{R}$. Therefore, 
$$|g_z(x_1) - g_z(x_2)| = \bigg|\frac{1}{\mathbbm{i}x_1 - z} - \frac{1}{\mathbbm{i}x_2 - z}\bigg| = \frac{|x_1 - x_2|}{|\mathbbm{i}x_1 - z||\mathbbm{i}x_2-z|} \leq \frac{|x_1 - x_2|}{\Re^2(z)}.$$

Note that $||g_z||_{BL} \leq {1}/{|\Re(z)|} + {1}/{\Re ^2(z)} < \infty$. Then for $g := {g_z}/{||g_z||_{BL}}$, we have $||g||_{BL} = 1$.

By (\ref{Levy_B}) and (\ref{Levy_vs_uniform}) and using $dF_n(\mathbbm{i}x) = d\overline{F}_n(x)$ for $x \in \mathbb{R}$, we have:
\begin{align*}
L_{im}(F_n, G_n) \rightarrow 0
\iff &  L(\overline{F}_n, \overline{G}_n) \rightarrow 0 \\
\iff  & \beta(\overline{F}_n, \overline{G}_n) \rightarrow 0\\
\implies & \bigg|\int_\mathbb{R}g(x)d\overline{F}_n(x) - \int_\mathbb{R}g(x)d\overline{G}_n(x)\bigg| \rightarrow 0\\
\implies & \bigg|\int_\mathbb{R}\frac{1}{\mathbbm{i}x - z}dF_n(\mathbbm{i}x) - \int_\mathbb{R}\frac{1}{\mathbbm{i}x-z}dG_n(\mathbbm{i}x)\bigg| \rightarrow 0\\
\implies & |s_{F_n}(z) - s_{G_n}(z)| \longrightarrow 0.
\end{align*}
\end{proof}

\begin{lemma}\label{lA.2}
    Let $\{X_{jn}, Y_{jn}: 1 \leq j \leq n\}_{n=1}^\infty$ be triangular arrays of random variables. Suppose we have $\underset{1 \leq j \leq n}{\max}|X_{jn}| \xrightarrow{a.s.} 0$ and $\underset{1 \leq j \leq n}{\max}|Y_{jn}| \xrightarrow{a.s.} 0$. Then $\underset{1 \leq j \leq n}{\max}|X_{jn} + Y_{jn}| \xrightarrow{a.s.} 0$.
\end{lemma}

\begin{proof}
Let $A_x := \{\omega: \underset{n \rightarrow \infty}{\lim}\underset{1 \leq j \leq n}{\max}|X_{jn}(\omega)| = 0\}$, $A_y := \{\omega: \underset{n \rightarrow \infty}{\lim}\underset{1 \leq j \leq n}{\max}|Y_{jn}(\omega)| = 0\}$. Then $\mathbb{P}(A_x) = 1 = \mathbb{P}(A_y)$. Then for all $ \omega \in A_x \cap A_y$, we have $0 \leq |X_{jn}(\omega) + Y_{jn}(\omega)| \leq |X_{jn}(\omega)| + |Y_{jn}(\omega)|$. Hence, 
$\underset{n \rightarrow \infty}{\lim}\underset{1 \leq j \leq n}{\max}|X_{jn}(\omega) + Y_{jn}(\omega)| = 0$. But, $\mathbb{P}(A_x \cap A_y) = 1$. Therefore, the result follows.
\end{proof}

\begin{lemma}\label{lA.3}
     Let $\{A_{jn}, B_{jn}, C_{jn}, D_{jn}: 1 \leq j \leq n\}_{n=1}^\infty$ be triangular arrays of random variables. Suppose $\underset{1 \leq j \leq n}{\max}|A_{jn} - C_{jn}| \xrightarrow{a.s.} 0$ and $\underset{1 \leq j \leq n}{\max}|B_{jn} - D_{jn}| \xrightarrow{a.s.} 0$ and there exists $N_0 \in \mathbb{N}$ such that $|C_{jn}| \leq B_1$ a.s. and $|D_{jn}| \leq B_2$ a.s. when $n > N_0$ for some $B_1, B_2 \geq 0$. Then $\underset{1 \leq j \leq n}{\max}|A_{jn}B_{jn} - C_{jn}D_{jn}| \xrightarrow{a.s.} 0.$
\end{lemma}

\begin{proof}

Let $\Omega_1 = \{\omega: \underset{n \rightarrow \infty}{\lim}\underset{1 \leq j \leq n}{\max}|A_{jn}(\omega) - C_{jn}(\omega)|= 0\}$, $\Omega_2 = \{\omega: \underset{n \rightarrow \infty}{\lim}\underset{1 \leq j \leq n}{\max}|B_{jn}(\omega) - D_{jn}(\omega)|= 0\}$, $\Omega_3 = \{\omega: |C_{jn}(\omega)| \leq B_1 \text{ for } n > N_0\}$ and $\Omega_4 = \{\omega: |D_{jn}(\omega)| \leq B_2 \text{ for } n > N_0\}$. Then $\Omega_0 = \cap_{j=1}^4\Omega_j$ is a set of probability $1$. Then for all $\omega \in \Omega_0$, $\underset{1 \leq j \leq n}{\max} |B_{jn}(\omega)| \leq B_2$ for large $n$. Therefore, for  $\omega \in \Omega_0$ and large $n$, we get the following that concludes the proof.
\begin{align*}
\underset{1 \leq j \leq n}{\max}|A_{jn}B_{jn} - C_{jn}D_{jn}| 
&\leq \underset{1 \leq j \leq n}{\max}|A_{jn} - C_{jn}||B_{jn}| + \underset{1 \leq j \leq n}{\max}|C_{jn}||B_{jn} - D_{jn}|\\
&\leq B_2 \underset{1 \leq j \leq n}{\max}|A_{jn} - C_{jn}| + B_1 \underset{1 \leq j \leq n}{\max}|B_{jn} - D_{jn}| \xrightarrow{a.s.} 0.
\end{align*}
\end{proof}

\begin{lemma}\label{lA.4}
     Let $\{X_{jn}, Y_{jn}: 1 \leq j \leq n\}_{n=1}^\infty$ be triangular arrays of random variables such that $\underset{1 \leq j \leq n}{\max}|X_{jn} - Y_{jn}| \xrightarrow{a.s.} 0$. Then $|\frac{1}{n}\sum_{j=1}^n (X_{jn} - Y_{jn})| \xrightarrow{a.s.} 0$.
\end{lemma}

\begin{proof}
    Let $M_n := \underset{1 \leq j \leq n}{\max}|X_{jn} - Y_{jn}|$. We have $|\frac{1}{n}\sum_{j=1}^n (X_{jn} - Y_{jn})| \leq \frac{1}{n}\sum_{j=1}^n |X_{jn} - Y_{jn}| \leq M_n$. Let $\epsilon > 0$ be arbitrary. Then there exists $\Omega_0 \subset \Omega$ such that $\mathbb{P}(\Omega_0) = 1$ and for all $\omega\in \Omega_0$, we have $M_n(\omega) < \epsilon$ for sufficiently large $n \in \mathbb{N}$. Then, $\mathbb{P}(\{\omega: |\frac{1}{n}\sum_{j=1}^n (X_{jn} - Y_{jn})| < \epsilon\}) = 1$. Since $\epsilon > 0$ is arbitrary, the result follows.
\end{proof}

We state the following result (Lemma B.26 of \cite{BaiSilv09}) without proof.
\begin{lemma}\label{lA.5}
Let $A = (a_{ij})$ be an $n \times n$ non-random matrix and $x = (x_1,\ldots, x_n)^T$ be a vector of independent entries. Suppose $\mathbb{E}x_i = 0, \mathbb{E}|x_i|^2 = 1$, and $\mathbb{E}|x_i|^l \leq \nu_l$. Then for $k \geq 1$, $\exists$ $C_k > 0$ independent of $n$ such that 
\begin{align*}
  \mathbb{E}|x^*Ax - \operatorname{trace}(A)|^k \leq C_k\bigg((\nu_4\operatorname{trace}(AA^*))^\frac{k}{2} + \nu_{2k}\operatorname{trace}\{(AA^*)^\frac{k}{2}\}\bigg).
\end{align*}
\end{lemma}

\textbf{Simplification}: For deterministic matrix $A$ with $||A||_{op} < \infty$, let $B = \frac{A}{||A||_{op}}$. Then, $||B||_{op} = 1$ and by (\ref{R6}), we have $\operatorname{trace}(BB^*) \leq n||B||_{op}^2 = n$ and $\operatorname{trace}\{(BB^*)^\frac{k}{2}\} \leq n||B||_{op}^k = n$. By Lemma \ref{lA.5}, we have

\begin{align}\label{I1}
    &\mathbb{E}|x^*Bx - \operatorname{trace}(B)|^k 
    \leq C_k\bigg((\nu_4\operatorname{trace}(BB^*))^\frac{k}{2} + \nu_{2k}\operatorname{trace}\{(BB^*)^\frac{k}{2}\}\bigg)\\ \notag
    \implies & \frac{\mathbb{E}|x^*Ax - \operatorname{trace}(A)|^k }{||A||_{op}^k} \leq C_k[(n\nu_4)^\frac{k}{2} + n\nu_{2k}]\\ \notag
    \implies & \mathbb{E}|x^*Ax - \operatorname{trace}(A)|^k  \leq C_k||A||_{op}^k[(n\nu_4)^\frac{k}{2} + n\nu_{2k}].
\end{align}

We will be using this form of the inequality going forward.
\begin{lemma}\label{quadraticForm}
     Let $\{x_{jn}: 1 \leq j \leq n\}_{n =1}^\infty$ be a triangular array of complex valued random vectors in $\mathbb{C}^p$ with independent entries. For $1 \leq r \leq n$, denote the $r^{th}$ element of $x_{jn}$ as $x_{jn}^{(r)}$. Suppose $\mathbb{E}x_{jn}^{(r)} = 0, \mathbb{E}|x_{jn}^{(r)}|^2 = 1$ and for $k \geq 1$ and  $|x_{jn}| \leq n^b$ for some $0<b < \frac{1}{2}$. Suppose $A_j \in \mathbb{C}^{p \times p}$ is independent of $x_{jn}$ and $||A_j||_{op} \leq B$ a.s. for some $B > 0$. Then, $$\underset{1\leq j \leq n}{\max} \absmod{\dfrac{1}{n}x_{jn}^*Ax_{jn} - \dfrac{1}{n}\operatorname{trace}(A_j)} \xrightarrow{a.s} 0.$$
\end{lemma}
\begin{proof}
   Note that \begin{enumerate}
       \item $\nu_4 := \underset{j;n}{\sup} \mathbb{E}|x_{jn}|^4 \leq \sup n^{2b}\mathbb{E}|x_{jn}|^2 = n^{2b}$.
       \item In general, when $k \geq 2$, we similarly deduce that $\nu_{2k}= \underset{j;n}{\sup} \mathbb{E}|x_{jn}|^{2k} \leq n^{2b(k-1)}$.
   \end{enumerate}
For arbitrary $\delta > 0$ and $k \geq 1$, we have
\begin{align*}
p_n:= &\mathbb{P}\bigg(\underset{1\leq j \leq n}{\max} \left|\dfrac{1}{n}x_{jn}^*A_jx_{jn} - \dfrac{1}{n}\operatorname{trace}(A_j)\right| > \delta\bigg)\\
 &\leq \sum_{j=1}^{n} \mathbb{P}\bigg(\left|\dfrac{1}{n}x_{jn}^*A_jx_{jn} - \dfrac{1}{n}\operatorname{trace}(A_j)\right| > \delta\bigg) \text{, by union bound}\\
&\leq \sum_{j=1}^{n} \dfrac{\mathbb{E}\left|\dfrac{1}{n}x_{jn}^*A_jx_{jn} - \dfrac{1}{n}\operatorname{trace}(A_j)\right|^{k}}{\delta^{k}} \text{, for any } k \in \mathbb{N}\\
&= \sum_{j=1}^{n} \dfrac{\mathbb{E} \bigg(\mathbb{E} \bigg[|\dfrac{1}{n}x_{jn}^*A_jx_{jn} - \dfrac{1}{n}\operatorname{trace}(A_j)|^k\bigg|A_j\bigg]\bigg)}{\delta^{k}}\\
&\leq \sum_{j=1}^{n} \dfrac{\mathbb{E}||A_j||_{op}^{k}C_{k}((n\nu_4)^\frac{k}{2} + n\nu_{2k})}{n^{k}\delta^{k}}  \text{ by } (\ref{I1})\\
&\leq \sum_{j=1}^{n} \dfrac{K[(n^{1+2b})^\frac{k}{2} + n^{1+2b(k-1)}]}{n^{k}} \text{, where } K = C_k\bigg(\frac{B}{\delta}\bigg)^k,\\
&= \frac{K}{n^{k(\frac{1}{2}-b) - 1}} + \frac{K}{n^{k(1 - 2b) + 2b-2}}.
\end{align*}
Since $b < 0.5$ and the above inequality holds for arbitrary $k \in \mathbb{N}$, we can choose $k \in \mathbb{N}$ large enough so that $\min\{k(0.5-b) - 1, k(1-2b)+2b-2\} > 1$ to ensure that $\sum_{n=1}^\infty p_n$ converges. Therefore, by Borel Cantelli lemma, we have the result.
\end{proof}

\begin{corollary}\label{quadraticFormCorollary}
     Let $\{x_{jn}: 1 \leq j \leq n\}_{n =1}^\infty$ be as in Lemma \ref{quadraticForm}. Suppose all conditions in the lemma are satisfied, except $\mathbb{E}|x_{jn}^{(r)}|^2 = 1$. However, $\mathbb{E}|x_{jn}^{(r)}|^2$ converge uniformly to $1$. Then, $$\underset{1\leq j \leq n}{\max} \absmod{\dfrac{1}{n}x_{jn}^*A_jx_{jn} - \dfrac{1}{n}\operatorname{trace}(A_j)} \xrightarrow{a.s} 0.$$
\end{corollary}
\begin{proof}
    Let $a_{ii}^{(j)}$ be the $i^{th}$ diagonal element of $A_j$ and $\sigma_{j,i}^2 := \mathbb{E}|x_{jn}^{(i)}|^2$. Then, $\mathbb{E}x_{jn}^*A_jx_{jn} = \sum_{i=1}^p\sigma_{i,j}^2a_{ii}^{(j)}$.
    
    The proof follows upon observing that
\begin{align*}
    \frac{1}{n}\bigg|x_{jn}^*A_jx_{jn} - \operatorname{trace}(A_j)\bigg| \leq \frac{1}{n}\bigg|x_{jn}^*A_jx_{jn} - \sum_{i=1}^p\sigma_{i,j}^2a_{ii}^{(j)}\bigg| + \frac{1}{n}\bigg|\sum_{i=1}^p \sigma_{i,j}^2a_{ii}^{(j)} -\operatorname{trace}(A_j)\bigg|.
\end{align*}
The first term converges to 0 almost surely by Lemma \ref{quadraticForm}. The second term goes to $0$ deterministically due to the uniform convergence of $\sigma_{i,j}^2$ to 1.
\end{proof}

\color{black}
\begin{corollary}\label{quadraticForm_xy}
Let $\{u_{jn},v_{jn}:1 \leq j \leq n\}_{n=1}^\infty$ be triangular arrays and $A_j$ be complex matrices as in Lemma \ref{quadraticForm}/ Corollary \ref{quadraticFormCorollary} with $u_{jn}$ and $v_{jn}$ independent of each other. Then,  
$$\underset{1 \leq j \leq n}{\max}\absmod{\frac{1}{n}u_{jn}^*A_jv_{jn}}\xrightarrow{ a.s.} 0.$$
\end{corollary}
\begin{proof}
    Let $Q_j(u,v) := \frac{1}{n}u_{jn}^*A_jv_{jn}$. Define $Q_j(v,v), Q_j(u,u), Q_j(v,u)$ similarly. Let $x_{jn} = \frac{1}{\sqrt{2}}(u_{jn} + v_{jn})$. Now applying Lemma \ref{quadraticForm}/ Corollary \ref{quadraticFormCorollary}, we get
\begin{align}\label{sum}
    &\underset{1 \leq j \leq n}{\max}\absmod{\frac{1}{n}x_{jn}^*A_jx_{jn} - \frac{1}{n}\operatorname{trace}(A_j)} \xrightarrow{a.s.} 0\\ \notag
    \implies& \underset{1 \leq j \leq n}{\max}\absmod{\frac{1}{2}(Q_j(u,u) - T_j) + \frac{1}{2}(Q_j(v,v) - T_j) + \frac{1}{2}(Q_j(u,v) + Q_j(v,u))} \xrightarrow{a.s.} 0,
\end{align}
where $T_j := \frac{1}{n}\operatorname{trace}(A_j)$. Now setting $x_{jn} = \frac{1}{\sqrt{2}}(u_{jn} + \mathbbm{i}v_{jn})$ and applying Lemma \ref{quadraticForm}/ Corollary \ref{quadraticFormCorollary}, we get
\begin{align}\label{antisum}
    &\underset{1 \leq j \leq n}{\max}\absmod{\frac{1}{n}x_{jn}^*A_jx_{jn} - \frac{1}{n}\operatorname{trace}(A_j)} \xrightarrow{a.s.} 0\\\notag
    \implies& \underset{1 \leq j \leq n}{\max}\absmod{\frac{1}{2}(Q_j(u,u) - T_j) + \frac{1}{2}(Q_j(v,v) - T_j) + \frac{1}{2}\mathbbm{i}(Q_j(u,v) - Q_j(v,u))} \xrightarrow{a.s.} 0.
\end{align}
Using Lemma \ref{lA.2} on (\ref{sum}) and (\ref{antisum}), we get $\underset{1 \leq j \leq n}{\max}|Q_j(u,v)| \xrightarrow{a.s.} 0$.
\end{proof}

\begin{lemma} \label{Rank2Woodbury}
\textbf{2-rank perturbation equality:}
Let $B \in \mathbb{C}^{p \times p}$ be of the form $B = A - zI_p$ for some skew-Hermitian matrix $A$ and $z \in \mathbb{C}_L$. For vectors $u, v \in \mathbb{C}^p$, define $\langle u, v\rangle := u^*B^{-1}v$. Then,
\begin{description}
    \item[1] $(B+ uv^{*} - vu^{*})^{-1}u = B^{-1}(\alpha_1 u + \beta_1 v)$; $\alpha_1 = (1 - \langle u,v\rangle)D(u,v)$; 
     $\beta_1 = \langle u,u\rangle D(u,v),$
    \item[2] $(B+ uv^{*} - vu^{*})^{-1}v = B^{-1}(\alpha_2 v + \beta_2 u)$; $\alpha_2 = (1 + \langle v,u\rangle) D(u,v)$; $\beta_2 = -\langle v, v\rangle D(u,v),$
\end{description}
 where $D(u,v) =  \bigg((1 - \langle u,v\rangle)(1 + \langle v,u\rangle) +  \langle u,u\rangle\langle v,v\rangle\bigg)^{-1}.$
\end{lemma}
\begin{proof}
Clearly, $B$ cannot have zero as eigenvalue. So $\langle u,v \rangle$ is well-defined. For $P \in \mathbb{C}^{p\times p}, Q,R \in  \mathbb{C}^{p\times n}$ with $P + QR^*$ and $P$ being invertible, we use Woodbury's formula to get the following:
\begin{align}\label{ShermanMorrison}
  (P + QR^*)^{-1} &= P^{-1} - P^{-1}Q(I_n + R^*P^{-1}Q)^{-1}R^*P^{-1}\\ \notag
  \implies (P + QR^*)^{-1}Q &= P^{-1}Q - P^{-1}Q(I_n + R^*P^{-1}Q)^{-1}R^*P^{-1}Q\\ \notag
  &= P^{-1}Q\bigg(I_n - (I_n + R^*P^{-1}Q)^{-1}R^*P^{-1}Q\bigg).
\end{align}

Let $P = B$, $Q=[u:v]$ and $R=[v: -u]$. Note that $\operatorname{det}(I_2 + R^*P^{-1}Q)^{-1} = D(u,v)$. So, $D(u,v)$ is well-defined. Finally, observing that $B + uv^*-vu^* = P + QR^*$, we use (\ref{ShermanMorrison}) to get:
\begin{align*}
    &(B + uv^*-vu^*)^{-1}[u:v]\\
    =& B^{-1}[u:v]\bigg(I_2 - (I_2 + R^*P^{-1}Q)^{-1}R^*P^{-1}Q\bigg)\\
    =& B^{-1}[u:v] \bigg(I_2 -{\begin{bmatrix}
        1+\langle v,u \rangle & \langle v,v \rangle\\
        -\langle u,u \rangle & 1-\langle u,v \rangle
    \end{bmatrix}}^{-1}\begin{bmatrix}
        \langle v,u \rangle &  \langle v,v \rangle\\
        -\langle u,u \rangle &  -\langle u,v \rangle\\        
    \end{bmatrix}\bigg)\\
    =&  B^{-1}[u:v] \bigg(I_2 - D(u,v)\begin{bmatrix}
        1-\langle u,v \rangle & -\langle v,v \rangle\\
        \langle u,u \rangle & 1+\langle v,u \rangle
    \end{bmatrix}\begin{bmatrix}
        \langle v,u \rangle &  \langle v,v \rangle\\
        -\langle u,u \rangle &  -\langle u,v \rangle\\        
    \end{bmatrix}\bigg)\\
    =& D(u,v)B^{-1}[u:v] \begin{bmatrix}
        \alpha_1 & \beta_2\\
        \beta_1 & \alpha_2
    \end{bmatrix}.
\end{align*}
\end{proof}

\section{Intermediate Results}
\subsection{Results related to proof of uniqueness in (\ref{h_main_eqn})}
\begin{definition}\label{defining_Irs}
For $\textbf{h} = (h_1, h_2) \in \mathbb{C}_R^2$ and $r,s \in \mathbb{N}\cup \{0\}$, we define:
$$I_{r,s}(\textbf{h}(z),H) := \displaystyle \int \frac{\lambda_1^r\lambda_2^sdH(\boldsymbol{\lambda})}{|-z + \boldsymbol{\lambda}^T\boldsymbol{\rho}(c\textbf{h})|^2}.$$    
\end{definition}
Fix $z = -u+\mathbbm{i}v \in \mathbb{C}_L$ with $u >0$. Suppose $\textbf{h} = (h_1, h_2)$ satisfy (\ref{h_main_eqn}). With the above definition, we observe that,
\begin{align}
    \Re(h_1) &= \int \frac{\lambda_1\Re(\overline{-z+\boldsymbol{\lambda}^T\boldsymbol{\rho}(c\textbf{h}))}}{|-z+\boldsymbol{\lambda}^T\boldsymbol{\rho}(c\textbf{h})|^2}dH(\boldsymbol{\lambda})\notag\\ 
    &= u\int \frac{\lambda_1dH(\boldsymbol{\lambda})}{|-z+\boldsymbol{\lambda}^T\boldsymbol{\rho}(c\textbf{h})|^2} + \Re(\rho_1(c\textbf{h}))\int \frac{\lambda_1^2dH(\boldsymbol{\lambda})}{|-z+\boldsymbol{\lambda}^T\boldsymbol{\rho}(c\textbf{h})|^2} + \Re(\rho_2(c\textbf{h}))\int \frac{\lambda_1\lambda_2dH(\boldsymbol{\lambda})}{|-z+\boldsymbol{\lambda}^T\boldsymbol{\rho}(c\textbf{h})|^2}\notag\\
    &= uI_{1,0}(\textbf{h}, H) + \Re(\rho_1(c\textbf{h}))I_{2,0}(\textbf{h}, H)+ \Re(\rho_2(c\textbf{h}))I_{1,1}(\textbf{h}, H).\label{real_of_h1}
\end{align}

Similarly, we get
\begin{align}
    \Re(h_2) &= uI_{0,1}(\textbf{h},H) + \Re(\rho_1(ch))I_{1,1}(\textbf{h},H) + \Re(\rho_2(ch))I_{0,2}(\textbf{h},H),\\ \notag
    \Im(h_1) &= uI_{1,0}(\textbf{h},H) - \Im(\rho_1(ch))I_{2,0}(\textbf{h},H) - \Im(\rho_2(ch))I_{1,1}(\textbf{h},H), \text{ and}\\ \notag
    \Im(h_2) &= uI_{0,1}(\textbf{h},H) + \Im(\rho_1(ch))I_{1,1}(\textbf{h},H) - \Im(\rho_2(ch))I_{0,2}(\textbf{h},H).
\end{align}

\begin{lemma}\label{Lipschitz}
    \textbf{(Lipschitz within an isosceles sector)}: Recall the definition of $\mathcal{S}(b)$ from (\ref{defining_sector}). For $0 < b$, the functions $\rho_1, \rho_2$ are Lipschitz continuous on $\mathcal{S}(b)^2 = \mathcal{S}(b) \times \mathcal{S}(b)$.
\end{lemma}
\begin{proof}
Let $\textbf{h}=(h_1, h_2), \textbf{g} = (g_1, g_2) \in \mathcal{S}(b)^2$. First, we establish a bound for $|1 + h_1h_2|^{-1}$ and $|1 + g_1g_2|^{-1}$. Clearly $\Re(h_1h_2) \geq 0$ and therefore,
\begin{align}
    \frac{1}{|1+ h_1h_2|} = \frac{1}{\sqrt{(1+\Re(h_1h_2))^2 + \Im^2(h_1h_2)}} \leq 1.
\end{align}
The same bound works for $|1 + g_1g_2|^{-1}$ as well. We have,
\begin{align}
|\rho_1(\textbf{h})| = |\rho_1(h_1, h_2)| = \absmod{\frac{h_2}{1+h_1h_2}} \leq b.
\end{align}
Therefore, we observe that
    \begin{align*}
        \absmod{\rho_1(\textbf{h}) - \rho_1(\textbf{g})}
        &= \absmod{\frac{h_2}{1+h_{1}h_{2}} - \frac{g_2}{1+g_{1}g_{2}}}\\
        &= \absmod{\frac{(h_2 - g_2) + h_2g_2(g_1-h_1)}{(1+h_1h_2)(1+g_1g_2)}}\\
        &\leq \absmod{\frac{h_2 - g_2}{(1+h_1h_2)(1+g_1g_2)}} + |\rho_1(\textbf{h})||\rho_1(\textbf{g})|g_1-h_1|\\
        &\leq |h_2-g_2| + b^2|h_1-g_1|\\
        &\leq K_0||\textbf{h}-\textbf{g}||_1 \text{, where } K_0:= \max\{1, b^2\}.
    \end{align*}
\end{proof}

\subsubsection{\textbf{Proof of Lemma \ref{location_of_solution_General}}}\label{sec:ProofOflocation_of_solution_General}

\begin{proof}
Note that $\boldsymbol{\rho}(c\textbf{h}) \in \mathbb{C}_R^2$ follows from Remark \ref{remark_real_of_rho}. Therefore, for $\boldsymbol{\lambda} \in \mathbb{R}_+^2$, we have
\begin{align}
    |-z + \boldsymbol{\lambda}^T\boldsymbol{\rho}(c\textbf{h})| \geq |u + \lambda_1\Re(\rho_1(c\textbf{h})) + \lambda_2\Re(\rho_2(c\textbf{h}))| \geq u.
\end{align}

Therefore, using (\ref{defining_C0}), we have
\begin{align}
    |h_k(z)| \leq \int \absmod{\frac{\lambda_k}{-z + \boldsymbol{\lambda}^T\boldsymbol{\rho}(c\textbf{h})}}dH(\boldsymbol{\lambda}) \leq \int \frac{\lambda_kdH(\boldsymbol{\lambda})}{u} \leq \frac{C_0}{u}.
\end{align}

 For arbitrary $\epsilon > 0$, there exist $\delta(\epsilon)>0$ such that $|\theta_1|,|\theta_2| < \delta(\epsilon) \implies |\rho_k(\theta_1, \theta_2)| < \epsilon$. Without loss of any generality, we can choose $\delta(\epsilon) < 1$. By choosing $u > cC_0/\delta(\epsilon)$, we can ensure that $|ch_k(z)| < \delta(\epsilon)$. Then for such $z$ and $k=1,2$, we have 
\begin{align}\label{rho_bound}
   |\rho_k(c\textbf{h})|= |\rho_k(ch_1, ch_2)| < \epsilon.
\end{align}

Now by (\ref{real_of_h1}), we have
\begin{align}
  \Re(h_1(z)) &=\, uI_{1,0}(\textbf{h}, H) + \Re(\rho_1(c\textbf{h}))I_{2,0}(\textbf{h}, H)+ \Re(\rho_2(c\textbf{h}))I_{1,1}(\textbf{h}, H) \label{real_h1}, \text{ and}\\
  \Im(h_1(z)) &=\, vI_{1,0}(\textbf{h}, H) - \Im(\rho_1(c\textbf{h}))I_{2,0}(\textbf{h}, H)- \Im(\rho_2(c\textbf{h}))I_{1,1}(\textbf{h}, H). \label{im_h1}
\end{align}

Now,  note that $I_{1,1}(\textbf{h},H) \leq \sqrt{I_{2,0}(\textbf{h},H)I_{0,2}(\textbf{h},H)}$ is immediate from the Cauchy-Schwarz inequality. Using $\boldsymbol{T}_5$ of Theorem \ref{mainTheorem}, we observe that \begin{align}
    I_{2,0}(\textbf{h},H) = \int \frac{\lambda_1^2dH(\boldsymbol{\lambda})}{|-z + \boldsymbol{\lambda}^T\boldsymbol{\rho(c\textbf{h})}|^2} \leq \frac{1}{u^2}\int \lambda_1^2dH(\boldsymbol{\lambda}) \leq \frac{D_0}{u^2}.
\end{align}
Similarly, $I_{0,2}(\textbf{h},H) \leq D_0/u^2$ and therefore, $I_{1,1}(\textbf{h},H) \leq D_0/u^2$.
From (\ref{rho_bound}), we have 
$$\max \,\{|\Re(\rho_1(c\textbf{h})), |\Re(\rho_2(c\textbf{h}))|, |\Im(\rho_1(c\textbf{h})), |\Im(\rho_2(c\textbf{h}))|\} < \epsilon.$$

By similar arguments, we also have $I_{1,0}(\textbf{h}, H)| \, \leq C_0/u^2$. Then it turns out that
\begin{align}
|\Im(h_1(z))| &\leq \, I_{1,0} + |\Im(\rho_1(c\textbf{h}))\,|I_{2,0} + |\Im(\rho_2(c\textbf{h}))\,|I_{1,1} \leq \frac{C_0}{u^2} + \frac{2D_0\epsilon}{u^2}.
\end{align}
By (\ref{real_h1}), we observe that for arbitrary $M > 0$, we have
\begin{align}\label{lower_bound_of_real_of_h1}
    \Re(h_1(z)) &\geq \, uI_{(1,0)}(\textbf{h},H)\\ \notag
    &= \int_{\mathbb{R}_+^2} \frac{u\lambda_1dH(\boldsymbol{\lambda})}{|-z + \boldsymbol{\lambda}^T\boldsymbol{\rho}(c\textbf{h})|^2}\\ \notag
    &\geq \int_{[0, M]^2} \frac{u\lambda_1dH(\boldsymbol{\lambda})}{|-z + \boldsymbol{\lambda}^T\boldsymbol{\rho}(c\textbf{h})|^2}\\ \notag    
    &\geq \int_{[0, M]^2} \frac{u\lambda_1 dH(\boldsymbol{\lambda})}{(|z| + |\boldsymbol{\lambda}^T\boldsymbol{\rho}(c\textbf{h})|)^2}\\\notag
    &\geq \,  \frac{u\int_{[0, M]^2} \lambda_1dH(\boldsymbol{\lambda})}{(|z|+2M\epsilon)^2}\notag,
\end{align}
where, we used the fact that $|\boldsymbol{\lambda}^T\boldsymbol{\rho}(c\textbf{h})| \leq ||\boldsymbol{\lambda}||_2 ||\boldsymbol{\rho}(c\textbf{h})||_2 \leq 2M\epsilon$ by the Cauchy-Schwarz inequality. 

To produce a positive lower bound, we define the following quantity
\begin{align}\label{defining_C0_min}
    E_0 := \underset{k=1,2}{\min} \int \lambda_kdH(\boldsymbol{\lambda}),
\end{align}
and, by our choice of $H$ in Theorem \ref{mainTheorem}, we have $E_0 > 0$. $M=M_H>0$ will be chosen, depending on $H$ such that
$$\int_{[0, M]^2} \lambda_1dH(\boldsymbol{\lambda}) \geq \frac{1}{2}\int_{\mathbb{R}_+^2} \lambda_1dH(\boldsymbol{\lambda}).$$ 

Now, we derive some precise bounds for the numerator and the denominator of the RHS of (\ref{lower_bound_of_real_of_h1}). Since $|v| \leq u$,  choosing $z = -u + \mathbbm{i}v$ with $u \geq 2$ and, $\epsilon=1/M$ gives us 
\begin{align}\label{denominator_bound}
    (|z| + 2)^2 = |z|^2 + 4|z| + 4 \leq |z|^2 + 4|z|^2 + |z|^2 \leq 6|z|^2 \leq 6(u^2 + u^2) = 12u^2.
\end{align}

Therefore, we get
\begin{align}
    \Re(h_1(z)) \geq \frac{u\int_{[0, M]^2} \lambda_1dH(\boldsymbol{\lambda})}{(|z|+2)^2} \geq \frac{u\int_{\mathbb{R}_+^2}\lambda_1dH(\boldsymbol{\lambda})/2}{12u^2} \geq \frac{uE_0}{24u^2} = \frac{E_0}{24u}.
\end{align}
Denoting $\epsilon_H = 1/M_H$, we define the following quantity:
\begin{align}\label{cutoff_for_u}
 U_0 := \max\bigg\{2, 
\dfrac{cC_0}{\delta(\epsilon_H)}, \dfrac{24(C_0 + 2D_0\epsilon_H)}{E_0}\bigg\}.
\end{align}
Combining everything we conclude that when $u > U_0$ and $|v| < u$, then for $z = -u + \mathbbm{i}v$ and $k=1,2$, we must have $|\Im(h_k(z))| \leq \Re(h_k(z))$. We emphasize the fact that $\epsilon > 0$ in (\ref{cutoff_for_u}) depends on $H$.
\end{proof}

\begin{remark}\label{remark_lipschitz_constant_1}
    If $u > U_0$, we have 
\begin{align*}
    \frac{cC_0}{u} \leq \frac{cC_0}{U_0} \leq \frac{cC_0}{cC_0/\delta(\epsilon_H)} =\delta(\epsilon_H) < 1,
\end{align*}
since, we chose $\delta < 1$ without loss of generality. Then, setting $b = cC_0/u$ in Lemma \ref{Lipschitz}, we conclude that the Lipschitz constant for $\rho_k(\cdot, \cdot)$ in the region $\mathcal{S}(b) \times \mathcal{S}(b)$ must be equal to $K_0= \max\{1,b^2\} = 1$. 
\end{remark}

\subsubsection{Proof of Theorem \ref{Uniqueness}}\label{sec:ProofOfUniqueness}
\begin{proof}
Suppose there exists two distinct analytic solutions $\textbf{h} = (h_1, h_2)$ and $\textbf{g}= (g_1, g_2)$ to (\ref{h_main_eqn}) and they both map $\mathbb{C}_L$ to $\mathbb{C}_R^2$. We start with a sketch of the proof.
\begin{itemize}
    \item[1] Define the quantity 
    \begin{align}\label{defining_R0}
        R_0 := \max\{U_0, 2\sqrt{cD_0}\},
    \end{align}
    where $U_0$ was defined in (\ref{cutoff_for_u}) in the proof of Lemma \ref{location_of_solution_General}. Let $z = -u+\mathbbm{i}v \in \mathbb{C}_L$ with $|v| < u$ and $u > R_0$. By the same Lemma, any solution of (\ref{h_main_eqn}) lies in $\mathcal{S}(C_0/u)^2 = \mathcal{S}(C_0/u) \times \mathcal{S}(C_0/u)$.
    \item[2] In particular, $c\textbf{h}, c\textbf{g} \in \mathcal{S}(cC_0/u)^2$. By Remark \ref{remark_lipschitz_constant_1}, $\rho_1, \rho_2$ are Lipschitz continuous on $\mathcal{S}(cC_0/u)^2$ with Lipschitz constant equal to unity.
    \item[3] We will first show that $g_k(z) = h_k(z);k=1,2$ for $z$ as defined in item 1.
    \item[4] By the uniqueness of analytic extensions, we must have $g_k(z) = h_k(z)$ for all $z \in \mathbb{C}_L$.
\end{itemize}

To show item 3, note that
\begin{align*}
    g_1 - h_1 &= \int \frac{\lambda_1dH(\boldsymbol{\lambda})}{-z+\boldsymbol{\lambda}^T\boldsymbol{\rho}(c\textbf{g})} - \int \frac{\lambda_1dH(\boldsymbol{\lambda})}{-z+\boldsymbol{\lambda}^T\boldsymbol{\rho}(c\textbf{h})}\\
    &= \int \frac{\lambda_1\boldsymbol{\lambda}^T(\boldsymbol{\rho}(c\textbf{h}) - \boldsymbol{\rho}(c\textbf{g})))dH(\boldsymbol{\lambda})}{[-z+\boldsymbol{\lambda}^T\boldsymbol{\rho}(c\textbf{g})][-z+\boldsymbol{\lambda}^T\boldsymbol{\rho}(c\textbf{h})]}\\
    &=\int \frac{\bigg(\lambda_1^2(\rho_1(c\textbf{h})-\rho_1(c\textbf{g})) + \lambda_1\lambda_2(\rho_2(c\textbf{h})-\rho_2(c\textbf{g}))\bigg)dH(\boldsymbol{\lambda})}{[-z+\boldsymbol{\lambda}^T\boldsymbol{\rho}(c\textbf{g})][-z+\boldsymbol{\lambda}^T\boldsymbol{\rho}(c\textbf{h})]}.
\end{align*}

We have $cg_1,cg_2,ch_1,ch_2 \in \mathcal{S}(cC_0/u)$ and $\rho_1, \rho_2$ are Lipschitz continuous with constant $K_0=1$. Now using $\Ddot{H}$older's Inequality, we get
\begin{align*}
    |g_1-h_1|
    &\leq \int \frac{\bigg(\lambda_1^2|\rho_1(c\textbf{h})-\rho_1(c\textbf{g})| + \lambda_1\lambda_2|\rho_2(c\textbf{h})-\rho_2(c\textbf{g})|\bigg)dH(\boldsymbol{\lambda})}{|-z+\boldsymbol{\lambda}^T\boldsymbol{\rho}(c\textbf{g})| \times |-z+\boldsymbol{\lambda}^T\boldsymbol{\rho}(c\textbf{h})|}\\
    &\leq K_0||c\textbf{h}-c\textbf{g}||_1 \int \frac{(\lambda_1^2+\lambda_1\lambda_2)dH(\boldsymbol{\lambda})}{|-z+\boldsymbol{\lambda}^T\boldsymbol{\rho}(c\textbf{g})| \times |-z+\boldsymbol{\lambda}^T\boldsymbol{\rho}(c\textbf{h})|} \\
    &\leq cK_0||\textbf{h}-\textbf{g}||_1\bigg(\sqrt{I_{2,0}(\textbf{g}, H)I_{2,0}(\textbf{h},H)} +\sqrt{I_{2,0}(\textbf{g},H)I_{0,2}(\textbf{h},H)}\bigg).
\end{align*}

Similarly, we get 
\begin{align*}
    |g_2 - h_2| \leq cK_0||\textbf{h}-\textbf{g}||_1\bigg(\sqrt{I_{0,2}(\textbf{g},H)I_{2,0}(\textbf{h},H)} + \sqrt{I_{0,2}(\textbf{g},H)I_{0,2}(\textbf{h},H)}\bigg).
\end{align*}

Then, using the inequality $\sqrt{wx} + \sqrt{yz} \leq \sqrt{w+y}\sqrt{x+z}$ for $w,x,y,z, \geq 0$, we have
\begin{align}\label{the_great_equation}
    ||\textbf{h}-\textbf{g}||_1 \leq 2cK_0||\textbf{h}-\textbf{g}||_1\underbrace{\bigg(\sqrt{I_{2,0}(\textbf{g},H)+I_{0,2}(\textbf{g},H)}\sqrt{I_{2,0}(\textbf{h},H)+I_{0,2}(\textbf{h},H)}\bigg)}_{:=P_0}.
\end{align}

Now note that with $D_0$ as specified in (\ref{secondSpectralBound}), we have
\begin{align}\label{I_20_bound}
    &I_{2,0}(\textbf{h},H) = \int \frac{\lambda_1^2dH(\boldsymbol{\lambda})}{|-z + \boldsymbol{\lambda}^T\boldsymbol{\rho}(c\textbf{h})|^2} \leq \frac{D_0}{u^2} \text{ and}\\\notag
    &I_{0,2}(\textbf{h},H) = \int\frac{\lambda_2^2dH(\boldsymbol{\lambda})}{|-z + \boldsymbol{\lambda}^T\boldsymbol{\rho}(c\textbf{h})|^2} \leq \frac{D_0}{u^2}\\\notag
     \implies& P_0 \leq \frac{2D_0}{u^2}.
\end{align}
Therefore we have,
\begin{align}
 2cK_0P_0 \leq \frac{4cD_0}{u^2} < 1 \text{, when } u > R_0.
\end{align}
Now (\ref{the_great_equation}) implies that 
$||\textbf{h}-\textbf{g}||_1 < ||\textbf{h}-\textbf{g}||_1$ which is a contradiction. Therefore, for $z \in \mathbb{C}_L$ with (absolute value of) real part larger than $R_0$, we have established uniqueness of the solution to (\ref{h_main_eqn}).

So for $u = |\Re(z)| > R_0$ and $|v|<u$, we have $\textbf{h}(z) = \textbf{g}(z)$. Now observe that $h_1, h_2, g_1, g_2$ are all analytic functions on $\mathbb{C}_L$. For $k=1,2$, $h_k$ and $g_k$ agree whenever $|\Re(z)| > M_0$ and in particular over an open subset of $\mathbb{C}_L$. This implies that $h_k = g_k$ over all of $\mathbb{C}_L$ by the Identity Theorem. Thus $\textbf{h}(z)=\textbf{g}(z), \forall\, z \in \mathbb{C}_L$.
\end{proof}

\subsection{\textbf{Results related to proof of existence in (\ref{h_main_eqn})}}

\subsubsection{Proof of Lemma \ref{tightness}}\label{sec:proofOfTightness}
\begin{proof}
For $n \in \mathbb{N}$, define the following:
\begin{align*}
 A_n := \frac{1}{\sqrt{n}}\begin{bmatrix}
    Z_{1n} & 0\\
    0 & Z_{2n}
\end{bmatrix};\quad B_n := \frac{1}{\sqrt{n}}\begin{bmatrix}
    Z_{2n}^* & 0\\
    0 & -Z_{1n}^*
\end{bmatrix};\quad P_n := [\Sigma_{1n}^\frac{1}{2}:\Sigma_{2n}^\frac{1}{2}];\quad Q_n := \begin{bmatrix}
    (\Sigma_{2n}^\frac{1}{2})^*\\
    (\Sigma_{1n}^\frac{1}{2})^*
\end{bmatrix}.
\end{align*}
Then, $S_n =  P_nA_nB_nQ_n$. Also note that $F^{A_nA_n^*} = F^{B_n^*B_n}$ and $F^{P_nP_n^*} = F^{Q_n^*Q_n}$. Note that while the support of $F^{S_n}$ is purely imaginary, those of $F^{\sqrt{A_nA_n^*}}$, $F^{P_nP_n^*}$ are purely real.

For arbitrary $K_1, K_2> 0$, let $K = K_1^2K_2^2$. Using (\ref{ESD_sksym_equal_sym}) and Lemma 2.3 of \cite{BaiSilv95}, we have
\begin{align}\label{tight_equation}
&F^{S_n}\{(-\infty, -\mathbbm{i}K)\cup (\mathbbm{i}K, \infty)\} \\\notag
&= F^{\sqrt{S_nS_n^*}}\{(K, \infty)\} \notag\\
&\leq F^{\sqrt{A_nA_n^*}}\{(K_1, \infty)\} + F^{\sqrt{P_nP_n^*}}\{(K_2, \infty)\} + F^{\sqrt{Q_nQ_n^*}}\{(K_2, \infty)\} + F^{\sqrt{B_nB_n^*}}\{(K_1, \infty)\} \notag\\
&= 2F^{A_nA_n^*}\{(K_1^2, \infty)\}  + 2 F^{P_nP_n^*}\{(K_2, \infty)\}. \notag
\end{align}
In the second term of the last equality, we used the fact that the sets of non-zero eigenvalues of $Q_nQ_n^*$ and of $Q_n^*Q_n$ coincide and the sets of non-zero eigenvalues of $A_nA_n^*$ and of $B_n^*B_n$ coincide. 

Note that  $\{F^{P_nP_n^*}\}_{n=1}^\infty$ and $\{F^{A_nA_n^*}\}_{n=1}^\infty$ are tight sequences. We have $P_nP_n^* = \Sigma_{1n} + \Sigma_{2n}$. Since $\{F^{\Sigma_{kn}}\}_{n=1}^\infty$ is tight for $k=1,2$ and $\Sigma_{1n}$ and $\Sigma_{2n}$ commute, tightness of $\{F^{P_nP_n^*}\}_{n=1}^\infty$ is immediate. The fact that $\{F^{\frac{1}{n}Z_{kn}Z_{kn}^*}\}_{n=1}^\infty$, $k=1,2$ are tight sequences automatically imply that $\{F^{A_nA_n^*}\}_{n=1}^\infty$ is tight.

Now we prove the first result. Suppose $H_n \xrightarrow{d} H=\delta_{(0,0)}$ a.s. Choose $\epsilon, K_1>0$ arbitrarily and set $K_2 = \sqrt{\epsilon/K_1}$. Then, $\{F^{\Sigma_{1n} + \Sigma_{2n}}\}_{n=1}^\infty$ converges weakly to $\delta_0$ for $k =1,2$, we have
\begin{align*}
  \underset{n \rightarrow \infty}{\limsup}F^{P_nP_n^*}\{(K_2, \infty)\} = 0.
\end{align*}

Now letting $K_1 \rightarrow \infty$ in (\ref{tight_equation}), we see that
$$\underset{n \rightarrow \infty}{\limsup}F^{S_n}\{(-\infty, -\mathbbm{i}\epsilon)\cup (\mathbbm{i}\epsilon, \infty)\} \leq \underset{K_1 \rightarrow \infty}{\limsup}F^{A_nA_n^*}\{(K_1^2, \infty)\} = 0.$$ 

Since $\epsilon > 0$ was chosen arbitrarily, we conclude that $F^{S_n} \xrightarrow{d} \delta_0$ a.s. This justifies why we exclusively stick to the case where $H \neq \delta_{(0,0)}$ in Theorem \ref{mainTheorem}.

Now suppose $H_n \xrightarrow{d} H \neq \delta_{(0,0)}$ a.s. The tightness of $\{F^{S_n}\}_{n=1}^\infty$ is immediate from (\ref{tight_equation}) upon utilizing the tightness of $\{F^{P_nP_n^*}\}_{n=1}^\infty$ and $\{F^{A_nA_n^*}\}_{n=1}^\infty$.
\end{proof}

\begin{lemma}\label{Rank2Perturbation}
Let $M_n \in \mathbb{C}^{p \times p}$ be a sequence of deterministic matrices with bounded operator norm, i.e. $||M_n||_{op} \leq B$ for some $B \geq 0$. Under Assumptions \ref{A12}, for $1 \leq j \leq n$, $z \in \mathbb{C}_L$ and sufficiently large $n$, we have $$\underset{1 \leq j \leq n}{\max}|\operatorname{trace}\{M_nQ(z)\} - \operatorname{trace}\{M_nQ_{-j}(z)\}| \leq \dfrac{4cC_0B}{\Re^2 (z)} \text{ a.s.}$$
Consequently, $\underset{1 \leq j \leq n}{\max}|\frac{1}{p}\operatorname{trace}\{M_n(Q(z) - Q_{-j}(z))\}| \xrightarrow{a.s.} 0$.
\end{lemma}
\begin{proof}
Fix $z \in \mathbb{C}_L$ and denote $Q(z)$ as Q. By $\boldsymbol{R}_0$ and (\ref{R4}), for any $1\leq j \leq n$, we have
\begin{align}\label{B.2}
        & |\operatorname{trace}\{M_nQ\} - \operatorname{trace}\{M_nQ_{-j}\}|\\ \notag
        =&|\operatorname{trace}\{M_n(S_n - zI_p)^{-1}\} - \operatorname{trace}\{M_n(S_{nj}-zI_p)^{-1}\}|\\ \notag
        =& |\operatorname{trace}\{M_nQ(\dfrac{1}{n}X_{1j}X_{2j}^* - \dfrac{1}{n}X_{2j}X_{1j}^*)Q_{-j}\}|\\ \notag
        =&\frac{1}{n}|X_{2j}^*Q_{-j}M_nQ X_{1j} - X_{1j}^*Q_{-j} M_nQ X_{2j}|\\ \notag
        \leq& \dfrac{1}{n}|X_{2j}^*Q_{-j}MQ X_{1j}| +\dfrac{1}{n}|X_{1j}^*Q_{-j} MQ X_{2j}|\\ \notag
        \leq& ||Q_{-j}M_nQ||_{op}\bigg(\sqrt{\frac{1}{n}X_{2j}^*X_{2j}}\sqrt{\frac{1}{n}X_{1j}^*X_{1j}} + \sqrt{\frac{1}{n}X_{1j}^*X_{1j}}\sqrt{\frac{1}{n}X_{2j}^*X_{2j}}\bigg).
\end{align}

Note that, we have
$$||Q_{-j}M_nQ||_{op} < B/\Re ^2(z) \text{ since } ||Q_{-j}||_{op},||Q||_{op} \leq 1/|\Re(z)|, ||M_n||_{op} < B.$$

For a fixed $k=1,2$, we have $X_{kj}^*X_{kj} = Z_{kj}^*\Sigma_{kn}Z_{kj}$ where $\Sigma_{kn}$ satisfies \textbf{A1} and $Z_{1}, Z_{2}$ satisfy \textbf{A2} respectively of Assumptions \ref{A12}. Setting $x_{jn} = Z_{kj}$ and $A_j = \Sigma_{kn}$ for $1 \leq j \leq n$ and applying Corollary \ref{quadraticFormCorollary}, we have 
$$\underset{1 \leq j \leq n}{\max}\bigg|\frac{1}{n}X_{kj}^*X_{kj} - \frac{1}{n}\operatorname{trace}(\Sigma_{kn})\bigg| \xrightarrow{a.s.} 0.$$

From $\boldsymbol{T}_1$ of Theorem \ref{mainTheorem} and (\ref{firstSpectralBound}), for sufficiently large $n$, we have 
$$\dfrac{1}{n}\operatorname{trace}(\Sigma_{kn}) = c_n\bigg(\dfrac{1}{p}\operatorname{trace}(\Sigma_{kn})\bigg) < 2cC_0.$$ This implies that for large $n$,
$$\underset{1 \leq j \leq n}{\max}\bigg|\frac{1}{n}X_{kj}^*X_{kj}\bigg| < 2cC_0 \text{ a.s.}$$

Combining everything with (\ref{B.2}), for large $n$, we must have
\begin{align*}
    \underset{1 \leq j \leq n}{\max}|\operatorname{trace}\{M_nQ\} - \operatorname{trace}\{M_nQ_{-j}\}|
    \leq \frac{B}{\Re^2(z)}(2cC_0 + 2cC_0) = \frac{4cC_0 B}{\Re ^2(z)} \text{ a.s.}
\end{align*}

For $z \in \mathbb{C}_L$, it is clear that for arbitrary $\epsilon > 0$,
$\underset{1 \leq j \leq n}{\max}|\frac{1}{p}\operatorname{trace}\{M(Q - Q_{-j})\}| < \epsilon$ a.s. for large $n$. Therefore, $\underset{1 \leq j \leq n}{\max}|\frac{1}{p}\operatorname{trace}\{M(Q - Q_{-j})\}| \xrightarrow{a.s.} 0$.
\end{proof}

\begin{lemma}\label{ConcetrationLemma}
    Under Assumptions \ref{A12}, for $z \in \mathbb{C}_L$ and $k=1,2$, we have $|h_{kn}(z) - \mathbb{E}h_{kn}(z)| \xrightarrow{a.s.} 0$.
\end{lemma}
\begin{proof}
Define $\mathcal{F}_j = \sigma(\{X_{1r}, X_{2r}:j+1 \leq r \leq n\})$ and for a measurable function $f$, we denote $\mathbb{E}_{j}f(X) := \mathbb{E}(f(X)|\mathcal{F}_j)$ for $0 \leq j \leq n-1$ and $\mathbb{E}_{n}f(X) := \mathbb{E}f(X)$. For $k=1,2$, we observe that
\begin{align*}
    h_{kn}(z) - \mathbb{E}h_{kn}(z)
    &= \dfrac{1}{p}\operatorname{trace}(\Sigma_{kn}Q(z)) - \mathbb{E}\bigg(\dfrac{1}{p}\operatorname{trace}(\Sigma_{kn}Q(z))\bigg)\\
    &= \dfrac{1}{p}\sum_{j=1}^{n}(\mathbb{E}_{j-1} - \mathbb{E}_{j})\operatorname{trace}(\Sigma_{kn}Q(z))\\    
    &= \dfrac{1}{p}\sum_{j=1}^{n}(\mathbb{E}_{j-1} - \mathbb{E}_{j})\bigg(\underbrace{\operatorname{trace}(\Sigma_{kn}Q(z)) - \operatorname{trace}(\Sigma_{kn}Q_{-j}(z))}_{:= Y_j}\bigg)\\
    &= \dfrac{1}{p}\sum_{j=1}^{n}\underbrace{(\mathbb{E}_{j-1} - \mathbb{E}_{j}) Y_j}_{:= D_j}
    = \frac{1}{p}\sum_{j=1}^{n}D_j.
\end{align*}
Denote $Q(z)$ as $Q$ and $Q_{-j}(z)$ as $Q_{-j}$. From (\ref{B.2}), we have
\begin{align*}
    |Y_j| = |\operatorname{trace}\{\Sigma_{kn}Q\} - \operatorname{trace}\{\Sigma_{kn}Q_{-j}\}| \leq \dfrac{\tau}{\Re^2(z)}W_{nj} \text{, where } W_{nj} := \frac{1}{n}(||X_{1j}||^2 + ||X_{2j}||^2).
\end{align*}
So, we have  $\absmod{D_j} \leq \dfrac{2\tau}{\Re^2(z)}W_{nj}$. By Lemma 2.12 of \cite{BaiSilv09}, there exists $K_4$ depending only on $z \in \mathbb{C}_L$ such that
\begin{align}\label{concInequality}
   \mathbb{E}\absmod{h_{kn}(z) - \mathbb{E}h_{kn}(z)}^4 = \mathbb{E}\absmod{\frac{1}{p}\sum_{j=1}^n D_j}^4 &\leq \frac{K_4}{p^4} \mathbb{E}\bigg(\sum_{j=1}^n |D_j|^2\bigg)^2 \leq \frac{16K_4\tau^4}{p^4\Re^8(z)} \mathbb{E}\bigg(\sum_{j=1}^n |W_{nj}|^2\bigg)^2 \\ \notag
   &= \frac{K_0}{p^4}\bigg(\sum_{j=1}^n\mathbb{E}|W_{nk}|^4+\sum_{j\neq l}\mathbb{E}|W_{nj}|^2\mathbb{E}|W_{nl}|^2\bigg).
\end{align}

We have the following inequalities.
\begin{enumerate}
    \item $||X_{kj}||^m = (Z_{kj}^*\Sigma_{kn}Z_{kj})^\frac{m}{2} \leq (||\Sigma_{kn}||_{op}||Z_{kj}||^2)^\frac{m}{2} \leq \tau^\frac{m}{2}||Z_{kj}||^m$ for $m \geq 1$.
    \item $|W_{nj}|^2 \leq  \dfrac{2}{n^2}(||X_{1j}||^4 + ||X_{2j}||^4) \leq \dfrac{2\tau^2}{n^2}(||Z_{1j}||^4 + ||Z_{2j}||^4)$.
    \item $|W_{nj}|^4 \leq  \dfrac{8}{n^4}(||X_{1j}||^8 + ||X_{2j}||^8) \leq \dfrac{8\tau^4}{n^4}(||Z_{1j}||^8 + ||Z_{2j}||^8)$.
\end{enumerate}

Recall that $Z_{kj}$ is the $j^{th}$ column of $Z_k;k=1,2$ and $z_{rj}^{(k)}$ represents the $r^{th}$ element of $Z_{kj}$. By Assumptions \ref{A12}, we have the following bounds.
\begin{enumerate}
    \item For $1 \leq t \leq 2$, there exists $M_t<\infty$ depending on t such that  $$\mathbb{E}|z_{rj}^{(k)}|^m \leq M_t < \infty.$$
    \item For $2< t$, we have 
    $$\mathbb{E}|z_{rj}^{(k)}|^t \leq n^{b(t-2)}.$$
\end{enumerate}

So, we have
\begin{align}\label{Bound_norm_4}
    \mathbb{E}||Z_{1j}||^4 = \mathbb{E} \bigg(\sum_{r=1}^p|z_{rj}^{(1)}|^2\bigg)^2 &= \mathbb{E} \bigg(\sum_{r=1}^p|z_{rj}^{(1)}|^4 + \sum_{r \neq s}|z_{rj}^{(1)}|^2|z_{sj}^{(1)}|^2\bigg) \\ \notag
    &\leq pn^{2b} + p(p-1) = O(n^2), \text{ and}
\end{align}
\begin{align}\label{Bound_norm_8}
    \mathbb{E}||Z_{1j}||^8 = \mathbb{E} \bigg(\sum_{r=1}^p|z_{rj}^{(1)}|^2\bigg)^4 
    &= \mathbb{E} \bigg(\sum_{r=1}^p|z_{rj}^{(1)}|^8 + \sum_{r \neq s}|z_{rj}^{(1)}|^6|z_{sj}^{(1)}|^2 + \sum_{r \neq s}|z_{rj}^{(1)}|^4|z_{sj}^{(1)}|^4 \bigg)\\ \notag
    &\leq pn^{6b}
     + p(p-1)[n^{4b} + (n^{2b})^2]\\\notag
     &= O(\max\{n^{1+6b}, n^{2+4b}\}) = O(n^{2+4b}).
\end{align}
Therefore, combining everything, we get
\begin{align*}
    \mathbb{E}|W_{nj}|^2 \leq 
    \frac{4\tau^2}{n^2}K_1n^2 = 4K_1\tau^2; \quad 
    \mathbb{E}|W_{nj}|^4 \leq \frac{16\tau^4}{n^4} K_2n^{2+4b} = 16K_2\tau^4n^{4b-2}.
\end{align*}
Using these in (\ref{concInequality}), we get
\begin{align*}
    \mathbb{E}\absmod{h_{kn}(z) - \mathbb{E}h_{kn}(z)}^4 = \mathbb{E}\absmod{\frac{1}{p}\sum_{j=1}^n D_j}^4 &\leq \frac{K_0}{p^4}\bigg(\sum_{j=1}^n\mathbb{E}|W_{nk}|^4+\sum_{j\neq l}\mathbb{E}|W_{nj}|^2\mathbb{E}|W_{nl}|^2\bigg)\\
    &\leq \frac{K_0}{p^4}\bigg(n\frac{16K_2\tau^4}{n^{2-4b}} + n^2(4K_1\tau^2)^2\bigg) = O\bigg(\frac{1}{n^2}\bigg).
\end{align*}

Finally, by Borel Cantelli Lemma, we have $|h_{kn}(z) - \mathbb{E}h_{kn}(z)| \xrightarrow{a.s.} 0$. The other result follows similarly.
\end{proof}

\color{black}
\begin{definition}
    Let $\mathcal{H}_{s,t}$ denote the region $\mathcal{H}_{s,t} := \{h \in \mathbb{C}_R: \Re(h) \geq s, |h| \leq t\}$ for $0 < s \leq t$.
\end{definition}

\begin{lemma}\label{boundedAwayfromZero}
Let $z \in \mathbb{C}_L$. Then there exists $s,t$ independent of $n$ such that $0 < s \leq t$ and for sufficiently large $n$ and under \textbf{A1} of Assumptions \ref{A12}, we have 
    \begin{enumerate}
        \item[(1)] $c_n\textbf{h}_n(z) = (c_nh_{1n}(z), c_nh_{2n}(z)) \in \mathcal{H}_{s,t}^2$,
        \item[(2)] $c_n\mathbb{E}\textbf{h}_n(z) = (c_n\mathbb{E}h_{1n}(z), c_n\mathbb{E}h_{2n}(z)) \in \mathcal{H}_{s,t}^2$,
        \item[(3)] $c_n\tilde{\textbf{h}}_n(z) = (c_n\tilde{h}_{1n}(z), c_n\tilde{h}_{2n}(z)) \in \mathcal{H}_{s,t}^2$.
    \end{enumerate}
\end{lemma}
\begin{proof}
Under \textbf{A1} of Assumptions \ref{A12}, we have $||\Sigma_{1n}||_{op}, ||\Sigma_{2n}||_{op}\leq \tau$. Since $H_n$ and $H$ are compactly supported on (a subset of) $[0, \tau]^2$ and $H_n \xrightarrow{d} H$ a.s., we get
\begin{align}
    \int_0^\tau\lambda_k dH_n(\boldsymbol{\lambda}) \xrightarrow{} \int_0^\tau\lambda_k dH(\boldsymbol{\lambda}) \quad k=1,2.
\end{align}
Moreover, this limit must be positive since $H$ is not supported entirely on the real or the imaginary axis. Therefore,
\begin{align}\label{trace_limit}
\frac{1}{n}\operatorname{trace}(\Sigma_{1n}) = c_n\int_0^\tau\lambda_1 dH_n(\boldsymbol{\lambda}) \xrightarrow{} c\int_0^\tau\lambda_1 dH(\boldsymbol{\lambda}) > 0.
\end{align}

Let $z = -u + \mathbbm{i}v$ with $u > 0$. Denoting $a_{ij}$ as the $ij^{th}$ element of $A:= P^*\Sigma_{1n}P$ where $S_n = P\Lambda P^*$ and $\Lambda = \operatorname{diag}(\{\mathbbm{i}\lambda_j\}_{j=1}^p)$ is a diagonal matrix containing the purely imaginary (or zero) eigenvalues of $S_n$. Then,
    \begin{align*}
        c_nh_{1n}(z) = \frac{p}{n}\frac{1}{p}\operatorname{trace}\{\Sigma_{1n}Q(z)\}
        = \frac{1}{n}\operatorname{trace}\{P^*\Sigma_{1n}P(\Lambda - zI_p)^{-1}\}
         =\frac{1}{n}\sum_{j=1}^p \frac{a_{jj}}{\mathbbm{i}\lambda_j  - z}.
    \end{align*}

For any $\delta > 0$, we have
\begin{align}
||S_n||_{op} &= ||\frac{1}{n}X_1X_2^* - \frac{1}{n}X_2X_1^*||_{op}\\\notag
&\leq \,2\sqrt{||\frac{1}{n}X_1X_1^*||_{op}}\sqrt{||\frac{1}{n}X_2X_2^*||_{op}}\\ \notag
&\leq \,||\Sigma_{1n}||_{op}(1+\sqrt{p/n})^2 + \delta/2 + ||\Sigma_{2n}||_{op}(1+\sqrt{p/n})^2 + \delta/2\\\notag
&\leq \, 2\tau(1+\sqrt{c_n})^2 + \delta.
\end{align}

Let $B = 4\tau(1+\sqrt{c})^2$. Then $\mathbb{P}(|\lambda_j| > B \hspace{2mm} i.o.) = 0$.

\footnote{$\operatorname{sgn}(x)$ is the sign function}
Define $B^* :=\left\{\begin{matrix}
    -B\operatorname{sgn}(v) & \text{, if } v \neq 0, \\
    B &  \text{, if } v = 0.
\end{matrix} \right.$

Then $(\lambda_j - v)^2 \leq (B^* - v)^2$. Therefore,
\begin{align*}
    \Re(c_nh_{1n}(z)) &= \frac{1}{n}\sum_{j=1}^p\frac{a_{jj}u}{(\lambda_j - v)^2 + u^2}\\
    &\geq \frac{1}{n}\sum_{j=1}^p \frac{a_{jj}u}{(B^* - v)^2 + u^2}\\
    &= \frac{u}{(B^* - v)^2 + u^2} \bigg(\frac{1}{n}\sum_{j=1}^p a_{jj}\bigg)\\
    &= \frac{u}{(B^* - v)^2 + u^2} \bigg(\frac{1}{n}\operatorname{trace}(\Sigma_{1n})\bigg) \text{, as } \operatorname{trace}(A) = \operatorname{trace}(\Sigma_{1n})\\
    &\longrightarrow \frac{u}{(B^* - v)^2 + u^2}\bigg(c\int_0^\tau\lambda_1 dH(\boldsymbol{\lambda})\bigg) := K_1
    > 0 \text{ from } (\ref{trace_limit}).
\end{align*} 

Similarly, we define $$K_2 := \frac{u}{(B^* - v)^2 + u^2}\bigg(c\int_0^\tau\lambda_2 dH(\boldsymbol{\lambda})\bigg),$$ and let $K_x(c,z,\tau, H^\tau) := \min(K_1, K_2) > 0$. For $k=1,2$ and sufficiently large $n$, using (\ref{h_nBound}), we have $$K_x \leq \Re(c_nh_{kn}(z)) \leq |c_nh_{kn}(z)| \leq 2cC_0/u.$$
So let $s = K_x$ and $t = 2cC_0/u$. This establishes the first item. In conjunction with Lemma \ref{ConcetrationLemma} and (\ref{Consequence_DetEqv}), the second and third items respectively are immediate.
\end{proof}

\begin{lemma}\label{bound_for_denominator}
    Let $h_1, h_2 \in \mathcal{H}_{s,t}$. Then the quantity $|1+h_1h_2|^{-1}$ is upper bounded.
\end{lemma}

\begin{proof}
       Let $h_1, h_2 \in \mathcal{H}_{s,t}$. First, we establish a bound for $|1 + h_1h_2|^{-1}$. 
    
\textbf{Case1:} $\Re(h_1h_2) \geq 0$. In this case,
\begin{align}
    \frac{1}{|1+ h_1h_2|} = \frac{1}{\sqrt{(1+\Re(h_1h_2))^2 + \Im^2(h_1h_2)}} \leq 1.
\end{align}

\textbf{Case2:} $\Re(h_1h_2) < 0$. Then,  we define $\theta_0 := \cos^{-1}(s/t)$ and $\theta_k := \arg(h_k), k =1,2$.
Clearly,
\begin{align}
  \max\{|\theta_1|, |\theta_2| \leq \theta_0\}.
\end{align}

Since $\Re(h_1h_2) < 0$, this implies that either $\pi/2 < \theta_1 + \theta_2 \leq 2\theta_0$ or $-\pi/2 > \theta_1 + \theta_2 \geq -2\theta_0$ depending on whether $\Im(h_1h_2)$ is positive or negative. Irrespective of the sign of $\Im(h_1h_2)$, we observe that 
\begin{align}
    |\sin(\theta_1+ \theta_2)| \geq \sin(2\theta_0).
\end{align}
Since $|h_k| \geq \Re(h_k) \geq s > 0$ and $\theta_0 \neq 0$, we observe that 
\begin{align}
    |\Im(h_1h_2)| = |h_1h_2\Im(e^{\mathbbm{i}(\theta_1 + \theta_2)})| \geq |\Re(h_1)|\,|\Re(h_2)|\,|\sin(2\theta_0)| = s^2|\sin(2\theta_0)| := L_0 > 0.
\end{align}
Thus, we have
\begin{align}
    \frac{1}{|1+h_1h_2|} \leq \frac{1}{|\Im(h_1h_2)|} \leq \frac{1}{L_0}.
\end{align}
Combining both cases, we conclude that $|1+h_1h_2|^{-1} \leq M_0 := \max\{1, 1/L_0\}$.
\end{proof}

\begin{lemma}\label{Lipschitz2}
    \textbf{Lipschitz within a hemisphere:} For $0 < s \leq t$, the functions $\rho_k(\cdot, \cdot)$, $k=1,2$ are Lipschitz continuous on $\mathcal{H}_{s,t}^2$. 
\end{lemma}

\begin{proof}
Let $M_0$ be as defined in Lemma \ref{bound_for_denominator}. For $\textbf{h}=(h_1,h_2), \textbf{g}=(g_1,g_2) \in \mathcal{H}_{s,t}^2$, we observe that
\begin{align*}
        \absmod{\rho_1(\textbf{h}) - \rho_1(\textbf{g})}
        &= \absmod{\frac{h_2}{1+h_{1}h_{2}} - \frac{g_2}{1+g_{1}g_{2}}}\\
        &= \absmod{\frac{(h_2 - g_2) + h_2g_2(g_1-h_1)}{(1+h_1h_2)(1+g_1g_2)}}\\
        &\leq \absmod{\frac{h_2 - g_2}{(1+h_1h_2)(1+g_1g_2)}} + \frac{|h_2|}{|1+h_1h_2|}\frac{|g_2|}{|1+g_1g_2|} |g_1-h_1|\\
        &\leq |h_2-g_2|M_0^2 + (tM_0)^2|h_1-g_1|\\
        &\leq K_0||\textbf{h}-\textbf{g}||_1 \text{, where } K_0:= \max\{M_0^2, M_0^2t^2\}.
    \end{align*}
The same Lipschitz constant also works for $\rho_2(\cdot, \cdot)$.    
\end{proof}

\begin{lemma}\label{rho_imp_result}
Under Assumptions \ref{A12}, for $k = 1,2$, we have the following results for $z \in \mathbb{C}_L$:
\begin{description}
    \item[1] $|\rho_k(c_n\textbf{h}_{n}(z)) - \rho_k(c_n\mathbb{E}\textbf{h}_{n}(z))| \xrightarrow{a.s.} 0$,\text{ and}
    \item[2] $|\rho_k(c_n\Tilde{\textbf{h}}_{n}(z)) - \rho_k(c_n\mathbb{E}\textbf{h}_{n}(z))| \xrightarrow{} 0$.
\end{description}
\end{lemma}
\begin{proof}
The first result follows from Lemma \ref{ConcetrationLemma}, Lemma \ref{boundedAwayfromZero} and Lemma \ref{Lipschitz2}. The second result follows from (\ref{Consequence_DetEqv}) and Lemma \ref{Lipschitz2}.
\end{proof}

\begin{lemma}\label{normBound}
    Under Assumptions \ref{A12}, the operator norms of the matrices $\Bar{Q}(z), \Bar{\Bar{Q}}(z)$ defined in Theorem \ref{DeterministicEquivalent} and (\ref{defining_QBarBar}) respectively are bounded by $1/|\Re(z)|$ for $z \in \mathbb{C}_L$.
\end{lemma}
\begin{proof}
        Since $\Sigma_{1n}$ and $\Sigma_{2n}$ commute, there exists a common unitary matrix $P$ such that $\Sigma_{kn} = P\Lambda_k P^*$ where $\Lambda_k = \operatorname{diag}(\{\lambda_{kj}\}_{j=1}^p)$ with $\lambda_{kj} \geq 0$ for $k =1,2$. Therefore, 
\begin{align}\label{QBarz_expansion}
     \Bar{Q}(z)
    &= \bigg(-zPP^* + \rho_1(\mathbb{E}c_n\textbf{h}_n)P\Lambda_1 P^* + \rho_2(\mathbb{E}c_n\textbf{h}_n)P\Lambda_2 P^*\bigg)^{-1} \notag\\
    &= P\bigg(-zI_p + \rho_1(\mathbb{E}c_n\textbf{h}_n)\Lambda_1 + \rho_2(\mathbb{E}c_n\textbf{h}_n)\Lambda_2\bigg)^{-1}P^*.
\end{align}

For sufficiently large $n$, we have $\Re(c_nh_{kn}(z)) > 0$ from Lemma \ref{boundedAwayfromZero}. Since $\rho_k(\mathbb{C}_R^2) \subset \mathbb{C}_R$, we observe that for any $1 \leq j \leq p$, the following holds:
\begin{align}\label{Real_of_eigen_QBarz}
\Re(-z +  \rho_1(\mathbb{E}c_n\textbf{h}_n)\lambda_{1j}+ \rho_2(\mathbb{E}c_n\textbf{h}_n)\lambda_{2j}) \geq \Re(-z) > 0.
\end{align}
Using (\ref{QBarz_expansion}) and (\ref{Real_of_eigen_QBarz}), we have $||\Bar{Q}(z)||_{op} \leq 1/|\Re(z)|$. For the other result, note that,
\begin{align*}
    \Bar{\Bar{Q}}(z) = P\bigg(-zI_p + \rho_1(c_n\Tilde{h}_{1n}, c_n\Tilde{h}_{2n})\Lambda_1 + \rho_2(c_n\Tilde{h}_{1n}, c_n\Tilde{h}_{2n})\Lambda_2\bigg)^{-1}P^*.
\end{align*}

Using Lemma \ref{rho_imp_result} and Lemma \ref{Lipschitz}, we conclude that $||\Bar{\Bar{Q}}(z)||_{op} \leq {1}/{|\Re(z)|}$.
\end{proof}

\begin{lemma}\label{hn_tilde_tilde2}
Under Assumptions \ref{A12}, $z \in \mathbb{C}_L$ and $k=1,2$, we have $|\Tilde{h}_{kn}(z) - \Tilde{\Tilde{h}}_{kn}(z)| \rightarrow 0$.
\end{lemma}
\begin{proof}
    Following definitions (\ref{defining_hn_tilde}) and (\ref{defining_hn_tilde2}), we observe that 
\begin{align*}
    |\Tilde{h}_{kn}(z) - \Tilde{\Tilde{h}}_{kn}(z)|
    &= \frac{1}{p}|\operatorname{trace}\{\Sigma_{kn} (\Bar{Q} - \Bar{\Bar{Q}})\}|\\
    &= \frac{1}{p}\absmod{\operatorname{trace}\{\Sigma_{kn} \Bar{Q} \bigg(\sum_{k=1}^2[\rho_k(c_n\Tilde{\textbf{h}}_{n}) - \rho_k(c_n\mathbb{E}\textbf{h}_{n})\Sigma_{kn}]\bigg) \Bar{\Bar{Q}}\}} \text{, by } (\ref{R0})\\
    &\leq \bigg(\frac{1}{p}\operatorname{trace}(\Sigma_{kn})\bigg) \sum_{k=1}^2\absmod{\rho_k(c_n\Tilde{\textbf{h}}_{n})-\rho_k(c_n\mathbb{E}\textbf{h}_{n})\Sigma_{kn}} \times ||\Bar{Q}\Sigma_{kn} \Bar{\Bar{Q}}||_{op} \text{, by } (\ref{R5})\\
    &\leq \sum_{k=1}^2 C_0 |\rho_k(c_n\Tilde{\textbf{h}}_{n}) - \rho_k(c_n\mathbb{E}\textbf{h}_{n}))| \frac{\tau}{\Re^2 (z)} \text{, for large } n \text{ and using Lemma \ref{normBound}}\\ 
    &= \frac{C_0\tau}{\Re^2 (z)}\,||\boldsymbol{\rho}(c_n\Tilde{\textbf{h}}_{n}) - \boldsymbol{\rho}(c_n\mathbb{E}\textbf{h}_{n})||_1.
\end{align*}
Now we use Lemma \ref{rho_imp_result} to conclude the result.
\end{proof}

\begin{remark}\label{column_index_notation}
We will be using $X_{rj}$ (resp. $Z_{rj}$) to denote the $j^{th}$ column of $X_r$ (resp. $Z_r$) for $r=1,2$ and $1\leq j \leq n$. With this notation, we introduce a few quantities.
\end{remark}
\begin{definition}\label{defining_Ejrs}
    $E_j(r,s) := \frac{1}{n}X_{rj}^*Q_{-j}X_{sj} = \frac{1}{n}Z_{rj}^*\Sigma_{rn}^\frac{1}{2}Q_{-j}\Sigma_{sn}^\frac{1}{2}Z_{sj} \text{ for } r,s \in \{1,2\}, 1\leq j \leq n$.
\end{definition}
\begin{definition}\label{defining_Fjrs}
        $F_j(r,s) := \frac{1}{n}X_{rj}^*\Bar{Q}M_nQ_{-j}X_{sj} \text{ for } r,s \in \{1,2\}, 1\leq j \leq n$.
\end{definition}
\begin{definition}\label{defining_mkn}
    $m_{rn}(z) := \frac{1}{n}\operatorname{trace}\{\Sigma_{rn}\Bar{Q}M_nQ\} \text{ for } r \in \{1,2\}$.
\end{definition}
\begin{definition}\label{defining_vn}
    $v_n(z) := \dfrac{1}{1 + c_n^2h_{1n}(z)h_{2n}(z)}$.
\end{definition}

\begin{remark}\label{remark_vn_bound}
    For a fixed $z \in \mathbb{C}_L$, $|v_n(z)|$ is bounded above by a quantity independent of $n$ by Lemma \ref{boundedAwayfromZero} and Lemma \ref{bound_for_denominator}.
\end{remark}

\begin{lemma}\label{uniformConvergence}
Under Assumptions \ref{A12}, the quantities $c_{1j},c_{2j}, d_{1j}, d_{2j}, v_n$ and $F_j(r,s), m_{rn}$ for $r,s =1,2$ as defined throughout the proof of Theorem \ref{DeterministicEquivalent} satisfy the following results.
\begin{align*}
    &\underset{1\leq j\leq n}{\max }|c_{1j} - v_n| \xrightarrow{a.s.} 0; \hspace{18mm} \underset{1\leq j\leq n}{\max }|d_{1j} - v_n|\xrightarrow{a.s.} 0;\\
    &\underset{1\leq j\leq n}{\max }|c_{2j} - c_nv_nh_{1n}|\xrightarrow{a.s.} 0; \hspace{10mm} \underset{1\leq j\leq n}{\max }|d_{2j} - c_nv_nh_{2n}|\xrightarrow{a.s.} 0;\\
    & \underset{1\leq j\leq n}{\max }|F_j(r,r) - m_{rn}| \xrightarrow{a.s.} 0, r \in \{1,2\};\\
    & \underset{1\leq j\leq n}{\max }|F_j(r,s)| \xrightarrow{a.s.} 0  \text{, where } r \neq s, r,s \in \{1,2\}.
\end{align*}  
\end{lemma}

\begin{proof}
Recall the definition of $E_j(r,s)$ from (\ref{defining_Ejrs}). We will first establish a few results related to $E_j(r,s)$. For a fixed $r = 1,2$, let $A_{j} := \Sigma_{rn}^\frac{1}{2}Q_{-j}\Sigma_{rn}^\frac{1}{2}$ and $x_{jn} = Z_{rj}, 1 \leq j \leq n$. We have $||A_{j}||_{op} \leq {\tau}/{|\Re(z)|}$. Then $\{x_{jn}: 1 \leq j \leq n\}_{n=1}^\infty$ and $A_j$ satisfy the conditions of Corollary \ref{quadraticFormCorollary}. Thus, we have 
\begin{align}\label{B.8}
 \underset{1\leq j \leq n}{\max} \absmod{\frac{1}{n}x_{jn}^*A_jx_{jn} - \dfrac{1}{n}\operatorname{trace}(A_{j})} = \underset{1\leq j \leq n}{\max} \left|E_j(r,r) - \dfrac{1}{n}\operatorname{trace}\{\Sigma_{rn} Q_{-j}\}\right| \xrightarrow{a.s.} 0.
\end{align}

From Lemma \ref{Rank2Perturbation}, $|\frac{1}{n}\operatorname{trace}\{\Sigma_{rn}(Q- Q_{-j})\}| \xrightarrow{a.s.} 0$. Observing that $c_nh_{rn} = \dfrac{1}{n}\operatorname{trace}\{\Sigma_{rn}Q\}$ we get
\begin{align}\label{E_j_rr}
    \underset{1\leq j \leq n}{\max} \left|E_j(r,r) - c_nh_{rn}\right| \xrightarrow{a.s.} 0.
\end{align}

From Corollary \ref{quadraticForm_xy}, we also get 
\begin{align}\label{E_j_rs}
\underset{1\leq j \leq n}{\max}|E_j(1,2)|\xrightarrow{a.s.}0 \hspace{6mm}\text{, and } \hspace{6mm} \underset{1\leq j \leq n}{\max}|E_j(2,1)| \xrightarrow{a.s.} 0.
\end{align}

Note that by Lemma \ref{normBound}, 
\begin{align}\label{QbarMQ_bound}
||\Bar{Q}M_nQ_{-j}||_{op} \leq ||\Bar{Q}||_{op}||M_n||_{op}||Q_{-j}||_{op} \leq  \frac{B}{\Re ^2(z)}.
\end{align}

Therefore, repeating the same arguments presented through (\ref{E_j_rr})-(\ref{E_j_rs}) (replacing $Q_{-j}$ with $\Bar{Q}M_nQ_{-j}$ throughout), we get the following uniform almost sure convergence results:
\begin{itemize}
    \item $\underset{1\leq j\leq n}{\max }|F_j(r,r) - m_{rn}| \xrightarrow{a.s.} 0, r \in \{1,2\}$, \text{ and}
    \item $\underset{1\leq j\leq n}{\max }|F_j(r,s)| \xrightarrow{a.s.} 0  \text{ where } r \neq s, r,s \in \{1,2\}$.
\end{itemize}

We now prove the result related to $c_{1j}$ defined in (\ref{defining_cj}). To show $\underset{1 \leq j \leq n}{\max}|c_{1j} - v_n| \xrightarrow{a.s.} 0$,  define for $1 \leq j \leq n$: 
\begin{description}
    \item[1] $A_{jn} = 1 - E_j(1,2)$,
    \item[2] $B_{jn} = Den(j)$ (\text{defined in} (\ref{defining_cj})),
    \item[3] $C_{jn} = 1$ and $D_{jn} = v_n$.
\end{description}
By Remark \ref{remark_vn_bound}, we see that $A_{jn}, B_{jn}, C_{jn}, D_{jn}$ satisfy the conditions of Lemma \ref{lA.3}. Therefore, we have the result associated with $c_{1j}$. The results for $c_{2j}, d_{1j}, d_{2j}$ follow from similar arguments. 
\end{proof}

\section{Proof of Theorem \ref{DeterministicEquivalent}}\label{sec:ProofDeterministicEquivalent}

\begin{proof}
Let $z \in \mathbb{C}_L$. Define $F(z) := \bigg(\Bar{Q}(z)\bigg)^{-1}$. Refer to the notation defined in \ref{column_index_notation}. Using (\ref{R0}), we have
\begin{align}\label{resolventIdentity}
Q - \Bar{Q} = Q \bigg(F + zI_p - \dfrac{1}{n}\sum_{j=1}^{n}(X_{1j}X_{2j}^{*} - X_{2j}X_{1j}^{*})\bigg) \Bar{Q}.
\end{align}

Using the above, we get
\begin{align}\label{Q_minus_QBar}
& \frac{1}{p}\operatorname{trace}\{(Q - \Bar{Q})M_n\}\\
=&\frac{1}{p}\operatorname{trace} \{Q(F + zI_p)\Bar{Q}M_n \} - \frac{1}{p} \operatorname{trace} \{Q \bigg(\sum_{j=1}^{n}\dfrac{1}{n} (X_{1j}X_{2j}^{*} - X_{2j}X_{1j}^{*})\bigg) \Bar{Q} M_n \} \notag \\
=&\frac{1}{p}\operatorname{trace} \{(F + zI_p)\Bar{Q}M_nQ \} - \frac{1}{p} \operatorname{trace} \{\bigg(\sum_{j=1}^{n}\dfrac{1}{n} (X_{1j}X_{2j}^{*} - X_{2j}X_{1j}^{*})\bigg) \Bar{Q} M_n Q\}\notag \\
= & \underbrace{\frac{1}{p}\operatorname{trace} \{(F + zI_p)\Bar{Q}M_nQ \}}_{{Term}_1} - \underbrace{\frac{1}{p} \sum_{j=1}^{n}\dfrac{1}{n} (X_{2j}^{*}\Bar{Q}M_nQX_{1j} - X_{1j}^{*}\Bar{Q}M_nQX_{2j})}_{{Term}_2} \notag.
\end{align}

Simplifying ${Term}_2$ using Lemma \ref{Rank2Woodbury}, with $A = Q_{-j}(z)$ (see (\ref{defining_Qzj})), $u = \frac{1}{\sqrt{n}}X_{1j}$ and $v = \frac{1}{\sqrt{n}}X_{2j}$, we get

\begin{align}\label{defining_cj}
        \frac{1}{\sqrt{n}} QX_{1j} &= Q_{-j} \bigg(\frac{1}{\sqrt{n}}X_{1j} c_{1j} + \frac{1}{\sqrt{n}}X_{2j} c_{2j}\bigg),\\ 
        \text{where } c_{1j} &= (1 - E_j(1,2))Den(j); \hspace{5mm} c_{2j} =E_j(1,1)Den(j)\notag \text{ and}\\
        Den(j) &= \bigg((1 - E_j(1,2))(1 + E_j(2,1)) + E_j(1,1)E_j(2,2)\bigg)^{-1} \notag
\end{align}
and,
\begin{align}\label{defining_dj}
        \frac{1}{\sqrt{n}} QX_{2j} &= Q_{-j} \bigg(\frac{1}{\sqrt{n}}X_{2j} d_{1j} - \frac{1}{\sqrt{n}}X_{1j} d_{2j}\bigg),\\ 
        \text{where } d_{1j} &= (1 + E_j(2,1))Den(j);\hspace{5mm} d_{2j} = E_j(2,2)Den(j)\notag.
\end{align}

Using (\ref{defining_cj}) and (\ref{defining_dj}), $Term_2$ of (\ref{Q_minus_QBar}) can be simplified as follows.
\begin{align}\label{T2_simplification}
&{Term}_2
        =\frac{1}{p}\sum_{j=1}^{n}\dfrac{1}{n} (X_{2j}^{*}\Bar{Q}M_nQX_{1j} - X_{1j}^{*}\Bar{Q}M_nQX_{2j}) \\
        &=\sum_{j=1}^{n}\dfrac{1}{p\sqrt{n}} X_{2j}^*\Bar{Q}M_n \bigg(\frac{1}{\sqrt{n}} Q X_{1j}\bigg) -  \sum_{j=1}^{n}\dfrac{1}{p\sqrt{n}} X_{1j}^*\Bar{Q}M_n \bigg( \frac{1}{\sqrt{n}}QX_{2j}\bigg) \notag\\
        &= \sum_{j=1}^{n}\dfrac{1}{p\sqrt{n}} X_{2j}^*\Bar{Q}M_n Q_{-j} \bigg(\frac{X_{1j}c_{1j} + X_{2j}c_{2j}}{\sqrt{n}}\bigg) - \sum_{j=1}^{n}\dfrac{1}{p\sqrt{n}} X_{1j}^*\Bar{Q}M_n Q_{-j}\bigg( \frac{X_{2j}d_{1j} - X_{1j}d_{2j}}{\sqrt{n}}\bigg) \notag\\
        &= \frac{1}{p}\sum_{j=1}^{n}\Bigg[\bigg(c_{1j}F_j(2,1) + c_{2j}F_j(2,2)\bigg) - \bigg(d_{1j}F_j(1,2) - d_{2j}F_j(1,1)\bigg)\Bigg]  \text{, using } (\ref{defining_Fjrs}). \notag
\end{align}

To proceed further, we need the limiting behavior of $c_{1j}, c_{2j}, d_{1j}, d_{2j}, F_j(r,s), r,s \in \{1,2\}$ for $1 \leq j \leq n$. This is established in Lemma \ref{uniformConvergence} and the summary of results is given below:
\begin{equation}\label{uniformResults}
\left\{ \begin{aligned} 
    &\underset{1\leq j\leq n}{\max }|c_{1j} - v_n| \xrightarrow{a.s.} 0; \hspace{5mm}
    \underset{1\leq j\leq n}{\max }|d_{1j} - v_n|\xrightarrow{a.s.} 0;\\
    &\underset{1\leq j\leq n}{\max }|c_{2j} - c_nh_{1n}v_n|\xrightarrow{a.s.} 0; \hspace{5mm}
    \underset{1\leq j\leq n}{\max }|d_{2j} - c_nh_{2n}v_n|\xrightarrow{a.s.} 0;\\
    & \underset{1\leq j\leq n}{\max }|F_j(r,r) - m_{rn}| \xrightarrow{a.s.} 0;\\
    & \underset{1\leq j\leq n}{\max }|F_j(r,s)| \xrightarrow{a.s.} 0  \text{, where } r \neq s. 
\end{aligned} \right.
\end{equation}

For sufficiently large $n$ and $k=1,2$, we have the following observations.
\begin{enumerate}
    \item Remark (\ref{remark_vn_bound}) established an upper bound for $|v_n|$.
    \item $|h_{kn}|$ is bounded above by (\ref{h_nBound}).
    \item $|m_{kn}|$ is bounded above using (\ref{R5}) and Lemma \ref{normBound} as shown below:
\begin{align}\label{mn_bound}
|m_{kn}| = \bigg|\frac{1}{n}\operatorname{trace}\{\Sigma_{kn}\Bar{Q}M_nQ\}\bigg| \leq \bigg(\frac{1}{n}\operatorname{trace}(\Sigma_{kn})\bigg) ||\Bar{Q}M_nQ||_{op} \leq \frac{BC_0}{\Re ^2(z)}.
\end{align}
\end{enumerate}

Using the above bounds with (\ref{uniformResults}) and applying Lemma \ref{lA.3}, we get the following results:
\begin{enumerate}
    \item $\underset{1 \leq j \leq n}{\max}|c_{1j}F_j(2,1)| \xrightarrow{a.s.} 0$; \hspace{5mm} $\underset{1 \leq j \leq n}{\max}|d_{1j}F_j(1,2)| \xrightarrow{a.s.} 0$ \text{, and}
    \item $\underset{1 \leq j \leq n}{\max}|c_{2j}F_j(2,2) - c_nv_nh_{1n}m_{2n}| \xrightarrow{a.s.} 0$; \hspace{5mm} $\underset{1 \leq j \leq n}{\max}|d_{2j}F_j(1,1) - c_nv_nh_{2n}m_{1n}| \xrightarrow{a.s.} 0$.
\end{enumerate}

With the above results and applying Lemma \ref{lA.4} on (\ref{T2_simplification}) gives 
\begin{align}\label{Term2_simplification2}
|Term_2 - v_nh_{1n}m_{1n}-v_nh_{2n}m_{2n}| \xrightarrow{a.s.} 0.
\end{align}

Now note that 
\begin{align*}
    v_nh_{1n}m_{2n} &= \frac{n}{p} \dfrac{c_nh_{1n}}{1 + c_n^2h_{1n}h_{2n}}\frac{1}{n}\operatorname{trace}\{\Sigma_{2n} \Bar{Q}M_nQ\} \text{, by definitions } (\ref{defining_vn}), (\ref{defining_mkn}) \\
    &= \frac{1}{p} \rho_2(c_n\textbf{h}n)\operatorname{trace}\{\Sigma_{2n}\Bar{Q}M_nQ\}\\
    &=\frac{1}{p} \operatorname{trace}\{\rho_2(c_n\textbf{h}n)\Sigma_{2n}\Bar{Q}M_nQ\},
\end{align*}
where, the last equality follows from definition (\ref{defining_rho}). Similarly, we have $$v_nh_{2n}m_{1n} = \frac{1}{p} \operatorname{trace}\{\rho_1(c_n\textbf{h}_n)\Sigma_{1n}\Bar{Q}M_nQ\}.$$
Finally from Lemma \ref{rho_imp_result} and (\ref{mn_bound}), we get
\begin{align}\label{rho1_rho2}
&\bigg|\frac{1}{p} \operatorname{trace}\{\rho_2(c_n\textbf{h}_n)\Sigma_{2n}\Bar{Q}M_nQ\} - \frac{1}{p} \operatorname{trace}\{\rho_2(\mathbb{E}c_n\textbf{h}_n)\Sigma_{2n}\Bar{Q}M_nQ\}\bigg| \xrightarrow{a.s.} 0 \text{, and}\\
&\bigg|\frac{1}{p} \operatorname{trace}\{\rho_1(c_n\textbf{h}_n)\Sigma_{1n}\Bar{Q}M_nQ\} - \frac{1}{p} \operatorname{trace}\{\rho_1(\mathbb{E}c_n\textbf{h}_n)\Sigma_{1n}\Bar{Q}M_nQ\}\bigg| \xrightarrow{a.s.} 0\notag.
\end{align}

Combining (\ref{T2_simplification}), (\ref{Term2_simplification2}) and (\ref{rho1_rho2}), we get
\begin{align*}
    &\absmod{Term_2 - \frac{1}{p} \operatorname{trace}\{\bigg(\rho_1(\mathbb{E}c_n\textbf{h}_n)\Sigma_{1n} + \rho_2(\mathbb{E}c_n\textbf{h}_n)\Sigma_{2n}\bigg)\Bar{Q}M_nQ\}} \xrightarrow{a.s.} 0\\
    \implies & \absmod{Term_2 - \frac{1}{p}\operatorname{trace}\{\bigg(zI_p -zI_p + \rho_1(\mathbb{E}c_n\textbf{h}_n)\Sigma_{1n}+\rho_2(\mathbb{E}c_n\textbf{h}_n)\Sigma_{2n}\bigg)\Bar{Q}M_nQ\}} \xrightarrow{a.s.} 0\\
    \implies & |Term_2 - \frac{1}{p} \operatorname{trace}\{(F(z) + zI_p)\Bar{Q}M_nQ\}| \xrightarrow{a.s.} 0\\
    \implies& |Term_2 - Term_1| \xrightarrow{a.s.} 0.
\end{align*}
This concludes the proof.
\end{proof}

\section{Proof of Theorem \ref{Existence_A12}}\label{sec:ProofExistence_A12}

\begin{proof}
By Theorem \ref{CompactConvergence}, every sub-sequence of $\{\textbf{h}_n(\cdot)\}_{n=1}^\infty$ has a further sub-sequence that converges uniformly in each compact subset of $\mathbb{C}_L$. Let $\textbf{h}^\infty(\cdot) = (h_1^\infty(\cdot), h_2^\infty(\cdot))$ be one such subsequential limit corresponding to the sub-sequence $\{\textbf{h}_{n_m}(\cdot)\}_{m=1}^\infty$. Additionally, due to (\ref{totalVariation}) and (\ref{firstSpectralBound}), $\{\textbf{h}_{n_m}(\cdot)\}_{m=1}^\infty$ satisfies the conditions of Theorem \ref{Grommer}. Therefore, it turns out that $h_k^\infty(\cdot);k=1,2$ are themselves Stieltjes Transforms of some measures on the imaginary axis. By (\ref{Property1}), for any $z \in \mathbb{C}_L$, we have
\begin{align}\label{min_real_positive1}
  \underset{k=1,2}{\min}\Re(h_k^\infty(z)) > 0.
\end{align}

Fix $z \in \mathbb{C}_L$. Consider the subsequences of $\textbf{h}_n$ (see \ref{defining_hn}), $\tilde{\textbf{h}}_n$ (see \ref{defining_hn_tilde}), $\tilde{\tilde{\textbf{h}}}_n$ (see \ref{defining_hn_tilde2}), $c_n = p/n$ and $H_n$ (see \ref{defining_Hn}) along the subsequence $\{n_m\}_{m=1}^\infty$. For simplicity, we denote them as follows:
\begin{enumerate}
    \item $\textbf{g}_{m} = (g_{1m}, g_{2m}) = \textbf{h}_{n_m} = (h_{1,n_m}, h_{2,n_m})$,
    \item $\Tilde{\textbf{g}}_{m} = (\tilde{g}_{1m}, \tilde{g}_{2m})= \Tilde{\textbf{h}}_{n_m} = (\tilde{h}_{1,n_m}, \tilde{h}_{2,n_m})$,
    \item $\Tilde{\Tilde{\textbf{g}}}_{m} = (\tilde{\tilde{g}}_{1m}, \tilde{\tilde{g}}_{2m}) = \Tilde{\Tilde{\textbf{h}}}_{n_m} = (\tilde{\Tilde{h}}_{1,n_m},\tilde{\tilde{h}}_{2,n_m})$,
    \item $d_m = c_{n_m}$, \text{ and}
    \item $G_m = H_{n_m} = JESD(\Sigma_{1,n_m}, \Sigma_{2,n_m})$.
\end{enumerate}

With the above definitions, for $k=1,2,$ we have  $g_{km}(z) \xrightarrow{a.s.} h_k^\infty(z)$ since, $\textbf{h}^\infty(\cdot)$ is a subsequential limit. Therefore, using (\ref{Consequence_DetEqv}), we have  
$$|\Tilde{g}_{km}(z) - h_k^\infty(z)| \leq |\Tilde{g}_{km}(z) - g_{km}(z)| + |g_{km}(z) - h_k^\infty(z)| \rightarrow 0.$$
In other words, we have
\begin{align}\label{LHS_limit}
    \Tilde{g}_{k,m}(z) \rightarrow h_k^\infty(z).
\end{align}

From Lemma \ref{hn_tilde_tilde2} and (\ref{simplifying_hn_tilde}), we have
\begin{align}\label{gm_tilde_tilde2}
&\Tilde{g}_{k,m}(z) - \Tilde{\Tilde{g}}_{k,m}(z) \rightarrow 0 \notag\\
\implies &\displaystyle \Tilde{g}_{k,m}(z)- \int\dfrac{\lambda_k dG_m(\boldsymbol{\lambda})}{-z + \boldsymbol{\lambda}^T\boldsymbol{\rho}(d_m\Tilde{\textbf{g}}_{m}(z))} \longrightarrow 0 \notag\\
\implies &\Tilde{g}_{k,m}(z) -\int\dfrac{\lambda_k d\{G_m(\boldsymbol{\lambda})-H(\boldsymbol{\lambda})\}}{-z +\boldsymbol{\lambda}^T\boldsymbol{\rho}(d_m\Tilde{\textbf{g}}_{m}(z))} -
\int\dfrac{\lambda_k dH(\boldsymbol{\lambda})}{-z + \boldsymbol{\lambda}^T\boldsymbol{\rho}(d_m\Tilde{\textbf{g}}_{m}(z))} \longrightarrow 0.
\end{align}

For large $m$, the common integrand in the second and third terms of (\ref{gm_tilde_tilde2}) can be bounded above as follows:
\begin{align}\label{integrandBound}
    &\bigg|\frac{\lambda_1}{-z + \boldsymbol{\lambda}^T\boldsymbol{\rho}(d_m\Tilde{\textbf{g}}_{m})}\bigg|
    \leq \frac{|\lambda_1|}{|\Re(-z + \boldsymbol{\lambda}^T\boldsymbol{\rho}(d_m\Tilde{\textbf{g}}_{m}))|}
    \leq \frac{|\lambda_1|}{|\Re(\lambda_1\rho_1(d_m\Tilde{\textbf{g}}_{m}))|} = \frac{1}{\Re(\rho_1(d_m\Tilde{\textbf{g}}_{m}))
    } \xrightarrow{} \frac{1}{\Re(\rho_1(c\textbf{h}^\infty))}.
\end{align}

The limit in (\ref{integrandBound}) follows upon observing that $\Re(\rho_1(c\textbf{h}^\infty)) > 0$ because of (\ref{min_real_positive1}) and (\ref{real_of_rho1}). Next note that $d_m\Tilde{g}_{k,m} = c_{k_m}\Tilde{h}_{k,n_m} \rightarrow ch_k$. By continuity of $\rho_1(\cdot)$ at $c\textbf{h}^\infty$, we have $\rho_1(d_m\Tilde{\textbf{g}}_{m}) \rightarrow \rho_1(c\textbf{h}^\infty)$.

Similarly, we also have
\begin{align}\label{integrandBound5}
    &\bigg|\frac{\lambda_2}{-z + \boldsymbol{\lambda}^T\boldsymbol{\rho}(d_m\Tilde{\textbf{g}}_{m})}\bigg|
    \leq \frac{|\lambda_2|}{|\Re(-z + \boldsymbol{\lambda}^T\boldsymbol{\rho}(d_m\Tilde{\textbf{g}}_{m}))|}
    \leq \frac{|\lambda_2|}{|\Re(\lambda_2\rho_2(d_m\Tilde{\textbf{g}}_{m}))|} = \frac{1}{\Re(\rho_2(d_m\Tilde{\textbf{g}}_{m}))
    } \xrightarrow{} \frac{1}{\Re(\rho_2(c\textbf{h}^\infty))}.
\end{align}

So the second term of (\ref{gm_tilde_tilde2}) can be made arbitrarily small as $G_m \xrightarrow{d} H$. Applying D.C.T. in the third term of (\ref{gm_tilde_tilde2}) and using (\ref{LHS_limit}), we get 
\begin{align}\label{subsequentialLimit}
h_k^\infty(z) = \displaystyle\int\dfrac{\lambda_k dH(\boldsymbol{\lambda})}{-z + \boldsymbol{\lambda}^T\boldsymbol{\rho}(c\textbf{h}^\infty(z))}.
\end{align}

Thus any subsequential limit ($h_k^\infty(z) \in \mathbb{C}_R$) satisfies (\ref{h_main_eqn}). By Theorem \ref{Uniqueness}, all these subsequential limits must coincide, which we will denote as $\textbf{h}^\infty=(h_1^{\infty},h_2^\infty)$ going forward. In particular, we have shown that 
\begin{align}\label{}
    h_{kn}(z) \xrightarrow{} h_k^\infty(z).
\end{align}
and $h_{k}(\cdot)$ are Stieltjes Transforms of measures on the imaginary axis.

We now show that $s_n(z) \xrightarrow{a.s.} s_{F}(z)$ where $s_{F}(z)$ is defined in (\ref{s_main_eqn}). From Theorem \ref{DeterministicEquivalent}, we have 
$$|s_n(z) - \frac{1}{p}\operatorname{trace}(\bar{Q}(z))| \xrightarrow{a.s.} 0.$$ 
Therefore, all that remains is to show that 
$$\displaystyle \absmod{\frac{1}{p}\operatorname{trace}(\bar{Q}(z)) - \int \frac{dH(\boldsymbol{\lambda})}{-z + \boldsymbol{\lambda}^T\boldsymbol{\rho}(c\textbf{h}^\infty(z))}} \rightarrow 0.$$
By $\boldsymbol{T}_3$ of Theorem \ref{mainTheorem}, we have
\begin{align}
    \frac{1}{p}\operatorname{trace}(\bar{Q}(z)) &= \int \frac{dH_n(\boldsymbol{\lambda})}{-z + \boldsymbol{\lambda}^T\boldsymbol {\rho}(c_n\mathbb{E}\textbf{h}_n(z))}
    = \int \frac{d\{H_n(\boldsymbol{\lambda})-H(\boldsymbol{\lambda})\}}{-z + \boldsymbol{\lambda}^T\boldsymbol {\rho}(c_n\mathbb{E}\textbf{h}_n(z))} + \int \frac{dH(\boldsymbol{\lambda})}{-z + \boldsymbol{\lambda}^T\boldsymbol {\rho}(c_n\mathbb{E}\textbf{h}_n(z))}.
\end{align}
The common integrand in both the terms is bounded by $1/|\Re(z)|$. Since $H_n \xrightarrow{d} H$, the second term goes to $0$. Applying D.C.T. in the second term and using Lemma \ref{ConcetrationLemma}, we get
\begin{align}
    \underset{n \rightarrow \infty}{\lim} \int \frac{dH(\boldsymbol{\lambda})}{-z + \boldsymbol{\lambda}^T\boldsymbol {\rho}(c_n\mathbb{E}\textbf{h}_n(z))}  = \int \frac{dH(\boldsymbol{\lambda})}{-z + \boldsymbol{\lambda}^T\boldsymbol {\rho}(c\textbf{h}^\infty(z))} = s_F(z).
\end{align}
Therefore, $s_n(z) \xrightarrow{a.s.} s_{F}(z)$. This establishes the equivalence between $(\ref{s_main_eqn})$ and (\ref{s_main_eqn_1}). From (\ref{h_nBound}), for sufficiently large $n$, we have $|h_{kn}(z)| \leq {C_0}/{|\Re(z)|}$. Thus for $y > 0$, $|h_k^{\infty}(-y)| \leq {C_0}/{|y|}$ and $\underset{y \rightarrow \infty}{\lim}h_k^{\infty}(-y) = 0$. This implies that 
\begin{align}
\underset{y \rightarrow +\infty}{\lim}ys_{F}(-y) = 1 - \frac{2}{c} + \underset{y \rightarrow \infty}{\lim}\frac{2}{c}\bigg(\frac{1}{1+c^2h_1^{\infty}(-y)h_2^{\infty}(-y)}\bigg) = 1.
\end{align}

Since $s_{F}(.)$ satisfies the necessary and sufficient condition from Proposition \ref{GeroHill}, it is the Stieltjes transform of some probability distribution. By Proposition \ref{GeroHill}, this underlying measure $F$ is the LSD of $F^{S_n}$. This completes the proof of Theorem \ref{mainTheorem} under Assumptions \ref{A12}.
\end{proof}

\section{Proof of Theorem \ref{Existence_General}}

\subsection{Proof of Step8 and Step9}\label{Step_8_9}
\begin{proof}    
Since Theorem \ref{mainTheorem} holds for $\Tilde{U}_n$, we have $F^{\Tilde{U}_n} \xrightarrow{d} F^{\tau}$ for some LSD $F^{\tau}$ and for $z \in \mathbb{C}_L$, there exists functions $s^\tau(z)$ and $\textbf{h}^\tau(z)$ satisfying (\ref{s_main_eqn}) and (\ref{h_main_eqn}) with $H^\tau$ replacing $H$ and mapping $\mathbb{C}_L$ to $\mathbb{C}_R$ and analytic on $\mathbb{C}_L$. We have to show existence of analogous quantities for the sequence $\{F^{S_n}\}_{n=1}^\infty$.

 First, assume that H has a bounded support. If $\tau_0 > 0$ is such that $H(\tau_0, \tau_0) = 1$, then $H^{\tau}(s,t) = H(s,t)$ for all $\tau \geq \tau_0$. By Theorem \ref{Existence_A12}, $\textbf{h}^{\tau}(z)=(h_1^\tau(z),h_2^\tau(z))$ must be the same for all large $\tau$. Hence $s^{\tau}(z)$ and in turn $F^{\tau}(.)$ must also be the same for all large $\tau$.  Denote this common LSD by $F$ and the common value of $\textbf{h}^\tau$ and $s^\tau$ by $\textbf{h}^\infty$ and $s_F$ respectively. This proves Theorem \ref{mainTheorem} when $H$ has a bounded support. 

Now we analyze the case where $H$ has unbounded support. We need to show there exist functions $\textbf{h}^\infty$, $s_F$ that satisfy equations (\ref{s_main_eqn}) and (\ref{h_main_eqn}) and an LSD $F$ serving as the almost sure weak limit of the ESDs of $\{S_n\}_{n=1}^\infty$.

 We will show that for $k \in \{1,2\}$, $\mathcal{H}_k = \{h_k^\tau: \tau > 0\}$ forms a normal family. Following arguments similar to those used in Theorem \ref{CompactConvergence}, let $K \subset \mathbb{C}_L$ be an arbitrary compact subset. Then $u_0 > 0$ where $u_0 := \inf\{|\Re(z)|: z \in K\}$. For arbitrary $z \in K$, using (\ref{R5}) and (\ref{firstSpectralBound}), for sufficiently large $n$, we have
    \begin{align}\label{hn_tau_Bound}
        |h_{kn}^\tau(z)| &= \frac{1}{p}|\operatorname{trace}\{\Sigma_{kn}^\tau Q\}| \leq \bigg(\frac{1}{p}\operatorname{trace}(\Sigma_{kn}^\tau)\bigg) ||Q||_{op} 
        \leq \frac{C_0}{|\Re(z)|} \leq \frac{C_0}{u_0}.
    \end{align}

By Theorem \ref{Existence_A12}, for any $\tau > 0$, $h_k^\tau(z)$ is the uniform limit of $h_{kn}^\tau(z) := \frac{1}{n}\operatorname{trace}\{\Sigma_{kn}^\tau Q(z)\}$. Therefore, for $z \in K$,    
\begin{align}\label{h_tauBound}
|h_k^\tau(z)|\leq \frac{C_0}{|\Re(z)|} \leq \frac{C_0}{u_0}.
\end{align}
Therefore as a consequence of \textit{Montel's theorem}, any subsequence of $\mathcal{H}_k$ has a further convergent subsequence that converges uniformly on compact subsets of $\mathbb{C}_L$. 

Let $\{\textbf{h}^{\tau_m}(\cdot)\}_{m=1}^\infty= \{h_1^{\tau_m}(\cdot), h_2^{\tau_m}(\cdot)\}_{m=1}^\infty$ be a convergent subsequence with $\textbf{h}^\infty(z)=(h_1^\infty(\cdot), h_2^\infty(z))$ as the subsequential limit, where $\tau_m \rightarrow \infty$ as $m \rightarrow \infty$. By Theorem \ref{Existence_A12}, for any $\tau>0$, $h_k^\tau;k=1,2$ are Stieltjes transforms of measures on the imaginary axis. Moreover, the underlying measures of these transforms have uniformly bounded total variation due to (\ref{firstSpectralBound}). Therefore, by Theorem \ref{Grommer}, we deduce that $h_k^\infty(\cdot);k=1,2$ themselves must also be Stieltjes transforms of measures on the imaginary axis. By (\ref{Property1}), for all $z \in \mathbb{C}_L$, we must have 
\begin{align}\label{min_real_positive2}
  \min\{\Re(h_{1}^\infty(z)), \Re(h_{2}^\infty(z))\} > 0.
\end{align}

Now fix $z \in \mathbb{C}_L$. By (\ref{real_of_rho1}), (\ref{real_of_rho2}) and the fact that $\Re(h_{k}^\infty(z)) > 0$, we have $\Re(\boldsymbol{\rho}_k(c\textbf{h}^\infty)) > 0$ for $k=1,2$. Therefore, by continuity of $\boldsymbol{\rho}(\cdot,\cdot)$ at $c\textbf{h}^\infty$, 
\begin{align}\label{integrandBound3}
\bigg|\frac{\lambda_1}{-z + \boldsymbol{\lambda}^T\boldsymbol{\rho}(c\textbf{h}^{\tau_m})}\bigg|
    \leq \frac{|\lambda_1|}{|\Re(-z + \boldsymbol{\lambda}^T\boldsymbol{\rho}(c\textbf{h}^{\tau_m}))|}
    \leq \frac{|\lambda_1|}{|\Re(\lambda_1\rho_1(c\textbf{h}^{\tau_m}))|} = \frac{1}{\Re(\rho_1(c\textbf{h}^{\tau_m}))} \xrightarrow{} \frac{1}{\Re(\rho_1(c\textbf{h}^\infty))}  < \infty.
\end{align}
as $m \rightarrow \infty$. Now, by Theorem \ref{Existence_A12}, $(\textbf{h}^{\tau_m}, H^{\tau_m})$ satisfy the below equation.

\begin{align*}
   \textbf{h}^{\tau_m}(z) &= \int \dfrac{\boldsymbol{\lambda} dH^{\tau_m}(\boldsymbol{\lambda})}{-z + \boldsymbol{\lambda}^T\boldsymbol{\rho}(c\textbf{h}^{\tau_m})}
  = \int \dfrac{\boldsymbol{\lambda} d\{H^{\tau_m}(\boldsymbol{\lambda}) - H(\boldsymbol{\lambda})\}}{-z + \boldsymbol{\lambda}^T\boldsymbol{\rho}(c\textbf{h}^{\tau_m})} + \int \dfrac{\boldsymbol{\lambda} dH(\boldsymbol{\lambda})}{-z +\boldsymbol{\lambda}^T\boldsymbol{\rho}(c\textbf{h}^{\tau_m})}.
\end{align*}
Note that, the first term of the last expression can be made arbitrarily small since the integrand is bounded by (\ref{integrandBound3}) and $H^{\tau_m} \xrightarrow{d} H$. The same bound on the integrand also allows us to apply D.C.T. in the second term, thus giving us
\begin{align}
  &\underset{m \rightarrow \infty}{\lim} \textbf{h}^{\tau_m}(z) = \underset{m \rightarrow \infty}{\lim} \int \dfrac{\boldsymbol{\lambda} dH(\boldsymbol{\lambda})}{-z +\boldsymbol{\lambda}^T\boldsymbol{\rho}(c\textbf{h}^{\tau_m})} \notag\\
  \implies &  \textbf{h}^\infty(z) = \int \dfrac{\boldsymbol{\lambda}dH(\boldsymbol{\lambda})}{-z +\boldsymbol{\lambda}^T\boldsymbol{\rho}(c\textbf{h}^\infty(z))}.\label{subsequentialLimit2}
\end{align}

Now $\{\tau_m\}_{m=1}^\infty$ is a further subsequence of an arbitrary subsequence and $\{\textbf{h}^{\tau_m}(z)\}$ converges to $\textbf{h}^\infty(z) \in \mathbb{C}_R$ that satisfies (\ref{h_main_eqn}). By Theorem \ref{Uniqueness}, all these subsequential limits coincide, which we will denote by $\textbf{h}^\infty(z) = (h_1^\infty(z), h_2^\infty(z))$. 

Now we will show that $s^\tau(z) \rightarrow s_F(z)$ as $\tau \rightarrow \infty$ where $s_F(\cdot)$ is given by (\ref{s_main_eqn}). Note that,
\begin{align}\label{s_tau_limit}
    &|s^\tau(z) - s_F(z)|\\ \notag
    &= \absmod{\int \dfrac{dH^{\tau}(\boldsymbol{\lambda})}{-z +\boldsymbol{\lambda}^T\boldsymbol{\rho}(c\textbf{h}^{\tau}(z))} - \int \dfrac{dH(\boldsymbol{\lambda})}{-z +\boldsymbol{\lambda}^T\boldsymbol{\rho}(c\textbf{h}^\infty(z))}}\\ \notag
    &\leq \absmod{\int \dfrac{d\{H^{\tau}(\boldsymbol{\lambda}) - dH(\boldsymbol{\lambda})\}}{-z +\boldsymbol{\lambda}^T\boldsymbol{\rho}(c\textbf{h}^{\tau}(z))}} + \int \absmod{\dfrac{1}{-z +\boldsymbol{\lambda}^T\boldsymbol{\rho}(c\textbf{h}^{\tau}(z))} - \dfrac{1}{-z +\boldsymbol{\lambda}^T\boldsymbol{\rho}(c\textbf{h}^{\infty}(z))}}dH(\boldsymbol{\lambda}).
\end{align}

Note that $\textbf{h}^\infty(z),\textbf{h}^\tau(z) \in \mathbb{C}_R^2$. In particular, this implies that the integrands of the first and second terms in (\ref{s_tau_limit}) are bounded by $1/|\Re(z)|$ and, $2/|\Re(z)|$ respectively. The first term can be made arbitrarily small by choosing $\tau$ to be very large, since $H^\tau \xrightarrow{d} H$. Note that $\textbf{h}^\tau(z) \rightarrow \textbf{h}^\infty(z)$ and $\boldsymbol{\rho}$ is analytic at $c\textbf{h}^\infty(z) \in \mathbb{C}_R^2$. Thus, applying D.C.T., we get
\begin{align}
    \underset{\tau \rightarrow \infty}{\lim} \int \absmod{\dfrac{1}{-z +\boldsymbol{\lambda}^T\boldsymbol{\rho}(c\textbf{h}^{\tau}(z))} - \dfrac{1}{-z +\boldsymbol{\lambda}^T\boldsymbol{\rho}(c\textbf{h}^{\infty}(z))}}dH(\boldsymbol{\lambda}) = 0.
\end{align}
Thus, we have proved that $s^\tau(z) \xrightarrow{} s_F(z)$ and we have established the equivalence between $(\ref{s_main_eqn})$ and (\ref{s_main_eqn_1}).

From (\ref{h_tauBound}), $|h_k^\tau(z)| \leq {C_0}/{|\Re(z)|}$. Thus, $|h_k^\infty(z)| \leq {C_0}/{|\Re(z)|}$ implying that $\underset{y \rightarrow \infty}{\lim}h_k^\infty(-y) = 0$. Therefore, 
\begin{align*}
    \underset{y \rightarrow +\infty}{\lim}ys_F(-y) = \bigg(1 - \frac{2}{c}\bigg) + \underset{y \rightarrow \infty}{\lim}\frac{2}{c}\bigg(\frac{1}{1 +c^2h_1^\infty(-y)h_2^\infty(-y)}\bigg)= 1.
\end{align*}
To conclude, we have \begin{itemize}
    \item $\textbf{h}^\tau \rightarrow \textbf{h}^\infty$ and $s^\tau \rightarrow s_F$,
    \item $h_1^\infty, h_2^\infty$ satisfy (\ref{h_main_eqn}) and is a Stieltjes transform of a measure over the imaginary axis, and
    \item $s_F$ satisfies the conditions of Proposition \ref{GeroHill} for a Stieltjes Transform of a probability measure on the imaginary axis.
\end{itemize}
\end{proof}

\subsection{Proof of Step10}\label{Step10}
\begin{proof}
\textbf{Impact of spectral truncation of $\Sigma$ matrices:}

Let $A = \frac{1}{n}Z_1Z_2^*,B = \frac{1}{n}Z_2Z_1^*, P = \Lambda_{1n}, Q = \Lambda_{2n}, R = \Lambda_{1n}^\tau, S = \Lambda_{2n}^\tau$. We have
\begin{itemize}
    \item $S_n = \frac{1}{n}(\Lambda_{1n} Z_1Z_2^*\Lambda_{2n} - \Lambda_{2n}Z_2Z_1^*\Lambda_{1n}) = PAQ - QBP$, and
    \item $T_n = \frac{1}{n}(\Lambda_{1n}^\tau Z_1Z_2^*\Lambda_{2n}^\tau - \Lambda_{2n}^\tau Z_2Z_1^*\Lambda_{1n}^\tau)= RAS - SBR$.
\end{itemize}

Finally, using (\ref{R1}), (\ref{R3}), we observe that,
\begin{align*}
  ||F^{S_n} - F^{T_n}||_{im} &\leq \dfrac{1}{p}\operatorname{rank}(S_n - T_n)\\
  &\leq \frac{1}{p}\operatorname{rank}(PAQ - RAS) + \operatorname{rank}(QBP-SBR)\\
  &\leq \frac{2}{p}\bigg(\operatorname{rank}(\Lambda_{1n}-\Lambda_{1n}^\tau) + \operatorname{rank}(\Lambda_{2n}-\Lambda_{2n}^\tau)\bigg)\\
  &= 2(1 - F^{\Sigma_{1n}}(\tau))+2(1 - F^{\Sigma_{2n}}(\tau))\\
  &\longrightarrow \quad 2(1-H_1(\tau)) + 2(1-H_2(\tau)) \xrightarrow{\text{as } \tau \rightarrow \infty} 0,
\end{align*}
where $H_1$ and $H_2$ are the marginal distributions of $H$. Here we used the fact that $\tau > 0$ was chosen such that $(\tau, \tau)$ is a continuity point of $H$.

\textbf{Impact of truncation of the innovation entries:}

Now we will show that $||F^{T_n} - F^{U_n}||_{im} \xrightarrow{a.s.} 0$. We have $T_n = \frac{1}{n}(\Lambda_{1n}^\tau Z_1Z_2^*\Lambda_{2n}^\tau - \Lambda_{2n}^\tau Z_2Z_1^*\Lambda_{1n}^\tau)$ and $U_n = \frac{1}{n}(\Lambda_{1n}^\tau \hat{Z}_1\hat{Z}_2^*\Lambda_{2n}^\tau - \Lambda_{2n}^\tau \hat{Z}_2\hat{Z}_1^*\Lambda_{1n}^\tau)$.

Using (\ref{R1}), (\ref{R3}), we have
\begin{align}\label{Tn_Un}
    ||F^{T_n} - F^{U_n}||_{im} &\leq \frac{1}{p}\operatorname{rank}(T_n - U_n) \\
    &=\frac{1}{p}\operatorname{rank}\bigg(\frac{1}{n}\Lambda_{1n}^\tau(Z_1Z_2^* - \hat{Z}_1\hat{Z}_2^*)\Lambda_{2n}^\tau - \frac{1}{n}\Lambda_{2n}^\tau(Z_2Z_1^* - \hat{Z}_2\hat{Z}_1^*)\Lambda_{1n}^\tau \bigg) \notag\\
    &\leq \frac{1}{p}\operatorname{rank}(Z_1Z_2^* - \hat{Z}_1\hat{Z}_2^*) + \frac{1}{p}\operatorname{rank}(Z_2Z_1^* - \hat{Z}_2\hat{Z}_1^*) \notag\\
    &= \frac{2}{p}\operatorname{rank}(Z_1Z_2^* - \hat{Z}_1\hat{Z}_2^*) \notag\\
    &\leq \frac{2}{p}\bigg(\operatorname{rank}(Z_1 - \hat{Z}_1) + \operatorname{rank}(Z_2 - \hat{Z}_2)\bigg).
\end{align}

For $k=1,2$, define $I_{ij}^{(k)} := \mathbbm{1}_{\{z_{ij}^{(k)} \neq \Hat{z}_{ij}^{(k)}\}} = \mathbbm{1}_{\{|z_{ij}^{(k)}| > n^b\}}$ where $b$ is defined in Assumption \ref{A12}. Using (\ref{R2}), we have 
$$\operatorname{rank}(Z_k - \Hat{Z}_k) \leq \sum_{ij} I_{ij}^{(k)}.$$ 
Note that
\begin{align*}
 &\mathbb{P}(I_{ij}^{(k)} = 1)
=\mathbb{P}(|z_{ij}^{(k)}| > n^b)
\leq \dfrac{\mathbb{E}|z_{ij}^{(k)}|^{2+\eta_0}}{n^{b(2+\eta_0)}}
\leq \dfrac{M_{2+\eta_0}}{n^{b(2+\eta_0)}}.
\end{align*}
Since $\frac{1}{2+\eta_0}<b<\frac{1}{2}$, we have 
\begin{align*}
\dfrac{1}{p}\sum_{i,j}\mathbb{P}(I_{ij}^{(k)} = 1)
\leq \frac{npM_{2+\eta_0}}{pn^{b(2+\eta_0)}}
     \longrightarrow 0.
\end{align*}
\color{black}
Also, we have $\operatorname{Var}I_{ij}^{(k)} \leq \mathbb{P}(I_{ij}^{(k)}=1)$. For arbitrary $\epsilon > 0$, we must  have $\sum_{i,j}\operatorname{Var}I_{ij}^{(k)} \leq p\epsilon/2$ for large enough $n$. Finally, we use Bernstein's Inequality to get the following bound:
\begin{align*}
    \mathbb{P}\bigg(\dfrac{1}{p}\sum_{i,j}I_{ij}^{(k)} > \epsilon\bigg)
    &\leq \hspace{1mm} \mathbb{P}\bigg(\sum_{i,j}(I_{ij}^{(k)} - \mathbb{P}(I_{ij}^{(k)} = 1)) > \dfrac{p\epsilon}{2}\bigg)\\
    &\leq \hspace{1mm} 2\exp{\bigg(-\dfrac{p^2\epsilon^2/4}{2(p\epsilon/2 + \sum_{i,j} \operatorname{Var}I_{ij}^{(k)})}\bigg)}\\
    &\leq \hspace{1mm} 2\exp{\bigg(-\dfrac{p^2\epsilon^2/4}{2(p\epsilon/2 + p\epsilon/2)}\bigg)} = 2\exp{\bigg(-\dfrac{p\epsilon}{8}\bigg)}.
\end{align*}
By Borel Cantelli lemma, $\frac{1}{p}\sum_{ij}I_{ij}^{(k)} \xrightarrow{a.s.}0$ and thus $\frac{1}{p}\operatorname{rank}(Z_k-\hat{Z}_k) \xrightarrow{a.s.} 0$. Combining this with (\ref{Tn_Un}), we have $||F^{T_n} - F^{U_n}||_{im} \xrightarrow{a.s.} 0$.

\textbf{Impact of centering of entries of $\hat{Z}$ matrices:}

The last result to be proved is $||F^{U_n}- F^{\Tilde{U}_n}||_{im} \xrightarrow{a.s.} 0$. Define $\check{Z}_k = (\check{z}^{(k)}_{ij}) := (z_{ij}^{(k)}I_{ij}^{(k)})$ for $k \in \{1,2\}$. Then, 

\begin{align}\label{rank_convergence}
\dfrac{1}{p}\operatorname{rank}\check{Z}_k = \dfrac{1}{p}\operatorname{rank}(Z_k -\hat{Z}_k)\leq \dfrac{1}{p}\sum_{i,j}\mathbb{P}(I_{ij}^{(k)} = 1) \xrightarrow{a.s.} 0.
\end{align}

Finally, from (\ref{R1}), (\ref{R3}), we have
\begin{align}
    ||F^{U_n} - F^{\Tilde{U}_n}||_{im}
    &\leq \frac{1}{p} \operatorname{rank}(U_n - \Tilde{U}_n)\\
    &\leq \frac{1}{p}\operatorname{rank}\bigg(\frac{1}{n}\Lambda_{1n}^\tau(\Hat{Z}_1\Hat{Z}_2^* - \Tilde{Z}_1\Tilde{Z}_2^*)\Lambda_{2n}^\tau - \frac{1}{n}\Lambda_{2n}^\tau(\Hat{Z}_2\Hat{Z}_1^* - \Tilde{Z}_2\Tilde{Z}_1^*)\Lambda_{1n}^\tau\bigg) \notag\\
    &\leq \frac{1}{p}\operatorname{rank}(\Hat{Z}_1\Hat{Z}_2^* - \Tilde{Z}_1\Tilde{Z}_2^*) + \frac{1}{p}\operatorname{rank}(\Hat{Z}_2\Hat{Z}_1^* - \Tilde{Z}_2\Tilde{Z}_1^*) \notag\\
    &\leq \frac{2}{p}\operatorname{rank}(\hat{Z}_1 - \Tilde{Z}_1) + \frac{2}{p}\operatorname{rank}(\Hat{Z}_2 - \Tilde{Z}_2) \notag\\
    &= \frac{2}{p}\operatorname{rank}(\mathbb{E}\Hat{Z}_1) + \frac{2}{p}\operatorname{rank}(\mathbb{E}\Hat{Z}_2) \notag\\
    &= \frac{2}{p}\operatorname{rank}(\mathbb{E}\check{Z}_1) + \frac{2}{p}\operatorname{rank}(\mathbb{E}\Check{Z}_2) \text{, since } \textbf{0} = \mathbb{E}Z_k = \mathbb{E}\hat{Z}_k + \mathbb{E}\check{Z}_k \notag\\
    &\longrightarrow 0 \text{, using } (\ref{rank_convergence}). \notag
\end{align}
\end{proof}

\subsection{Proof of Theorem \ref{pointMass}}\label{sec:ProofOfPointMass}

\begin{proof}
Note that for any $\epsilon > 0$ and $k=1,2$, we have $h_k(-\epsilon) = \overline{h_k(\epsilon)}$ implying that, $h_k(-\epsilon) \in \mathbb{R}$. Also, since $h_k(\mathbb{C}_L) \subset \mathbb{C}_R$, we must have $h_k(-\epsilon) > 0$. 

We will first show that when $c\geq 2/\beta$, we must have $\underset{\epsilon \downarrow 0}{\lim}\hspace{1mm}h_k(-\epsilon) = \infty$. If we assume the contrary, then there exists some $M > 0$ such that for all sufficiently small $\epsilon$, we have  
    \begin{align}\label{h_epsilon_bound}
        h_k(-\epsilon) < M.
    \end{align}
Then, for any sequence $\{\epsilon_n\}_{n=1}^\infty$ with $\epsilon_n \downarrow 0$, we have $|h_k(-\epsilon_n)| < M$ for sufficiently large $n$. So there exists a subsequence $\{n_m\}_{m=1}^\infty$ such that $$\underset{m \rightarrow \infty}{\lim} \textbf{h}(-\epsilon_{n_m}) =  (\underset{m \rightarrow \infty}{\lim} h_1(-\epsilon_{n_m}),\underset{m \rightarrow \infty}{\lim} h_2(-\epsilon_{n_m})) = (\theta_1, \theta_2)$$ 
where, $\theta_k \geq 0$ for $k=1,2$. By Fatou's Lemma, we observe the following inequality:
\begin{align}\label{Fatou}
         &h_k(-\epsilon_{n_m}) = \int \frac{\lambda_k dH(\boldsymbol{\lambda})}{\epsilon_{n_m} + \boldsymbol{\lambda}^T \boldsymbol{\rho}(c\textbf{h}(-\epsilon_{n_m}))}\\
      \implies& \theta_k = \underset{m \rightarrow \infty}{\liminf}\, h_k(-\epsilon_{n_m}) \geq \int \underset{m \rightarrow \infty}{\liminf} \, \frac{\lambda_k dH(\boldsymbol{\lambda})}{\epsilon_{n_m} + \boldsymbol{\lambda}^T \boldsymbol{\rho}(c\textbf{h}(-\epsilon_{n_m}))}.\notag
\end{align}

\textbf{Case1:} \textbf{$\theta_1 = 0 = \theta_2$}: In this case, we get $0 \geq \infty$ from (\ref{Fatou}). 

\textbf{Case2:} \textbf{Exactly one of $\theta_1$ and $\theta_2$ is $0$}: Without loss of generality, let $\theta_1 > 0$ and $\theta_2 = 0$. Then from (\ref{Fatou}), we observe that 
\begin{align}
 \theta_2 = 0 \geq \beta\int \frac{\lambda_2dH_1(\boldsymbol{\lambda})}{\lambda_1\rho_1(c\theta_1, 0) + \lambda_2\rho_2(c\theta_1, 0)} = \frac{\beta}{c\theta_1}\int \frac{\lambda_2dH_1(\boldsymbol{\lambda})}{\lambda_1}.
\end{align}
The expression on the right is either a positive real number or infinity, both of which leads to a contradiction.

\textbf{Case3:} $\theta_1, \theta_2 \in (0, \infty)$:

In this case, for large $m \in \mathbb{N}$, we have 
\begin{align*}
    &\absmod{\frac{\lambda_1}{\epsilon_{n_m} + \lambda_1\rho_1(c\textbf{h}(-\epsilon_{n_m})) + \lambda_2 \rho_2(c\textbf{h}(-\epsilon_{n_m}))}}\\
    \leq& \frac{1}{\rho_1(c\textbf{h}(-\epsilon_{n_m}))}
    = \frac{1 + c^2h_1(-\epsilon_{n_m})h_2(-\epsilon_{n_m})}{ch_2(-\epsilon_{n_m})}
    \leq \frac{2(1+c^2M^2)}{c\theta_2} < \infty.
\end{align*}

Similarly, we have
\begin{align*}
    \absmod{\frac{\lambda_2}{\epsilon_{n_m} + \lambda_1\rho_1(c\textbf{h}(-\epsilon_{n_m})) + \lambda_2 \rho_2(c\textbf{h}(-\epsilon_{n_m}))}}
    \leq \frac{2(1+c^2M^2)}{c\theta_1} < \infty.
\end{align*}

This allows us to use D.C.T. in (\ref{Fatou}) thus leading to:
\begin{align}\label{finite_limit}
    \theta_k &= \beta\int \frac{\lambda_k dH_1(\boldsymbol{\lambda})}{\lambda_1 \rho_2(c\theta_1, c\theta_2) + \lambda_2 \rho_1(c\theta_1, c\theta_2)} \text{ for } k = 1,2\notag\\
    \implies& \theta_1 \rho_2(c\theta_1, c\theta_2) + \theta_2\rho_1(c\theta_1, c\theta_2) = \beta\int \frac{\lambda_1 \rho_2(c\beta_1, cz_2) + \lambda_2\rho_1(c\theta_1, c\theta_2)dH_1(\boldsymbol{\lambda})}{\lambda_1 \rho_2(c\beta_1, c\beta_2) + \lambda_2 \rho_1(c\theta_1, c\theta_2)} \notag\\
    \implies& z_1\rho_1(c\theta_1, c\theta_2) + z_2\rho_2(c\theta_1, c\theta_2) = \beta \notag\\
    \implies& 2c\theta_1\theta_2 = \beta(1 + c^2\theta_1\theta_2)
    \implies c(2-c\beta)\theta_1\theta_2 = \beta.
\end{align}

When $c \geq 2/\beta$, we have a contradiction as the LHS is non-positive but the RHS is positive. Therefore, $\underset{\epsilon \downarrow 0}{\lim}\hspace{1mm}h_k(-\epsilon) = \infty$. Finally, using (\ref{s_main_eqn_1}) and (\ref{inversion_pointMass}), we get
     \begin{align}\label{point_mass_result1}
         \underset{\epsilon \rightarrow 0}{\lim}\epsilon s_{F}(-\epsilon) = 1 - \frac{2}{c} + \underset{\epsilon \rightarrow 0}{\lim}\frac{2}{c}\frac{1}{1+c^2h_1(-\epsilon)h_2(-\epsilon)} = 1 - \frac{2}{c}.
     \end{align}
\begin{remark}\label{remark_E16}
    One implication of this is that the existence of a bound ($M > 0$) for some $c < 2/\beta$ is sufficient to imply that any subsequential limit $\theta=(\theta_1, \theta_2)$ must satisfy $\theta_1\theta_2 = \dfrac{\beta}{c(2-c\beta)}$. 
\end{remark}

Now we show that for $0 < c < 2/\beta$, we have $\underset{\epsilon \downarrow 0}{\lim}\hspace{1mm} h_k(-\epsilon) = \theta_k$, where 
\begin{align}\label{defining_theta_k}
    \theta_k = \frac{1+c^2\theta_1\theta_2}{c}\int \frac{\lambda_k dH(\boldsymbol{\lambda})}{\lambda_1\theta_2 + \lambda_2\theta_1}.
\end{align}
Then, it is clear that $\theta_1, \theta_2$ satisfy
\begin{align}\label{theta_equation}
    c(2 - c\beta)\theta_1\theta_2 = \beta.  
\end{align}

Note that in light of Remark \ref{remark_E16}, all we need is show that $h_k(-\epsilon)$ is bounded. For $k=1,2$ and $\epsilon>0, \textbf{t}=(x, y) \in \mathbb{R}_+^2$, define the functions $G_k(\epsilon, \textbf{t}): \mathbb{R}_+^3 \rightarrow \mathbb{R}_+$ as follows:
\begin{align}\label{defining_G_k}
G_k(\epsilon, \textbf{t}) := \int \frac{\lambda_k dH(\boldsymbol{\lambda})}{\epsilon + \lambda_1 \rho_2(c\textbf{t}) + \lambda_2\rho_1(c\textbf{t})} = \int \frac{\beta \lambda_k dH_1(\boldsymbol{\lambda})}{\epsilon + \boldsymbol{\lambda}^T \boldsymbol{\rho}(c\textbf{t})}.
\end{align} 

By D.C.T, we have $\underset{\epsilon \downarrow 0}{\lim}\hspace{1mm}G_k( \boldsymbol{\theta},\epsilon) = \theta_k$ for $k=1, 2$. This is clear from the arguments presented in (\ref{finite_limit}). The following chain of arguments establishes an upper bound for $h_k(-\epsilon)$ as $\epsilon > 0$ goes to $0$ for $k=1,2$.

\begin{enumerate}
    \item[1] We employ a geometric approach to find the fixed points for the functions $G_1, G_2$. We project the surface of $G_1(z=\epsilon,x,y)$ to the $x-z$ plane to get a curve. The (unique) point (on the $x-$ axis) where this projected curve meets the diagonal $x=z$ is the first coordinate of the fixed point. For the other coordinate, we project $G_2(z=\epsilon, x, y)$ to the $y-z$ plane and find the (unique) point (on the $y-$axis) where the projected curve meets the line $y=z$. 
    \item[2] So, $G_k(\epsilon, \theta_1, \theta_2)$ increases to $\theta_k$ as $\epsilon > 0$ goes to $0$.
    \item[3] Let $\boldsymbol{C_H} =(C_1, C_2)= \int \boldsymbol{\lambda} dH(\boldsymbol{\lambda})  = (\beta \int \lambda_1 dH(\boldsymbol{\lambda}), \beta \int \lambda_2 dH(\boldsymbol{\lambda}))$. Clearly, $0 < C_1, C_2 < \infty$. So, for any $\epsilon>0$, $G_k(\epsilon,0,0) = C_k/\epsilon > 0$.    
    \item[4] We see that at a left neighborhood of (0,0), $G_k$ is above the diagonal and at a right neighborhood of $(\theta_1, \theta_2)$, $G_k$ is below the diagonal.
    \item[5] By continuity of $G_k$, it is clear that $h_k(-\epsilon) < \theta_k$. Therefore, $\theta_k$ is an upper bound for $h_k(-\epsilon)$ for any $0 < c < 2/\beta$.
\end{enumerate}
Now, using the arguments presented in (\ref{finite_limit}) and the subsequent remark, we get (\ref{theta_equation}).  Finally, using (\ref{s_main_eqn_1}) and (\ref{inversion_pointMass}), we get
     \begin{align}
         \underset{\epsilon \rightarrow 0}{\lim}\epsilon s_{F}(-\epsilon) = 1 - \frac{2}{c} + \underset{\epsilon \rightarrow 0}{\lim}\frac{2}{c}\frac{1}{1+c^2h_1(-\epsilon)h_2(-\epsilon)} = 1 - \beta.
     \end{align}

\end{proof}

\subsection{Proof of Theorem \ref{ContinuityGeneral}}\label{sec:ProofOfContinuityGeneral}
\begin{proof}
\textbf{Step1:} First, we prove the continuity of $\textbf{h}(z, H)$ as a function of $H$ for fixed $z \in \mathbb{C}_L$ with $|\Re(z)| > R_0 $ where $R_0$ was defined in Theorem \ref{Uniqueness}.

\textbf{Step2:} Let $H_n \xrightarrow{d} H_\infty$ and denote $\textbf{h}_n(z) = \textbf{h}(z, H_n)$ and $\textbf{h}_\infty(z) = \textbf{h}(z, H_\infty)$. Then, $\textbf{g}_n(z) = \textbf{h}_n(z) - \textbf{h}_\infty(z)$ is analytic over $\mathbb{C}_L$ and from Step1, $\underset{n \rightarrow \infty}{\lim}\textbf{g}_n(z) = \textbf{0}$ for all $z$ with large real component. It is easy to see that $\textbf{g}_n$ are uniformly locally bounded due to (\ref{TechnicalCondition_continuity}). In particular, $\{\textbf{g}_n\}_{n=1}^\infty$ satisfy the conditions of Theorem \ref{VitaliPorter}. So $\{\textbf{g}_n\}_{n=1}^\infty$ converges to an analytic function which is equal to $0$ for all $z \in \mathbb{C}_L$ with large real component. By Identity Theorem, $\underset{n \rightarrow \infty}{\lim}\textbf{g}_n(z) = 0$ for all $z \in \mathbb{C}_L$.

So, all that remains is to prove \textbf{Step1}. Fix $z = -u + \mathbbm{i}v \in \mathbb{C}_L$ such that $u > R_0$. For bi-variate probability distributions $G$ and $H$ on $\mathbb{R}_+^2$, let $\textbf{h}(z)= \textbf{h}(z, H)=(h_1, h_2)$ and $\textbf{g}(z) = \textbf{h}(z, G)=(g_1, g_2)$. Choose $\epsilon > 0$ arbitrarily. We have
\begin{align}
    &|h_1 - g_1| = \absmod{\int \frac{\lambda_1 dH(\boldsymbol{\lambda})}{-z + 
    \boldsymbol{\lambda}^T\boldsymbol{\rho}(c\textbf{h})} - \int \frac{\lambda_1 dG(\boldsymbol{\lambda})}{-z + 
    \boldsymbol{\lambda}^T\boldsymbol{\rho}(c\textbf{g})}}\\
    &\leq \absmod{\int \frac{\lambda_1dH(\boldsymbol{\lambda})}{-z + 
    \boldsymbol{\lambda}^T\boldsymbol{\rho}(c\textbf{h})} - \int \frac{\lambda_1 dG(\boldsymbol{\lambda})}{-z + 
    \boldsymbol{\lambda}^T\boldsymbol{\rho}(c\textbf{h})}} + \absmod{\int \frac{\lambda_1 dG(\boldsymbol{\lambda})}{-z + 
    \boldsymbol{\lambda}^T\boldsymbol{\rho}(c\textbf{h})} - \int \frac{\lambda_1 dG(\boldsymbol{\lambda})}{-z + 
    \boldsymbol{\lambda}^T\boldsymbol{\rho}(c\textbf{g})}} \notag\\
    &= \underbrace{\absmod{\frac{\lambda_1d\{H(\boldsymbol{\lambda})-G(\boldsymbol{\lambda})\}}{-z +\boldsymbol{\lambda}^T\boldsymbol{\rho}(c\textbf{h})}}}_{T_1} + \underbrace{\absmod{\int\frac{\lambda_1 \boldsymbol{\lambda}^T(\boldsymbol{\rho}(c\textbf{g})-\boldsymbol{\rho}(c\textbf{h})) dG(\boldsymbol{\lambda})}{(-z + 
    \boldsymbol{\lambda}^T\boldsymbol{\rho}(c\textbf{g}))(-z + 
    \boldsymbol{\lambda}^T\boldsymbol{\rho}(c\textbf{h}))}}}_{T_2}.\notag 
\end{align}

Similarly,
\begin{align}
    &|h_2 - g_2| 
    \leq  \underbrace{\absmod{\int\frac{\lambda_2 d\{H(\boldsymbol{\lambda})-G(\boldsymbol{\lambda})\}}{-z +\boldsymbol{\lambda}^T\boldsymbol{\rho}(c\textbf{h})}}}_{T_3} + \underbrace{\absmod{\int\frac{\lambda_2 \boldsymbol{\lambda}^T(\boldsymbol{\rho}(c\textbf{g})-\boldsymbol{\rho}(c\textbf{h})) dG(\boldsymbol{\lambda})}{(-z + 
    \boldsymbol{\lambda}^T\boldsymbol{\rho}(c\textbf{g}))(-z + 
    \boldsymbol{\lambda}^T\boldsymbol{\rho}(c\textbf{h}))}}}_{T_4}.
\end{align}

The integrand in $T_1$ is bounded by $1/\Re(\rho_1(c\textbf{h}))$ and that in $T_3$ is bounded by $1/\Re(\rho_2(c\textbf{h}))$. So by choosing G sufficiently close to $H$ (i.e. the Levy distance $L(H, G)$ is close to $0$), we can make $T_1$ and $T_3$ arbitrarily small. Now let's look at $T_2$. We have $cg_1,cg_2,ch_1,ch_2 \in \mathcal{S}(cC_0/u)$ and, due to Remark \ref{remark_lipschitz_constant_1}, $\rho_1, \rho_2$ are Lipschitz continuous with constant $K_0=1$. Using H$\Ddot{o}$lder's Inequality, we have

\begin{align}
    T_2 &= \absmod{\int\frac{\lambda_1^2 (\rho_1(c\textbf{g})-\rho_1(c\textbf{h})) + \lambda_1\lambda_2(\rho_2(c\textbf{g})-\rho_2(c\textbf{h}))}{(-z + 
    \boldsymbol{\lambda}^T\boldsymbol{\rho}(c\textbf{g}))(-z + 
    \boldsymbol{\lambda}^T\boldsymbol{\rho}(c\textbf{h}))}dG(\boldsymbol{\lambda})}\\
    &\leq K_0||c\textbf{g}-c\textbf{h}||_1 \bigg(\sqrt{I_{2,0}(g,G)}\sqrt{I_{2,0}(h,G)} + \sqrt{I_{2,0}(g,G)I_{0,2}(h,G)}\bigg). \notag
\end{align}

Repeating arguments from (\ref{I_20_bound}), we have
\begin{align}\label{use_of_T5}
   \max\{I_{2,0}(g,G), I_{2,0}(h,G), I_{0,2}(h,G)\} \leq \frac{D_0}{u^2}.
\end{align}

Therefore, $T_2 \leq \dfrac{2cK_0D_0}{u^2}||\textbf{g}-\textbf{h}||_1$. Similarly, it can be shown that $T_4 \leq \dfrac{2cK_0D_0}{u^2}||\textbf{g}-\textbf{h}||_1$.

So to summarize,
\begin{align}
    ||\textbf{g} - \textbf{h}||_1 \leq T_1 + T_3 + T_2 + T_4 \leq T_1 + T_3 + \frac{4cD_0}{u^2}||\textbf{g} - \textbf{h}||_1 \text{, since } K_0=1.
\end{align}

By making $L(H, G)$ close to $0$, we can make $T_1 + T_3$ arbitrarily small. We have $4cD_0/u^2 < 1$ since $u > R_0$. So, this establishes the continuity of $\textbf{h}(z, H)$ as a function of $H$.
\end{proof}

\section{Proofs related to Section \ref{sec:equal_covariance}}
\subsection{Proof of Theorem \ref{Uniqueness_Special}}\label{sec:ProofOfUniquenessSpecial}
\begin{proof}
    Suppose for some $z = -u + \mathbbm{i}v \in \mathbb{C}_L$, $\exists$ $h_1, h_2 \in \mathbb{C}_R$ such that for $j \in \{1,2\}$, we have $$h_j = \displaystyle \int \dfrac{\lambda dH(\lambda)}{-z + \lambda \sigma(ch_j)}.$$
    Further let $\Re(h_j) = h_{j1}, \Im(h_j) = h_{j2}$ where $h_{j1} > 0$ by assumption for $j \in \{1,2\}$. Using (\ref{realOfSigma}), we have 
\begin{align}\label{realOf_h}
        & h_{j1} = \Re(h_j)
         = \displaystyle\int \dfrac{\lambda\Re(\overline{-z + \lambda\sigma(ch_j)})dH(\lambda)}{|-z + \lambda \sigma(ch_j)|^2} = \int\dfrac{u\lambda + \lambda^2 [\sigma_2(ch_j)\Re(ch_j)]}{|-z + \lambda\sigma(ch_j)|^2}dH(\lambda)\\ \notag
\implies & h_{j1} = uI_1(h_j, H) + ch_{j1}\sigma_2(ch_j)I_2(h_j, H)\\           
\text {where, } I_k(h_j, H) :&= \int \dfrac{\lambda^k dH(\lambda)}{|-z + \lambda \sigma(ch_j)|^2} \text{ for } k \in \{1,2\}. \notag
\end{align}

Note that $I_k(h_j, H) > 0, k \in \{1,2\}$ due to the conditions on H. Since $h_{j1} > 0$ and $u > 0$, using (\ref{realOf_h}), we must have 
\begin{align}\label{lessThanOne}
 c\sigma_2(ch_j)I_2(h_j, H) < 1.
\end{align}
Then we have
\begin{align*}
         h_1 - h_2
        &= \displaystyle\int \dfrac{(\sigma(ch_2) - \sigma(ch_1))\lambda^2}{[-z + \lambda\sigma(ch_1)][-z + \lambda\sigma(ch_2)]}dH(\lambda)\\
        &= (h_1-h_2)\displaystyle\int \dfrac{\dfrac{c\lambda^2}{(\mathbbm{i} +ch_1)(\mathbbm{i} +ch_2)} + \dfrac{c\lambda^2}{(-\mathbbm{i} +ch_1)(-\mathbbm{i} +ch_2)}}{[-z + \lambda\sigma(ch_1)][-z + \lambda\sigma(ch_2)]}dH(\lambda).
\end{align*}

By $\Ddot{H}$older's inequality, we have $|h_1-h_2| \leq |h_1-h_2|(T_1 + T_2)$ where 

\begin{align*}
T_1 &= \displaystyle\sqrt{\int\dfrac{c|\mathbbm{i} +ch_1|^{-2}\lambda^2dH(\lambda)}{|-z + \lambda\sigma(ch_1)|^2}}\sqrt{\int\dfrac{c|\mathbbm{i} +ch_2|^{-2}\lambda^2dH(\lambda)}{|-z + \lambda\sigma(ch_2)|^2}}\\    
&= \sqrt{c|\mathbbm{i} +ch_1|^{-2}I_2(h_1, H)}\sqrt{c|\mathbbm{i} +ch_2|^{-2}I_2(h_2, H)},
\end{align*}
and
\begin{align*}
T_2 &= \displaystyle\sqrt{\int\dfrac{c|-\mathbbm{i} +ch_1|^{-2}\lambda^2dH(\lambda)}{|-z + \lambda\sigma(ch_1)|^2}}\sqrt{\int\dfrac{c|-\mathbbm{i} +ch_2|^{-2}\lambda^2dH(\lambda)}{|-z +\lambda\sigma(ch_2)|^2}}\\
&= \sqrt{c|-\mathbbm{i} +ch_1|^{-2}I_2(h_1, H)}\sqrt{c|-\mathbbm{i} +ch_2|^{-2}I_2(h_2, H)}.
\end{align*}

Then, using the inequality $\sqrt{wx} + \sqrt{yz} \leq \sqrt{w+y}\sqrt{x+z}$ for $w,x,y,z, \geq 0$, we get
\begin{align*}
    &T_1 + T_2\\
    =& \sqrt{c|\mathbbm{i} +ch_1|^{-2}I_2(h_1, H)}\sqrt{c|\mathbbm{i} +ch_2|^{-2}I_2(h_2, H)} +  \sqrt{c|-\mathbbm{i} +ch_1|^{-2}I_2(h_1, H)}\sqrt{c|-\mathbbm{i} +ch_2|^{-2}I_2(h_2, H)}\\
    \leq& \sqrt{(c|\mathbbm{i} +ch_1|^{-2} + c|-\mathbbm{i} +ch_1|^{-2})I_2(h_1, H)}\sqrt{(c|\mathbbm{i} +ch_2|^{-2} + c|-\mathbbm{i} +ch_2|^{-2})I_2(h_2, H)}\\
    =& \sqrt{c\sigma_2(ch_1)I_2(h_1, H)}\sqrt{c\sigma_2(ch_2)I_2(h_2, H)}
    <  1 \text{, using } (\ref{lessThanOne}).
\end{align*}
This implies that $|h_1 - h_2|< |h_1 - h_2|$ which is a contradiction, thus proving the uniqueness of $h(z) \in \mathbb{C}_R$.
\end{proof}

\subsection{Proof of Theorem \ref{Continuity_Special}}\label{sec:ProofOfContinuitySpecial}
\begin{proof}
 For a fixed $c > 0$ and $z \in \mathbb{C}_L$, let $h, \ubar{h}$ be the unique numbers in $\mathbb{C}_R$ corresponding to distribution functions $H$ and $\ubar{H}$ respectively that satisfy (\ref{h_main_eqn1}). Following \cite{PaulSilverstein2009}, we have 
\begin{align*}
    h - \ubar{h}
    = & \int\dfrac{\lambda dH(\lambda)}{-z + \lambda\sigma(ch)} - \int\dfrac{\lambda d\ubar{H}(\lambda)}{-z + \lambda\sigma(c\ubar{h})}\\
    =& \underbrace{\int\dfrac{\lambda d\{H(\lambda) - \ubar{H}(\lambda)\}}{-z +\lambda\sigma(ch)}}_{:=T_1} + \int\dfrac{\lambda d\ubar{H}(\lambda)}{-z +\lambda\sigma(ch)} - \int\dfrac{\lambda d\ubar{H}(\lambda)}{-z + \lambda\sigma(c\ubar{h})}\\
    =& T_1 + \int\dfrac{\lambda^2(\sigma(c\ubar{h})- \sigma(ch))}{(-z + \lambda\sigma(ch))(-z + \lambda\sigma(c\ubar{h}))}d\ubar{H}(\lambda)\\
    =& T_1 + \int\dfrac{\dfrac{\lambda^2c(h-\ubar{h})}{(\mathbbm{i} +ch)(\mathbbm{i} +c\ubar{h})} + \dfrac{\lambda^2c(h-\ubar{h})}{(-\mathbbm{i} +ch)(-\mathbbm{i} +c\ubar{h})}}{(-z + \lambda\sigma(ch))(-z + \lambda\sigma(c\ubar{h}))}d\ubar{H}(\lambda)\\
    =& T_1 + (h-\ubar{h}) \underbrace{\int\dfrac{\dfrac{\lambda^2c}{(\mathbbm{i} +ch)(\mathbbm{i} +c\ubar{h})} + \dfrac{\lambda^2c}{(-\mathbbm{i} +ch)(-\mathbbm{i} +c\ubar{h})}}{(-z + \lambda\sigma(ch))(-z + \lambda\sigma(c\ubar{h}))}d\ubar{H}(\lambda)}_{:=\gamma}\\
    =& T_1 + (h-\ubar{h})\gamma.
\end{align*}
Note that, $\Re(\sigma(ch)) = \sigma_2(ch)\Re(ch) > 0$ and the integrand in $T_1$ is bounded by $1/\Re(\sigma(ch))$. So by making $\ubar{H}$ closer to $H$, $T_1$ can be made arbitrarily small. Now, if we can show that $|\gamma| < 1$, this will essentially prove the continuous dependence of the solution to (\ref{h_main_eqn1}) on H. 

\begin{align*}
    \gamma &= \underbrace{\int\dfrac{\dfrac{\lambda^2c}{(\mathbbm{i} +ch)(\mathbbm{i} +c\ubar{h})}}{(-z + \lambda\sigma(ch))(-z + \lambda\sigma(c\ubar{h}))}d\ubar{H}(\lambda)}_{:=G_1} + \underbrace{\int\dfrac{ \dfrac{\lambda^2c}{(-\mathbbm{i} +ch)(-\mathbbm{i} +c\ubar{h})}}{(-z + \lambda\sigma(ch))(-z \lambda\sigma(c\ubar{h}))}d\ubar{H}(\lambda)}_{:=G_2}\\
    &= G_1 + G_2.
\end{align*}                                     

By $\Ddot{H}$older's Inequality we have,
\begin{align*}
    |G_1| &\leq \displaystyle \sqrt{\underbrace{\int\dfrac{c\lambda^2|\mathbbm{i} +ch|^{-2}d\ubar{H}(\lambda)}{|-z + \lambda\sigma(ch)|^2}}_{:= P_1}}
    \sqrt{\underbrace{\int\dfrac{c\lambda^2|\mathbbm{i} +c\ubar{h}|^{-2}d\ubar{H}(\lambda)}{|-z + \lambda\sigma(c\ubar{h})|^2}}_{:=P_2}} = \sqrt{P_1 \times P_2}.
\end{align*}
From the definitions used in (\ref{realOf_h}), we have $|P_2| = c|\mathbbm{i} +c\ubar{h}|^{-2}I_2(\ubar{h}, \ubar{H})$
and 
\begin{align*}
    |P_1| &= c|\mathbbm{i} +ch|^{-2}\int\dfrac{\lambda^2 d\ubar{H}(\lambda)}{|-z + \lambda\sigma(ch)|^2}\\
    &= c|\mathbbm{i} +ch|^{-2}\bigg(\underbrace{\int\dfrac{\lambda^2 d\{\ubar{H}(\lambda)-H(\lambda)\}}{|-z + \lambda\sigma(ch)|^2}}_{:=K_1} + \int\dfrac{\lambda^2 dH(\lambda)}{|-z + \lambda\sigma(ch)|^2}\bigg)\\
    &= c|\mathbbm{i} +ch|^{-2}K_1 + c|\mathbbm{i} +ch|^{-2}I_2(h, H)\\
    & < \epsilon + c|\mathbbm{i} +ch|^{-2}I_2(h, H).
\end{align*}

for some arbitrarily small $\epsilon > 0$. The last inequality follows since the integrand in $K_1$ is bounded by ${|\Re(\sigma(ch))|^{-2}}$, we can arbitrarily control the first term by taking $\ubar{H}$ sufficiently close to $H$ in the Levy metric. The argument for bounding $|G_2|$ is exactly the same. 

Therefore, we have $$|G_1| < \sqrt{\epsilon + c|\mathbbm{i} +ch|^{-2}I_2(h, H)}\sqrt{c|\mathbbm{i} +c\ubar{h}|^{-2}I_2(\ubar{h}, \ubar{H})}.$$
Similarly, we also get 
$$|G_2| < \sqrt{\epsilon + c|-\mathbbm{i} +ch|^{-2}I_2(h, H)}\sqrt{c|-\mathbbm{i} +c\ubar{h}|^{-2}I_2(\ubar{h}, \ubar{H})}.$$

Using the inequality $\sqrt{wx} + \sqrt{yz} \leq \sqrt{w+y}\sqrt{x+z}$ for $w,x,y,z, \geq 0$, we have \begin{align*}
    &|G_1| + |G_2|\\
    <& \sqrt{\epsilon + c|\mathbbm{i} +ch|^{-2}I_2(h, H)}\sqrt{c|\mathbbm{i} +c\ubar{h}|^{-2}I_2(\ubar{h}, \ubar{H})} +\\
    &\sqrt{\epsilon + c|-\mathbbm{i} +ch|^{-2}I_2(h, H)}\sqrt{c|-\mathbbm{i} +c\ubar{h}|^{-2}I_2(\ubar{h}, \ubar{H})}\\
    \leq& \sqrt{2\epsilon + (c|\mathbbm{i} +ch|^{-2} + c|-\mathbbm{i} +ch|^{-2})I_2(h, H)}\sqrt{(c|\mathbbm{i} +c\ubar{h}|^{-2} + c|-\mathbbm{i} +c\ubar{h}|^{-2})I_2(\ubar{h}, \ubar{H})}\\
    =& \sqrt{2\epsilon + c\sigma_2(ch)I_2(h,H)}\sqrt{c\sigma_2(c\ubar{h})I_2(\ubar{h}, \ubar{H})}.
\end{align*}

From (\ref{lessThanOne}), we have $c\sigma_2(ch)I_2(h, H) < 1$ and $c\sigma_2(c\ubar{h})I_2(\ubar{h}, \ubar{H}) < 1$. By choosing $\epsilon > 0$ arbitrarily small, we finally have $|\gamma| = |G_1 + G_2| \leq |G_1| + |G_2| < 1$ for $\ubar{H}$ sufficiently close to H. This completes the proof.
\end{proof}

\section{Proofs related to Section \ref{sec:identity_covariance}}
\subsection{Results related to the density of the LSD in Section \ref{sec:identity_covariance}}
\begin{lemma}\label{SilvChoiResult2}
Let $s_{F}$ be as derived in Section \ref{sec:identity_covariance}. If a certain sequence $\{z_n\}_{n=1}^\infty \subset \mathbb{C}_L$ with $z_n \rightarrow \mathbbm{i}x$ satisfies $\underset{n \rightarrow \infty}{\lim} s_{F}(z_n) = s_0 \in \mathbb{C}_R$, then 
$s_{F}^0(x) := \underset{\mathbb{C}_L \ni z \rightarrow \mathbbm{i}x}{\lim} s_{F}(z)$ is well-defined, and equals $s_0$.
\end{lemma}

\begin{proof}

Consider the tuple $(z, s_{F}(z))$ for $z \in \mathbb{C}_L$. Define the function as follows:
    $$z_{F}:s_{F}(\mathbb{C}_L) \rightarrow \mathbb{C}_L \hspace{5mm}  z_{F}(s) := \frac{1}{s}\bigg(\frac{2}{c}-1\bigg) + \frac{1}{\mathbbm{i}cs}\bigg(\frac{1}{\mathbbm{i}+cs} - \frac{1}{-\mathbbm{i}+cs}\bigg).$$
We can extend the domain of $z_{F}$ to the set $\mathbb{C}\backslash\{0, \pm {\mathbbm{i}}/{c}\}$ where it is analytic. Note that on $s_{F}(\mathbb{C}_L)$, $z_{F}$ coincides with the inverse mapping of $s_{F}$. Clearly $z_{F}$ is continuous at $s_0$ as $s_0 \in \mathbb{C}_R$ and hence, $s_0\not \in \{0, \pm \mathbbm{i}/c\}$. Therefore, $z_{F}(s_0) = z_{F}(\underset{n \rightarrow \infty}{\lim}s_{F}(z_n)) = \underset{n \rightarrow \infty}{\lim}z_{F}(s_{F}(z_n)) = \underset{n \rightarrow \infty}{\lim}z_n = \mathbbm{i}x$.

Let $\{z_{1n}\}_{n=1}^\infty \subset \mathbb{C}_L$ be any another sequence such that $z_{1n} \rightarrow \mathbbm{i}x$. Since $s_0 \in \mathbb{C}_R$, we can choose an arbitrarily small $\epsilon$ such that $0 < \epsilon < \Re(s_0)$ and define $B := B(s_0; \epsilon)$\footnote{$B(x;r)$ indicates the open ball of radius $r$ centered at $x \in \mathbb{C}$}. $z_{F}$ being analytic and non-constant,  $z_{F}(B)$ is open by the Open Mapping Theorem and $\mathbbm{i}x \in z_{F}(B)$. So, for large $n$, $z_{1n} \in z_{F}(B)$. For these $z_{1n}$, there exists $s_{1n} \in B$ such that $z_{F}(s_{1n}) = z_{1n}$. By Theorem \ref{Uniqueness_Special}, we must have $s_{F}(z_{1n}) = s_{1n} \in B$. Since $\epsilon > 0$ is arbitrary, the result follows.
\end{proof}

\begin{lemma}\label{DistributionParameters}
    For the quantities defined in (\ref{supportParameters}) and $\Tilde{r},\Tilde{q},\Tilde{d}$ defined in (\ref{RQD}), the following results hold:
    \begin{description}
        \item[1] $L_c < U_c$ \vspace{-2mm},
        \item[2] $d(x) < 0$ on $S_c$ and $d(x) \geq 0$ on $S_c^c\backslash\{0\}$,
        \item[3] For $x\neq 0$, $r(x) = \mathbbm{i}\operatorname{sgn}(x)\bigg(-\dfrac{r_1}{|x|}  + \dfrac{r_3}{|x|^3}\bigg)$ and $q(x) = q_0 - \dfrac{q_2}{x^2}$, and
        \item[4] For $x\neq 0$, $d(x) = r^2(x) + q^3(x)$.
    \end{description}
\end{lemma}
\begin{proof}
Consider the polynomial $g(x) = d_0x^4 - d_2x^2 + d_4$. Reparametrizing $y = x^2$, the two roots in $y$ are given by $R_{\pm}$ ((1) of \ref{supportParameters}). We start with the fact for any $c \in (0, \infty)$, the discriminant term is positive since 
\begin{align}\label{discr_positive}
    d_2^2 - 4d_0d_4 = \bigg(\frac{4c+1}{9c^4}\bigg)^3 > 0.   
\end{align}
    
Now note that for all $c \in (0, \infty)$, $R_{+}$ is positive for all values of c. In fact, we have $$R_{+} = \frac{d_2 +\sqrt{d_2^2 - 4 d_0 d_4}}{2d_0} = \frac{1}{2}\bigg((2c^2+10c-1) + (4c+1)^{\frac{3}{2}}\bigg) > 0.$$ 
    
However, $R_{-}$ is positive depending on the value of c. Note that
    \begin{align*}
            &R_{-} =\dfrac{d_2 - \sqrt{d_2^2 - 4 d_0 d_4}}{2d_0} > 0 \\
             \iff& d_2 > \sqrt{d_2^2 - 4d_0d_4} > 0 \text{, since } d_0 = 1/27c^2 > 0 \\
             \iff& 4d_0d_4 > 0 \iff d_4 > 0 \iff 1 - 2/c > 0 \iff c > 2.
    \end{align*}
For $0 < c \leq 2$, $R_{-} \leq 0 < R_{+} \implies L_c < U_c$. For $c > 2$, we have $d_2 > 0$ and using (\ref{discr_positive}), we get
\begin{align*}
\sqrt{d_2^2 - 4d_0d_4} < d_2 \implies \frac{d_2 - \sqrt{d_2^2 - 4d_0d_4}}{2d_0} < \frac{d_2 + \sqrt{d_2^2 - 4d_0d_4}}{2d_0}
\implies & R_- < R_+ \implies L_c < U_c.
\end{align*}

Therefore for all $c > 0$, $(L_c, U_c)$ is a valid interval in $\mathbb{R}$. This proves the first result.

Since $d_0 = 1/(27c^6) > 0$ for all $c > 0$, the polynomial $g(x)$ is a parabola (in $x^2$) with a convex shape. When $c > 2$, we have $0 < R_- < R_+$. In this case, $g(x) = 0$ when $x^2 = R_\pm$ and $g(x) < 0$ when $x^2 \in (R_{-}, R_{+})$. Thus for all $x \in (-\sqrt{R_{+}}, -\sqrt{R_{-}}) \cup (\sqrt{R_{-}}, \sqrt{R_{+}})=S_c$, we have $g(x) < 0$. Similarly, for $0 < c \leq 2$, $g(x) < 0$ for all $x \in (-\sqrt{R_{+}}, 0)\cup(0, \sqrt{R_{+}}) = S_c$. Therefore, for any $c > 0$, we have $g(x) < 0$ on the set $S_c$. By the convexity of $g(\cdot)$ in $x^2$, $g(x) \geq 0$ on $S_c^c \backslash\{0\}$ is immediate. This establishes the second result.

Let $x \neq 0$ and $\epsilon > 0$. Consider $z = -\epsilon + \mathbbm{i}x$. 
Using the definition of $R(z)$, $Q(z)$ from (\ref{RQD}), we have
 \begin{align}\label{value_of_rx}
    &r(x) = \underset{\epsilon \downarrow 0}{\lim}\hspace{1mm}R(-\epsilon + \mathbbm{i}x) = \underset{\epsilon \downarrow 0}{\lim}\hspace{1mm} \dfrac{r_1}{-\epsilon + \mathbbm{i}x} + \dfrac{r_3}{(-\epsilon + \mathbbm{i}x)^3}
    = \frac{r_1}{\mathbbm{i}x} + \frac{r_3}{(\mathbbm{i}x)^3}
    = \mathbbm{i}\operatorname{sgn}(x)\bigg(-\frac{r_1}{|x|} + \frac{r_3}{|x|^3}\bigg), \text{ and} \\
    & q(x) = \underset{\epsilon \downarrow 0}{\lim}\hspace{1mm}Q(-\epsilon + \mathbbm{i}x) = \underset{\epsilon \downarrow 0}{\lim}\hspace{1mm} \bigg(q_0 + \frac{q_2}{(-\epsilon + \mathbbm{i}x)^2}\bigg) = q_0 - \frac{q_2}{x^2}.
\end{align}

This proves the third result. For the final result, note that $q_0^3 = d_0$, $d_2 = 3q_0^2q_2 + r_1^2$, $d_4 = 3q_0q_2^2 + 2r_1r_3$ and $q_2^3 + r_3^2 = 0$. Therefore for $x \neq 0$, we have
\begin{align*}
    r^2(x)+q^3(x)
    &= -\bigg(-\frac{r_1}{|x|} + \frac{r_3}{|x|^3}\bigg)^2 + \bigg(q_0 - \frac{q_2}{x^2}\bigg)^3\\
    &= q_0^3 + \frac{-3q_0^2q_2 - r_1^2}{x^2} + \frac{3q_0q_2^2 + 2r_1r_3}{x^4} + \frac{q_2^3 +r_3^2}{x^6}\\
    &= d_0 - \frac{d_2}{x^2} + \frac{d_4}{x^4} = d(x).
\end{align*}
\end{proof}

We state the following result (Theorem 2.2 of \cite{SilvChoi}) without proof. This result will be used to establish the continuity of the density function.
\begin{lemma}\label{continuousOnClosure}
    Let X be an open and bounded subset of $\mathbb{R}^n$, let Y be an open and bounded subset of $\mathbb{R}^m$, and let $f:\overline{X} \rightarrow Y$ be a function continuous on X. If, for all $x_0 \in \partial X$, $\underset{x \in X \rightarrow x_0}{\lim}f(x) = f(x_0)$, then f is continuous on all of $\overline{X}$.
\end{lemma}

\subsection{Proof of Theorem \ref{DensityDerivation}}\label{ProofDensityDerivation}
\begin{proof}
To check for existence (and consequently derive the value), we employ the following strategy. We first show that $\underset{\epsilon \downarrow 0}{\lim}\hspace{1mm}\Re(s_{F}(-\epsilon + \mathbbm{i}x))$ exists. Then by Lemma \ref{SilvChoiResult2}, the conditions of Proposition \ref{SilvChoiResult1} are satisfied, implying existence of density at $x_0$. The value of the density is then extracted by using the formula in (\ref{inversion_density}).

Recall the definition of $r(x)$ and $q(x)$ from (\ref{supportParameters}). We will first show that for $x \in S_c$,
\begin{align}\label{abs_of_rx}
-\frac{r_1}{|x|} + \frac{r_3}{|x|^3} > 0.
\end{align}

For $0 < c \leq 2$, we have $0 = L_c < U_c$ and from (\ref{RQD}),
\begin{align*}
  \frac{r_3}{r_1} = \frac{(c-2)^3}{9(c+1)} < 0.
\end{align*}
Thus $x \in S_c \implies x^2 > 0 > \dfrac{r_3}{r_1}$. 

For $c>2$, we have 
\begin{align*}
  0 < \frac{r_3}{r_1} = \frac{(c-2)^3}{9(c+1)} < \frac{1}{2}((2c^2 + 10c - 1) - (4c+1)^\frac{3}{2}) = L_c^2. 
\end{align*}

Thus, $0 < \dfrac{r_3}{r_1} < L_c^2 < U_c^2$. Therefore $x \in S_c \implies x^2 > \dfrac{r_3}{r_1}$. In either case, since $r_1<0$ we have 
\begin{align}\label{tempResults2}
    x^2 > \frac{r_3}{r_1} \implies r_1x^2 < r_3 \implies -\frac{r_1}{|x|} +\frac{r_3}{|x|^3} > 0 \implies |r(x)| = \mathbbm{i}\operatorname{sgn}(x)r(x),
\end{align}
where the last equality follows from (\ref{value_of_rx}).

Having established this, we are now in a position to derive the value of the density. Without loss of generality, choose $x \in S_c$ such that $x > 0$.  We can do this since the limiting distribution is symmetric about 0 from Proposition \ref{symmetry}. Consider $z = -\epsilon + \mathbbm{i}x$. The roots of (\ref{5B}) are given in (\ref{RootFormula}) in terms of quantities $S_0(z), T_0(z)$ that satisfy (\ref{S0T0}). Using (\ref{tempResults2}) and Lemma \ref{DistributionParameters}, we get
\begin{align}\label{tempResults3}
& |r(x)|^2 > (\mathbbm{i}\operatorname{sgn}(x)r(x))^2 - q^3(x) = -(r^2(x) +q^3(x)) =-d(x) > 0\\ \notag
\implies& |r(x)| > \sqrt{-d(x)}.
\end{align}

Therefore $V_+(x) > V_-(x) > 0$. Now, let $s_0 := \mathbbm{i}(V_+)^\frac{1}{3}$ and $t_0 := -q(x)/s_0$ (note that $s_0 \neq 0$). Since $q(x) =q_0 - {q_2}/{x^2} > 0$ as $q_0 > 0, q_2 < 0$, both $s_0$ and $t_0$ are purely imaginary. Observe that,
\begin{align*}
    V_+(x)V_-(x) = |r(x)|^2 - (\sqrt{-d(x)})^2=-r^2(x) + d(x) = q^3(x).
\end{align*}
Therefore, we get 
\begin{align*}
    t_0^3 = -\frac{q^3(x)}{s_0^3} = \frac{V_+(x)V_-(x)}{\mathbbm{i}V_+(x)} = -\mathbbm{i}V_-(x).
\end{align*}
Finally we observe that $s_0, t_0$ satisfy the below relationship:
\begin{itemize}
    \item $s_0^3+t_0^3 = 2r(x) = \underset{\epsilon \downarrow 0}{\lim}\hspace{1mm}2R(-\epsilon + \mathbbm{i}x) = \underset{\epsilon \downarrow 0}{\lim}\hspace{1mm}\bigg(S_0^3(-\epsilon + \mathbbm{i}x) + T_0^3(-\epsilon + \mathbbm{i}x)\bigg)$ and
    \item $s_0t_0 = -q(x) = -\underset{\epsilon \downarrow 0}{\lim}\hspace{1mm}Q(-\epsilon + \mathbbm{i}x) = \underset{\epsilon \downarrow 0}{\lim}\hspace{1mm}\bigg(S_0(-\epsilon + \mathbbm{i}x)T_0(-\epsilon + \mathbbm{i}x)\bigg)$.
\end{itemize}

From the above, it turns out that
\begin{align*}
\bigg\{\underset{\epsilon \downarrow 0}{\lim}\hspace{1mm}S_0^3(-\epsilon + \mathbbm{i}x), \underset{\epsilon \downarrow 0}{\lim}\hspace{1mm}T_0^3(-\epsilon + \mathbbm{i}x)\bigg\} =\{s_0^3, t_0^3\}.
\end{align*}

This leaves us with the following three possibilities.
 \begin{align*}
     \bigg\{\underset{\epsilon \downarrow 0}{\lim}\hspace{1mm}S_0(-\epsilon + \mathbbm{i}x), \underset{\epsilon \downarrow 0}{\lim}\hspace{1mm} T_0(-\epsilon + \mathbbm{i}x)\bigg\} =\{s_0, t_0\} \text{ or } \{\omega_1 s_0, \omega_2 t_0\} \text{ or } \{\omega_2 s_0, \omega_1 t_0\}.
 \end{align*}
 
Fortunately, the nature of (\ref{RootFormula}) is such that all three choices lead to the same set of roots, denoted by $\{m_j(-\epsilon + \mathbbm{i}x)\}_{j=1}^3$.
Using (\ref{RootFormula}) and shrinking $\epsilon$ to 0, we find in the limit
\begin{equation*}
    \left\{ \begin{aligned} 
M_1(x) := \underset{\epsilon \downarrow 0}{\lim}\hspace{1mm}m_1(-\epsilon + \mathbbm{i}x) &= -\dfrac{1-2/c}{3(-\epsilon + \mathbbm{i}x)} + s_0 + t_0,\\
M_2(x) := \underset{\epsilon \downarrow 0}{\lim}\hspace{1mm}m_2(-\epsilon + \mathbbm{i}x) &= -\dfrac{1-2/c}{3(-\epsilon + \mathbbm{i}x)} + \omega_1s_0 + \omega_2t_0, \text{ and}\\
M_3(x) := \underset{\epsilon \downarrow 0}{\lim}\hspace{1mm}m_3(-\epsilon + \mathbbm{i}x) &= -\dfrac{1-2/c}{3(-\epsilon + \mathbbm{i}x)} + \omega_2s_0 + \omega_1t_0.
\end{aligned} \right.
\end{equation*}

We have $\underset{\epsilon \downarrow 0}{\lim}\hspace{1mm} \Re  \bigg(\dfrac{2/c-1}{3(-\epsilon + \mathbbm{i}x)}\bigg)= 0$ and $\Re(s_0) = 0 = \Re(t_0)$. Therefore, $\Re(M_1(x)) = 0$.
Focusing on the second root,
\begin{align*}
     \Re(M_2(x))
    =\Re(\omega_1s_0 + \omega_2t_0)
    &= \Re \bigg(-\frac{s_0+t_0}{2} + \mathbbm{i}\frac{\sqrt{3}}{2}(s_0 - t_0)\bigg)\\
    &= \frac{\sqrt{3}}{2}\Im(t_0 - s_0) = \frac{\sqrt{3}}{2}\bigg((V_-(x))^\frac{1}{3} -  (V_+(x))^\frac{1}{3}\bigg) < 0,
\end{align*}
and similarly, 
\begin{align*}
     \Re(M_3(x))
    =\Re(\omega_2s_0 + \omega_1t_0)
    &= \Re \bigg(-\frac{s_0+t_0}{2} - \mathbbm{i}\frac{\sqrt{3}}{2}(s_0 - t_0)\bigg)\\
    &= \frac{\sqrt{3}}{2}\Im(s_0 - t_0) = \frac{\sqrt{3}}{2}\bigg((V_+(x))^\frac{1}{3} - (V_-(x))^\frac{1}{3}\bigg) > 0.
\end{align*}

To summarize, we evaluated the roots of (\ref{5B}) at a sequence of complex numbers $-\epsilon + \mathbbm{i}x$ in the left half of the complex plane close to the point $\mathbbm{i}x$ on the imaginary axis. This leads to three sequences of roots $\{m_j(-\epsilon + \mathbbm{i}x)\}_{j=1}^3$, of which only one has real part converging to a positive number. Therefore, for $x \in S_c \cap \mathbb{R}_+$, $s_{F}(-\epsilon + \mathbbm{i}x) \rightarrow M_3(x)$ as $\epsilon \downarrow 0$ by Theorem \ref{Uniqueness_Special}. So, from (\ref{inversion_density}) and the symmetry about 0, the density at $x \in S_c$ is
\begin{align*}
    f_c(x) &= \dfrac{1}{\pi}\underset{\epsilon \downarrow 0}{\lim}\hspace{1mm} \Re  (s_{F}(-\epsilon + \mathbbm{i}x)) = \frac{\sqrt{3}}{2\pi}\bigg((V_+(x))^\frac{1}{3} -(V_-(x))^\frac{1}{3}\bigg).
\end{align*}

Now we evaluate the density when $x \in S_c^c\backslash\{0\}$. Without loss of generality, let $x > 0$ since the distribution is symmetric about 0. From Lemma \ref{DistributionParameters}, $d(x) \geq 0$ in this case. Noting that $r(x) = -\mathbbm{i}|r(x)|$ from (\ref{tempResults2}), define $s_0 := (\sqrt{d(x)} -\mathbbm{i}|r(x)|)^\frac{1}{3}$ be any cube root and $t_0 := -{q(x)}/{s_0}$. Note that $s_0 \neq 0$ since $d(x) \geq 0$ and $|r(x)| > 0$.
Then, \begin{align*}
    t_0^3 = -\frac{q^3(x)}{s_0^3}=-\frac{d(x)-r^2(x)}{s_0^3} = -\frac{(\sqrt{d(x)} -\mathbbm{i}|r(x)|)(\sqrt{d(x)} + \mathbbm{i}|r(x)|)}{\sqrt{d(x)} - \mathbbm{i}|r(x)|} = -\sqrt{d(x)} - \mathbbm{i}|r(x)|.
\end{align*}
Therefore, we have 
\begin{align*}
 s_0^3+t_0^3 = 2r(x) = \underset{\epsilon \downarrow 0}{\lim}\hspace{1mm}2R(-\epsilon + \mathbbm{i}x); \hspace{5mm} s_0t_0 = -q(x) = -\underset{\epsilon \downarrow 0}{\lim}\hspace{1mm}Q(-\epsilon + \mathbbm{i}x).
\end{align*}

Therefore, using (\ref{RootFormula}) to find the three roots of (\ref{5B}) and shrinking $\epsilon > 0$ to 0, we get in the limit 

\begin{equation*}
    \left\{ \begin{aligned} 
M_1(x) := \underset{\epsilon \downarrow 0}{\lim}\hspace{1mm}m_1(-\epsilon + \mathbbm{i}x) &= -\dfrac{1-2/c}{3x} + s_0 + t_0,\\
M_2(x) := \underset{\epsilon \downarrow 0}{\lim}\hspace{1mm}m_2(-\epsilon + \mathbbm{i}x) &= -\dfrac{1-2/c}{3x} + \omega_1s_0 + \omega_2t_0, \text{ and} \\
M_3(x) := \underset{\epsilon \downarrow 0}{\lim}\hspace{1mm}m_3(-\epsilon + \mathbbm{i}x) &= -\dfrac{1-2/c}{3x} + \omega_2s_0 + \omega_1t_0.
\end{aligned} \right.
\end{equation*}

Observe that $\bigg|\sqrt{d(x)}-\mathbbm{i}|r(x)|\bigg|^2 = d(x) + |r(x)|^2 = d(x) - r^2(x) = q^3(x)$. Therefore, we have
$$t_0 = -\dfrac{q(x)}{s_0} = -\dfrac{q(x)\overline{s_0}}{|s_0|^2}
    = -\dfrac{q(x)\overline{s_0}}{|(\sqrt{d(x)} -\mathbbm{i} |r(x)|)|^\frac{2}{3}}
    = -\dfrac{q(x)\overline{s_0}}{|q(x)^3|^\frac{1}{3}}
    = -\dfrac{q(x)\overline{s_0}}{q(x)} = -\overline{s_0},$$
using the fact that $q(x) > 0$ for $x \neq 0$. Therefore, $\Re(s_0)=-\Re(t_0)$ and $\Im(s_0) =\Im(t_0)$. In particular, $s_0 + t_0 = 2\mathbbm{i}\Im(s_0)$ and $s_0 - t_0 = 2\Re(s_0)$.
This leads to the following observations:
\begin{align*}
\Re  (M_1(x)) &= \Re \{s_0+t_0\} = 0,\\
\Re  (M_2(x)) &= \Re  \{-\frac{1}{2}(s_0+t_0) + \mathbbm{i}\frac{\sqrt{3}}{2}(s_0 - t_0)\} = 0, \text{ and}\\
\Re  (M_3(x)) &= \Re  \{-\frac{1}{2}(s_0+t_0) - \mathbbm{i}\frac{\sqrt{3}}{2}(s_0 - t_0)\} = 0.  
\end{align*}

So when $x \in S_c^c\backslash\{0\}$, all three roots (in particular, the one that agrees with the Stieltjes transform) of (\ref{5B}) at $z = -\epsilon + \mathbbm{i}x$ have real component shrinking to 0 as $\epsilon \downarrow 0$. Therefore, by (\ref{inversion_density}) and the symmetry about $0$, we have
\begin{equation*}
    \begin{split}
    f_c(x) &= -\dfrac{1}{\pi}\underset{\epsilon \downarrow 0}{\lim}\hspace{1mm} \Re  (s_{F}(-\epsilon + \mathbbm{i}x)) = 0.\\
    \end{split}
\end{equation*}
So, the density is positive on $S_c$ and zero on $S_c^c\backslash\{0\}$. 

Finally, we check for existence of density at $x=0$ for $0 < c < 2$. For this we evaluate $L:= \underset{\epsilon \downarrow 0}{\lim}\hspace{1mm} \Re(s_{F}(-\epsilon))$ as follows:
\begin{align*}
    &\frac{1}{s_{F}(-\epsilon)} = -(-\epsilon) + \frac{1}{\mathbbm{i} +cs_{F}(-\epsilon)} + \frac{1}{-\mathbbm{i} +  s_{F}(-\epsilon)}\\
    \implies &\frac{1}{L} = \frac{1}{\mathbbm{i} +cL} + \frac{1}{-\mathbbm{i} +cL}  = \frac{2cL}{1 + c^2L^2}\\
    \implies &2cL^2 = 1 + c^2L^2\\
    \implies & \underset{\epsilon \downarrow 0}{\lim}\hspace{1mm} s_{F}(-\epsilon)= \frac{1}{\sqrt{2c-c^2}},
\end{align*}
where we considered the positive root since $s_F$ is a Stieltjes Transform of a measure on the imaginary axis. Therefore, by (\ref{inversion_density}), when $0 < c < 2$, $$f_c(0) = \dfrac{1}{\pi\sqrt{2c-c^2}}.$$

Now we show the continuity of $f_c$. Consider the case $0 < c < 2$. We saw that $f_c(x) = 0$ for $x \in S_c^c$. So, we need to show the continuity of $f_c$ in $S_c$. When $0<c<2$, $\underset{\epsilon \downarrow 0}{\lim}\hspace{1mm}\Re(s_{F}(-\epsilon + \mathbbm{i}x))$ exists for all $x \in \mathbb{R}$. In particular, when $x \in S_c$, $\underset{\epsilon \downarrow 0}{\lim}\hspace{1mm}\Re(s_{F}(-\epsilon + \mathbbm{i}x)) > 0$. For an arbitrary $x_0 \in S_c$, take an open bounded set $E \subset \mathbb{C}_L$ and choose $K > 0$ such that
$$\mathbbm{i}x_0 \in (-\mathbbm{i}K, \mathbbm{i}K) \subset \partial E.$$

Then the function defined below
$$s_{F}^0: \overline{E} \rightarrow \mathbb{R}; \hspace{5mm} s_{F}^0(z_0) = \underset{E \ni z \rightarrow z_0}{\lim}\Re(s_{F}(z)),$$ 
is well-defined due to Lemma \ref{SilvChoiResult2}. It
is continuous on $E$ due to the continuity of $\Re(s_{F})$ on $\mathbb{C}_L$ and satisfies the conditions of Lemma \ref{continuousOnClosure} by construction. Hence, the continuity of $s_{F}^0$ and of $f_c$ at $x_0$ is immediate.

Now consider the case when $c \geq 2$. As before, we only need to show the continuity of $f_c$ at an arbitrary $x_0 \in S_c$. Note that $x_0$ cannot be 0 as $0 \not \in S_c$. We already proved that $\underset{\epsilon \downarrow 0}{\lim}\hspace{1mm}\Re(s_{F}(-\epsilon + \mathbbm{i}x_0)) > 0$. Construct an open bounded set $E \subset \mathbb{C}_L$ such that 
$$\bigg(-\frac{3\mathbbm{i}|x_0|}{2}, -\frac{\mathbbm{i}|x_0|}{2}\bigg) \cup \bigg(\frac{\mathbbm{i}|x_0|}{2}, \frac{3\mathbbm{i}|x_0|}{2}\bigg) \subset \partial E.$$
A similar argument establishes the continuity of $f_c$ at $x_0 \neq 0$.
\end{proof}

\end{document}